\documentclass[french]{article} 
\usepackage[utf8]{inputenc}
\usepackage[T1]{fontenc}
\usepackage{lmodern}
\usepackage[a4paper]{geometry}
\usepackage[english]{babel}

\usepackage{fullpage}

\usepackage{amssymb,amsmath,mathtools,amsthm,dsfont,stmaryrd, bm}
\usepackage[all]{xy}
\usepackage{todonotes}

\newtheorem{theorem}{Theorem}
\newtheorem{prop}[theorem]{Proposition}
\newtheorem{corollary}[theorem]{Corollary}
\newtheorem{lemma}[theorem]{Lemma}

\newtheorem{remark}[theorem]{Remark}
\newtheorem{defi}[theorem]{Definition}

\newcommand{\E}{\mathbb{E}}

\newcommand{\N}{\mathbb{N}}

\renewcommand{\P}{\mathbb{P}}

\newcommand{\R}{\mathbb{R}}

\newcommand{\Z}{\mathbb{Z}}

\newcommand{\cV}{\mathcal{V}}

\renewcommand{\1}{\mathds{1}}

\renewcommand{\epsilon}{\varepsilon}
\renewcommand{\phi}{\varphi}

\newcounter{numeroexo}

\newcommand{\card}{\mbox{card}}

\begin{document}

\title{\LARGE First-order behavior of the time constant in non-isotropic continuous first-passage percolation}
\author{Anne-Laure Basdevant\footnote{LPSM - UMR CNRS 8001, Sorbonne Universit\'e, 4 place Jussieu, 75005 Paris, France, {\it anne.laure.basdevant@normalesup.org}}, Jean-Baptiste Gou\'er\'e\footnote{Institut Denis-Poisson - UMR CNRS 7013, Universit\'e de Tours, Parc de Grandmont, 37200 Tours, France, {\it jean-baptiste.gouere@lmpt.univ-tours.fr}} ~and Marie Th\'eret\footnote{Modal'X - UMR CNRS 9023, UPL, Universit\'e Paris Nanterre, 92000 Nanterre, France, {\it marie.theret@parisnanterre.fr}}}
\date{}

\selectlanguage{english}

\maketitle

\begin{abstract}
Consider $\Xi$ a homogeneous Poisson point process on $\mathbb{R}^d$ ($d\geq 2$) with unit intensity  with respect to the Lebesgue measure. For $\varepsilon\ge 0$,  we define the Boolean model $\Sigma_{p, \varepsilon}$ as the union of the balls of volume $\varepsilon$ for the $p$-norm ($p\in [1,\infty]$) and centered at the points of $\Xi$. We define a random pseudo-metric on $\mathbb{R}^d$ by associating with any path a travel time equal to its $p$-length outside $\Sigma_{p,\varepsilon}$. This defines a continuous model of first-passage percolation, that has been studied in \cite{GT17,GT22} for $p=2$, the Euclidean norm. For $p=1$, this model is expected to share common properties with the classical first-passage percolation on the graph $\mathbb{Z}^d$ with a distribution of passage times of the form $\varepsilon \delta_0 + (1-\varepsilon) \delta_1$. The exact calculation of the time constant of this model $\tilde \mu_{p,\varepsilon} (x)$ is out of reach. We investigate here the behavior of $\varepsilon \mapsto \tilde \mu_{p,\varepsilon} (x)$ near $0$, and enlight how the speed at which $\| x \|_p -  \tilde \mu_{p,\varepsilon} (x) $ goes to $0$ depends on $x$ and $p$. For instance, for $p\in (1,\infty)$, we prove that $\| x \|_p  - \tilde \mu_{p,\epsilon} (x)$ is of order $\varepsilon ^{\kappa_p(x)}$ with 
$$\kappa_p(x): = \frac{1}{d- \frac{d_1(x)-1}{2} - \frac{d-d_1 (x)}{p}}\,,$$
where $d_1(x)$ is the number of non null coordinates of $x$. The exact order of $\| x \|_p  - \tilde \mu_{p,\epsilon} (x)$ is also given for $p=1$ and $p=\infty$. Related results are also discussed, about properties of the geodesics, and analog properties on closely related models.
\end{abstract}

\section{Introduction}

\subsection{Main results}

\paragraph{Boolean model.} Consider a fixed dimension $d\geq 2$, and equip $\mathbb{R}^d$ with the $p$-norm $N = \| \cdot\|_p $ with some $p\in [1,+\infty]$, defined as usual: if $z=(z_1,\cdots , z_d) \in \mathbb{R}^d$, then
\begin{align*}
\textrm{for }p\in [1,\infty) \,, \qquad \|z\|_p = \left( \sum_{i=1}^d |z_i|^p \right)^{1/p} \,;
\textrm{ for }p=\infty\,, \qquad \|z\|_\infty = \max_{1 \leq i \leq d} |z_i |\,.
\end{align*}
We denote by $B_N (c,r)$ the closed ball for the norm $N$ centered at $c \in \mathbb{R}^d$ and of radius $r\in (0,\infty)$. 
Let $\Xi$ be a Poisson point process on $\mathbb{R}^d $ with intensity $\lambda |\cdot |$, where $\lambda \in (0,+\infty)$ and $|\cdot |$ designates the Lebesgue measure on $\mathbb{R}^d$. For any $r\geq 0$, we construct the Boolean model $\Sigma_{N,\lambda, r}$ by adding a ball of radius $r \in [0,+\infty)$ around each center $c\in \Xi$. More formally, the Boolean model $\Sigma_{N,\lambda, r}$ is defined by
\begin{equation}
\label{e:defSigma}
 \Sigma_{N,\lambda, r} = \bigcup_{c \in \Xi} B_N (c,r) \,.
 \end{equation}
Here we consider only the case where $r$ is a constant, but the model can be generalized by considering random radii. We refer to the book by Meester and Roy \cite{MeesterRoy} for background on the Boolean model in the Euclidean case, and to books by Last and Penrose \cite{last_penrose_2017} and by Schneider and Weil \cite{SchneiderWeil} for background on Poisson processes.

\paragraph{Continuous first-passage percolation.} A continuous model of first-passage percolation can be defined from the Boolean model by assuming that propagation occurs at speed $1$ outside $\Sigma_{N,\lambda, r} $, and at infinite speed inside $\Sigma_{N,\lambda, r} $ \emph{i.e.} the travel time $\tilde T_{N,\lambda ,r} (\pi)$ of a path $\pi$ is the $N$-length of $\pi$ outside $\Sigma_{N,\lambda , r}$. Optimizing on polygonal paths between two points $x,y \in \mathbb{R}^d$ leads to the definition of the travel time $ \tilde{T}_{N,\lambda, r} (x,y)$ between $x$ and $y$ as the minimal time needed to see propagation from $x$ to $y$:
\[
\tilde{T}_{N,\lambda, r} (x,y) = \inf_{\pi} \tilde{T}_{N,\lambda, r} (\pi)
\]
where the infimum runs over all polygonal paths from $x$ to $y$. More formal definitions are given in Section \ref{s:def}. In the case where $p=2$, {\it i.e.}, the underlying norm is the Euclidean norm, this model was studied by the second and third authors in \cite{GT17,GT22}, and previously introduced by Marchand and the second author in \cite{GM08} as a tool to study another growth model considered by Deijfen \cite{Deijfen}. By standard subadditive argument, it is well known that for every $z\in \mathbb{R}^d$, there exists a constant $ \tilde \mu_{N,\lambda, r} (z) \in [0,1]$ such that
\begin{equation}
\label{e:deftildemu}
\lim_{s \rightarrow \infty} \frac{ \tilde{T}_{N,\lambda, r} (0,sz)}{s} = \tilde \mu_{N,\lambda, r} (z) \quad \textrm{ a.s. and in }L^1.
\end{equation}
This constant $ \tilde \mu_{N,\lambda, r} (z)$ is called the {\it time constant} in the direction $z$.

\paragraph{Properties of $(\lambda , r) \mapsto   \tilde\mu_{N,\lambda,r} (z)$.} Fix $z\in \mathbb{R}^d\setminus \{0\}$ and $N=\| \cdot \|_2$ ({\it i.e.} $p=2$). By isotropy, we have $  \tilde  \mu_{\| \cdot \|_2,\lambda, r} (z) =  \| z\|_2  \tilde  \mu_{\| \cdot \|_2,\lambda, r} $ where $\tilde\mu_{\| \cdot \|_2,\lambda, r}$ is a constant (depending on $\lambda$ and $r$ but not $z$). We know that $  \tilde \mu_{\| \cdot \|_2,\lambda, 0} =1$, that $ \tilde \mu_{ \| \cdot \|_2,\lambda, r} >0$ if and only if the corresponding Boolean model $\Sigma_{\| \cdot \|_2,\lambda,r} $ is in some strongly subcritical percolation regime (see \cite{GT17} for more details), and that $(\lambda , r) \mapsto   \tilde \mu_{ \| \cdot \|_2,\lambda, r} $ is continuous (see \cite{GT22} for more details). However, there is no way to calculate explicitly the value of  $ \tilde\mu_{ \| \cdot \|_2,\lambda, r} $ when it is not null. For $N = \| \cdot \|_p$ for any $p\in [1,+\infty]$, the deterministic case $r =0$ is still trivial: $ \tilde\mu_{\| \cdot \|_p,\lambda,0} (z)= \| z\|_p$, and it is expected that positivity and continuity of the time constant exhibit the same behaviour as in the Euclidean case, but it is even more difficult to compute $ \tilde\mu_{\| \cdot \|_p,\lambda, r} (z)$. It thus makes sense to consider specific cases. In this work, we focus on the following case
 $$\lambda = 1 \qquad \mbox{ and } \qquad  r=\frac{\varepsilon^{1/d}}{2} \qquad \mbox{ for small $\varepsilon$.} $$
 This means that the number of balls per unit area is on average equal to 1, but each of them  has a small diameter equal to $\varepsilon^{1/d}$ and so a volume proportional to $\varepsilon$.
  We write for short 
  \begin{equation*}
   \tilde T_{p,\varepsilon} (x,y) :=  \tilde T_{\| \cdot \|_p,1,  \varepsilon^{1/d}/2} (x,y) \qquad \mbox{ and } \qquad
   \tilde \mu_{p,\varepsilon} (z) :=  \tilde \mu_{\| \cdot \|_p,1,  \varepsilon^{1/d}/2} (z).
\end{equation*}
Our primary objective is to calculate the speed at which $\| z\|_p - \tilde \mu_{p,\varepsilon} (z)$ goes to $0$ when $\varepsilon$ goes to $0$.

\paragraph{Main results.} Let $N=\| \cdot \|_p$ with $p\in [1,+\infty]$. Let $u = (u_1,\dots,u_d) \in \mathbb{R}^d$. We define
\begin{align}
\label{e:ajout4}
d_1 (u) &= \textrm{card} ( \{ i\in \{1,\dots , d\} \,:\, u_i \neq 0 \} )\nonumber \\
d_2 (u) &= d-d_1 (u) \nonumber\\
d_3 (u) &=  \textrm{card} ( \{ i\in \{1,\dots , d\} \,:\, |u_i| = \| u\|_\infty \} ) \\
d_4 (u) &= d-d_3 (u)\nonumber
\end{align}
where $\card (A)$ designates the cardinality of the finite set $A$.
For a given $p\in [1,\infty]$, for $u  \in \mathbb{R}^d$ such that $\|u\|_p = 1$, we define
\begin{equation}
\label{e:ajout5}
\kappa_p (u) = \left\{ \begin{array}{ll} \dfrac{1}{d_1 (u)} & \textrm{ if }  p =1\,, \\ \dfrac{1}{d - \frac{d_1 (u)-1}{2} - \frac{d_2(u)}{p}} & \textrm{ if }  p \in(1,\infty)\,, \\ \dfrac{1}{d_4 (u) +1} & \textrm{ if }  p =\infty \,. \end{array}\right.
\end{equation}
Notice that for any $p\in [1,\infty]$, and $u  \in \mathbb{R}^d$ such that $\|u\|_p = 1$, $d_1(u)\geq 1$ and $d - \frac{d_1 (u)-1}{2} - \frac{d_2(u)}{p} \geq d - (d_1(u)-1) - d_2(u) = 1$, thus $\kappa_p(u) \leq 1$. Our main result is the following.
\begin{theorem}
\label{theo:tilde_mu_p} For all $p\in [1,\infty]$, for all $u\in \R^d$ such that $ \| u \|_p=1$,
there exist constants $C_i = C_i (p,d,u)>0$, $i=1,2$ such that for $\varepsilon$ small enough, we have 
$$C_1\varepsilon^{\kappa_p(u)} \le 1-\tilde \mu_{p,\varepsilon}(u) \le C_2\varepsilon^{\kappa_p(u)}.$$
\end{theorem}
We call {\it geodesic} from $x$ to $y$ any polygonal path $\tilde \gamma_{p,\varepsilon} (x,y)$ from $x$ to $y$ such that $ \tilde T_{p,\varepsilon}  (x,y) = \tilde T_{p,\varepsilon} ( \tilde \gamma_{p,\varepsilon} (x,y))$ (see Proposition \ref{prop:geodes} for the existence of a geodesic). In the course of the proof of Theorem \ref{theo:tilde_mu_p}, we also obtain an upper bound on the minimal length of a geodesic. More precisely, we prove that for all $u \in \mathbb{R}^d$ such that $\| u\|_p=1$, for $\varepsilon$ small enough, a.s. for large $s$, there exists a geodesic $\tilde \gamma_{p,\varepsilon} (x,y)$ from $x$ to $y$ such that
\begin{equation}
\label{e:longueur}
  \| \tilde \gamma_{p,\varepsilon}(su) \|_p -s \le C_3\varepsilon^{\kappa_p(u)} s
  \end{equation}
(see Theorem \ref{theo:FPP} $(ii)$).

Moreover, at least in some specific directions or for some specific values of $p$, we can prove the existence of the rescaled limit of $1-\tilde \mu_{N,\varepsilon}(u)$ when $\varepsilon$ goes to $0$.
\begin{theorem}
\label{theo:monotonie}
Suppose we are in one of the following cases:
\begin{itemize}
\item $p\in [1,\infty]$ and $u = (1,0,\dots ,0)$ ;
\item $p\in \{1,2\}$, whatever $u\in \mathbb{R}^d$ such that $\|u\|_p =1$ ;
\item $p=\infty$ and $u\in \mathbb{R}^d$ such that $\|u\|_\infty=1$ and $d_3 (u) \in \{1,2\}$.
\end{itemize}
Then there exists a constant $K_p (u) \in (0,\infty)$ such that
$$   \lim_{\varepsilon \rightarrow 0} \frac{1- \tilde\mu_{p , \varepsilon}(u) }{\varepsilon^{\kappa_p (u)}} = K_p (u) \,.$$
\end{theorem}
Thus, for $p=1$ and $p=2$, we obtain the existence of the limit of $(1-\tilde\mu_{p, \varepsilon} (u) )/ \varepsilon^{\kappa_p(u)}$ when $\varepsilon$ goes to zero for any direction $u$. For $p=2$, we have $\kappa_2 (u) = \frac{2}{d+1}$ whatever $u$, as prescribed by the isotropy of the model, and we obtain that
\begin{equation*}
\tilde\mu_{2, \varepsilon} (u)  =1 - C \varepsilon^{2/(d+1)} + o ( \varepsilon^{2/(d+1)})
\end{equation*}
for any $u\in \mathbb{R}^d$ such that $\| u \|_2 =1$, with $C = K_p (u)$ that does not depend on such a $u$. For $p=1$, our results are similar to the ones previously obtained on the graph $\mathbb{Z}^d$ for the corresponding time constant $ \bm{\mu}_{\varepsilon} (z)$, see Equation \eqref{e:casZd} below.

\begin{remark}
The image of the homogeneous Poisson point process $\Xi$ on $\mathbb{R}^d$ with intensity $1$ by the map $x \in \mathbb{R}^d \mapsto x \varepsilon^{-1/d}$ is a homogeneous Poisson point process  on $\mathbb{R}^d$ with intensity $\varepsilon$. Thus, the image of  $ \Sigma_{N, 1 , \varepsilon^{1/d}/2}$ by this map has the same distribution as $\Sigma_{N,\varepsilon,1/2}$. Using the homogeneity of a norm, we deduce that, for any $z\in \R^d$, 
$$\varepsilon^{1/d}\tilde{T}_{\| \cdot \|_p,\varepsilon,{1/2}}(0,z)\overset{\tiny{\mbox{law}}}{=}{\tilde{T}}_{\| \cdot \|_p,1, {\varepsilon^{1/d}/2}}(0,z\varepsilon^{1/d})$$ which in turn  implies that $$\tilde \mu_{p,\varepsilon} (z):=  \tilde \mu_{\| \cdot \|_p, 1 ,\varepsilon^{1/d}/2} (z) = \tilde \mu_{\| \cdot \|_p,\varepsilon,  1/2} (z) .$$ 
Hence, all our results also hold for this model.

\end{remark}


\subsection{Background and motivation}
\label{s:background_motivation}

\paragraph{First-passage percolation on $\mathbb{Z}^d$.} Classical first-passage percolation was introduced by Hammersley and Welsh \cite{Hammersley-Welsh} in 1965 as a toy model to understand propagation phenomenon on a graph. Consider the graph $\mathbb{Z}^d$  ($d\geq 2$) equipped with the set of edges between nearest neighbours, and associate with the edges a family $(\tau(e))$ of non-negative i.i.d. variables, with common distribution $F$. The variable $\tau (e)$ is called the passage time of $e$, and represents the time needed to cross the edge $e$. This interpretation leads naturally to the definition of a random pseudo metric $\tau$ on $\mathbb{Z}^d$: for any two points $x,y\in\mathbb{Z}^d$, $\tau (x,y)$ is the minimal amount of time needed for the propagation to occur from $x$ to $y$. We refer to the surveys \cite{AuffingerDamronHanson,Kesten-saint-flour} for an overview on classical results in this model. By Kingman's ergodic subadditive theorem, it is well known that under a moment assumption on $F$, for any $z \in \mathbb{R}^d \setminus \{0\}$, we have
\begin{equation}
\label{e:cte_temps_Zd}
  \lim_{n\rightarrow \infty } \frac{\tau (0, \lfloor nz \rfloor)}{n} = \bm{\mu}_{F} (z) \quad \textrm{ a.s. and in }L^1 \,,
  \end{equation}
where $\lfloor nz \rfloor$ designates the coordinate-wise integer part of $nz$. The term $ \bm{\mu}_{F} (z)$ appearing here is a constant, again called the {\it time constant} of the model in the direction $z$, that depends on $z$ but also on the dimension $d$ of the underlying graph $\mathbb{Z}^d$ and on the distribution $F$ of the passage times. The function $z \mapsto  \bm{\mu}_{F} (z)$ is either the null function or a norm on $\mathbb{R}^d$. As for the continuous counterpart of the time constant defined previously, only few is known about $ \bm{\mu}_{F} (z)$, and in particular its value is not computable. 
It is thus relevant to look at particular and simple choices of distributions $F$. One natural choice among others is to consider a Bernoulli distribution with parameter $1-\varepsilon$, {\it i.e.}, $\mathbb{P} [\tau(e) = 0] = \varepsilon = 1 - \mathbb{P} [ \tau (e) =1] $. Let us denote by $ \bm{\mu}_{\varepsilon} (z)$ the corresponding time constant. The authors investigate in \cite{BGT1} at what speed $ \bm{\mu}_{\varepsilon}  (z)$ goes to $\bm{\mu}_{0}  (z) = \|z\|_1$ when $\varepsilon$ goes to $0$, and they prove that
\begin{equation}
\label{e:casZd}
 \bm{\mu}_{\varepsilon} (z) = \| z\|_1 - C(z)\, \varepsilon^{1/d_1 (z)} + o ( \varepsilon^{1/d_1 (z)})
\end{equation}
where $d_1 (z)$ is the number of non null coordinates of $z$, and $C(z)$ is a constant whose dependence on $z$ is partially explicit.

\paragraph{Motivation.} The asymptotic of \eqref{e:casZd} is related to the geometry of the lattice where the natural reference distance is the $\|\cdot\|_1$ distance.
The primary motivation of the article is to understand the influence of distance on the asymptotic. For this purpose, it is natural to work in a continuous setting where we have all freedom on the choice of the norm. Our main result shows in particular that forgetting the lattice but keeping the $\|\cdot\|_1$ distance does not change the asymptotic at first order. In the case of the Euclidean norm ($p=2$), we prove that the behaviour of this asymptotic is different, as prescribed by the isotropy of the model. We want to enlighten the fact that the proofs we build here to handle the $p$-norm for any $p\in [1,+\infty]$ are very different from the ones used in \cite{BGT1} to prove \eqref{e:casZd} (we refer to Section \ref{s:sketch} for a detailed sketch of the proofs).

\subsection{Framework}

\paragraph{A related model.} Since the diameter of each ball in the Boolean model is $\varepsilon^{1/d}$ and since we investigate the first order when $\varepsilon$ tends to $0$,
it is natural to neglect the possible overlaps between the balls and to rule out paths traveling inside a ball without reaching its center.
We are thus led to a new model where we only consider polygonal paths going through points of $\Xi$ (with the possible exception of the starting and ending points) and where at each passage at a ball center, the path saves time $\epsilon^{1/d}$. We see this time saved as a reward and refer to
the model as the model with rewards. Let us rephrase the definition (a formal definition is given in Section  \ref{s:def}). In the new model, the travel time of a path $\pi$ is replaced by what we call abusively (because it can be negative) an alternative travel time $T_{p,\epsilon}(\pi)$ defined as the $p$-length of $\pi$ minus $\epsilon^{1/d}$ times the number of rewards that $\pi$ collects. Optimizing on relevant paths from $x$ to $y$ leads to the definition of the alternative travel time $T_{p,\epsilon}(x,y)$ :
\[
{T}_{p,\epsilon} (x,y) = \inf_{\pi} T_{p,\epsilon} (\pi)
\]
where the infimum runs over all polygonal paths from $x$ to $y$. For large $\epsilon>0$ one easily checks that $T_{p,\epsilon}(x,y)=-\infty$ for any points $x,y$. For small enough $\epsilon>0$ (which is the regime we are interested in) Theorem \ref{t:existence} ensures  
the existence of an alternative time constant: for every $z\in \mathbb{R}^d$, there exists a constant $\mu_{p,\varepsilon} (z)$ such that
\begin{equation*}
\lim_{s \rightarrow \infty} \frac{T_{p,\varepsilon} (0,sz)}{s} =  \mu_{p,\varepsilon} (z) \quad \textrm{ a.s. and in }L^1.
\end{equation*}
Obviously, $ \mu_{p,0} (z)=\tilde \mu_{p,0} (z) =\|z \|_p$.

The model with rewards is more amenable to analysis than the model with balls. We will prove first the analog of Theorems \ref{theo:tilde_mu_p} and \ref{theo:monotonie} for $\mu_{p,\varepsilon}$ instead of $\tilde \mu_{p,\varepsilon}$. Then, we will compare $\mu_{p,\varepsilon}$ to $\tilde \mu_{p,\varepsilon}$ when it is possible, or adapt the proofs from the model with rewards to the model with balls when needed. We refer to Theorem \ref{theo:mu_p} for a precise statement of the results we obtain for the model with rewards.

\paragraph{A general norm $N$.} The most natural norms to consider are probably the $\|\cdot\|_1$ norm (corresponding to the discrete framework) and the $\|\cdot\|_2$ norm (the isotropic setting and thus maybe the easiest one). However, if one wants to understand the influence of geometry, it is natural to investigate general norms. We will indeed work in such a general setting and isolate the relevant geometric properties of the unit ball which influence the asymptotic behaviour. We will obtain somehow abstract results (see Theorems \ref{theo:mu} and \ref{theo:FPP}) and then specialize them to the $\|\cdot\|_p$ norms in order to get the results announced in the previous sections. As we will see in Section \ref{s:Np}, handling the $p$-norms requires to treat separately the cases $p=1$ and $p+\infty$ that are somehow degenerated, it is thus valuable to work in the setting of a general norm $N$ and to specialize to the $p$-norms subsequently.


\subsection{Definitions}
\label{s:def}

We gather in this section all the definitions we need. Some of them are straightforward generalization to a general norm $N$ of definitions that have been stated informally in the previous sections for the $p$-norm.

\paragraph{The space $\mathbb{R}^d$.} Throughout the paper, $d \in \mathbb{N}$ designates the dimension of the space, and we always assume that $d\geq 2$. Let $\cdot$ designate the standard scalar product on $\mathbb{R}^d$. 
We write $N$ to designate a general norm on $\mathbb{R}^d$.
We denote by $B_N(c,r)$ the closed ball for the norm $N$ centered at $c \in \mathbb{R}^d$ and of radius $r\in (0,\infty)$, and write $B_N (r)= B_N(0,r)$ for short. We designate by $| \cdot |$ the Lebesgue measure - usually on $\mathbb{R}^d$ or on $\mathbb{R}^{d-1}$, we omit the precision when it is not confusing, but write $| \cdot |_k$ for the Lebesgue measure on $\mathbb{R}^k$ if needed. For a subset $A$ of $\mathbb{R}^d$, we denote by $A^c = \mathbb{R}^d \setminus A$ its complement. Throughout the paper, $u$ designates a vector of $\mathbb{R}^d$ such that $N(u)=1$. For such a vector $u$, that lies on the boundary of $B_N(1)$, we can consider $u+H$ a supporting hyperplane of $B_N(1)$ at $u$, {\it i.e.}, satisfying
$$ \forall v\in H \,,\qquad N(u+v) \geq N(u) = 1\,. $$
Such a supporting hyperplane always exists by convexity of $B_N(1)$ (but it may not be unique). For given $u$ and $H$, we designate by $u^\star$ the vector normal to $H$ and satisfying $u \cdot u^\star = 1$.

\paragraph{The Poisson point process.} Let $\Xi$ be a Poisson point process on $\mathbb{R}^d$ with intensity $|\cdot | = |\cdot |_d$. Let $\varepsilon >0$.
We consider two different models (see $T$ and $\tilde T$ below). The points of $\Xi$ are seen as the locations of rewards of value $\varepsilon ^{1/d}$, or as centers of balls of diameter $\varepsilon^{1/d}$ for the norm $N$. In this setting, we define the Boolean model $\Sigma_{N, \varepsilon}$ as
$$ \Sigma_{N, \varepsilon} = \bigcup_{c \in \Xi} B_N \left(c, \frac{\varepsilon^{1/d}}{2} \right) .$$

\paragraph{The paths.} In the course of this paper, we will consider two types of path, referred to as \emph{polygonal paths} or \emph{generalized paths}.

\begin{defi}\label{def:path}
A \textbf{polygonal path}  $\pi$ (or path, for short) is a finite sequence of distinct points $\pi = (x_0, x_1, \dots , x_n)$, except maybe that $x_0=x_n$, and such that $x_i \in \Xi$ for all $i\in \{1,\dots, n-1 \}$ (we emphasize the fact that we do not require that $x_0$ or $x_n$ belong to $\Xi)$. We denote $\Pi (x,y)$ the set of polygonal paths $\pi = (x_0 = x, x_1, \dots , x_n=y)$ from $x$ to $y$.

A \textbf{generalized path}  $\pi$ is a finite sequence of distinct points $\pi = (x_0, x_1, \dots , x_n)$, except maybe that $x_0=x_n$ (without consideration for the Poisson point process $\Xi$ at all). We denote $\hat \Pi (x,y)$ the set of generalized paths from $x$ to $y$.

\end{defi}

Sometimes, we will also need to consider the polygonal curve associated with $\pi=(x_0,\ldots,x_n)$ which we will denote by $[\pi]$. For a given segment $[a,b] \subset \mathbb{R}^d$, we write $N ([a,b]) = N(b-a)$ for the $N$-length of the segment $[a,b]$. By a slight abuse of notation, since $[a,b] \cap \Sigma_{N, \varepsilon}$ is a finite disjoint union of segments, we denote by $N ([a,b] \cap \Sigma_{N, \varepsilon})$ the finite sum of the $N$-lengths of the disjoint connected components of $[a,b] \cap \Sigma_{N, \varepsilon}$. We define $N([a,b] \cap \Sigma_{N, \varepsilon}^c) = N ([a,b]) - N ([a,b] \cap \Sigma_{N, \varepsilon})$. For any generalized path $\pi = (x_0, x_1, \dots , x_n)$, we denote by $N(\pi)$ the length of $\pi$ for the norm $N$, {\it i.e.}, 
$$ N(\pi) = \sum_{i=1}^n N( [x_{i-1} , x_{i}]) =  \sum_{i=1}^n N( x_{i} - x_{i-1} ) \,. $$
We denote by $N(\pi \cap \Sigma_{N, \varepsilon}^c)$ the $N$-length of $\pi$ outside $\Sigma_{N, \varepsilon}$, {\it i.e.},
$$ N(\pi  \cap \Sigma_{N, \varepsilon}^c ) = \sum_{i=1}^n N( [x_{i-1} , x_{i}] \cap  \Sigma_{N, \varepsilon}^c)  \,. $$
We denote by $\sharp \pi$ the cardinality of $\Xi\cap[\pi]$. Note that for any generalized path, this quantity is a.s. finite and for a polygonal path $\pi =  (x_0, x_1, \dots , x_n)$ such that $x_0,x_n\notin \Xi$, we have a.s.
\begin{equation*}
   \sharp \pi =  n-1 \,.
\end{equation*}

\paragraph{The model with balls.} We first consider the model where a ball of diameter $\varepsilon^{1/d}$ is located at each point of $\Xi$. The {\it travel time} $\tilde T_{N,\varepsilon} (\pi)$ of a generalized path $\pi =(x_0, x_1, \dots , x_n) $ is
$$ \tilde T_{N,\varepsilon} (\pi) = N(\pi  \cap \Sigma_{N, \varepsilon}^c ).$$
We define the travel time $\tilde T_{N,\varepsilon}(x,y)$ between $x$ and $y$ as
$$ \tilde T_{N,\varepsilon}  (x,y)=\inf_{\pi \in\hat \Pi (x,y)} \tilde T_{N,\varepsilon}(\pi)\,.$$
We want to emphasize that the optimization can be made on polygonal paths instead of generalized paths, {\it i.e.},
$$ \tilde T_{N,\varepsilon}  (x,y)=\inf_{\pi \in  \Pi (x,y)} \tilde T_{N,\varepsilon}(\pi) \,.$$
This equality is stated in Proposition \ref{prop:geodespoly}.
We usually denote by $\tilde \gamma_{N,\varepsilon} (x,y)$ any {\it geodesic} from $x$ to $y$ for the time $ \tilde T_{N,\varepsilon}$, {\it i.e.}, $ \tilde \gamma_{N,\varepsilon} (x,y)\in \hat \Pi (x,y)$ such that $ \tilde T_{N,\varepsilon}  (x,y) =  \tilde T_{N,\varepsilon}( \tilde \gamma_{N,\varepsilon} (x,y))$
(see Proposition \ref{prop:geodes} for the existence of a geodesic). For short, we write $\tilde \gamma_{N,\varepsilon} (z) := \tilde \gamma_{N,\varepsilon} (0,z)$. The time constant in this model is defined through standard subadditive arguments in the following way: for every $z\in \mathbb{R}^d$, there exists a constant $ \tilde \mu_{ N,\varepsilon} (z) \ge 0$ such that
\begin{equation}
\label{e:cte_temps_1}
\lim_{s \rightarrow \infty} \frac{ \tilde{T}_{N,\varepsilon} (0,sz)}{s} = \tilde \mu_{N,\varepsilon} (z) \quad \textrm{ a.s. and in }L^1.
\end{equation}
We recall that when $N$ is the $p$-norm, we write $\tilde T_{p,\varepsilon}$, $\tilde \gamma_{p,\varepsilon}$ and $\tilde \mu_{p,\varepsilon}$ instead of $\tilde T_{\| \cdot \|_p,\varepsilon}$, $\tilde \gamma_{\| \cdot \|_p,\varepsilon}$ and $\tilde \mu_{\| \cdot \|_p,\varepsilon}$.

\paragraph{The model with rewards.} We now consider the model where a reward of value $\varepsilon^{1/d}$ is located on each point of $\Xi$. The {\it alternative travel time} $T_{N,\varepsilon}(\pi)$ of a path $\pi =(x_0, x_1, \dots , x_n) $ is
$$ T_{N,\varepsilon}(\pi) = N(\pi) - \varepsilon^{1/d} \sharp \pi \,.$$
We emphasize that $T_{N,\varepsilon}(\pi)$ may be negative, the terminology {\it alternative travel time} is chosen by analogy with the model with balls but may be confusing. We define the alternative travel time $T_{N,\varepsilon}(x,y)$ between $x$ and $y$ by
\begin{equation}
\label{e:ajoutv3}
 T_{N,\varepsilon}  (x,y)=\inf_{\pi \in\hat\Pi (x,y)} T_{N,\varepsilon}(\pi) \,.
 \end{equation}
As previously remarked, for large $\epsilon>0$ the model degenerates and $T_{N,\epsilon}(x,y)=-\infty$ for any points $x,y$. However, for small enough $\epsilon>0$ (which is the relevant regime for our asymptotic study) Theorem \ref{t:existence} ensures that a.s., for any points $x,y\in \R^d$, there exists a geodesic (that is, a minimiser in \eqref{e:ajoutv3}) among polygonal paths $\Pi(x,y)$. We usually denote by $ \gamma_{N,\varepsilon} (x,y)$ any chosen geodesic between $x$ and $y$. For short, we write $ \gamma_{N,\varepsilon} (z) := \gamma_{N,\varepsilon} (0,z)$. Moreover, an alternative time constant $\mu_{N,\varepsilon} (z)$  can be defined in this setting. Note that Theorem \ref{t:existence} does not say anything about the uniqueness of the geodesics. For some norms $N$ (for instance $N = \| \cdot \|_1$), they can in fact be non unique. We recall that when $N$ is the $p$-norm, we write $T_{p,\varepsilon}$, $\gamma_{p,\varepsilon}$ and $\mu_{p,\varepsilon}$ instead of $ T_{\| \cdot \|_p,\varepsilon}$, $\gamma_{\| \cdot \|_p,\varepsilon}$ and $ \mu_{\| \cdot \|_p,\varepsilon}$.


\subsection{Ideas of the proofs}
\label{s:sketch}

We present here a detailed outline of the proof. We hope it will be useful to the reader before beginning the reading of the proofs, and also that it will serve as a guide throughout the reading of the article.

\paragraph{Theorem \ref{theo:tilde_mu_p} and its version for the model with rewards.} 
The proof of Theorem \ref{theo:tilde_mu_p} relies on the analogous result in the model with rewards and on its proof.
These results are stated in the first item of Theorem \ref{theo:mu} in the case of a general norm $N$.
For the reader's convenience, we restate them here: for small $\varepsilon>0$,
\begin{equation}\label{e:just_below_intro}
\bar{g}_u(C_1\varepsilon^{-1})\le 1- \mu_{N,\varepsilon}(u)\le g_u(C_2\varepsilon^{-1}).
\end{equation}
In the above inequalities $\bar g_u$ and $g_u$ are two functions related to the local geometry of the unit ball $B_N(1)$ in the neighborhood of $u$, a vector such that $N(u)=1$.
We will give the definitions below.
For the moment, the important point is that the lower and the upper bounds match at first order when $N$ is symmetric enough, 
which is the case when $N=\|\cdot\|_p$ for some $p \in [1,+\infty]$ where both are of order $\varepsilon^{\kappa_p(u)}$.

Let us be  slightly more precise about the links between the results in the two models.
The upper bound in Theorem \ref{theo:tilde_mu_p} is a direct consequence of the upper bound in \eqref{e:just_below_intro} (which is the hardest part)
once we establish that $1- \tilde \mu_{N,\varepsilon}(u) \le 1- \mu_{N,\varepsilon}(u)$.
The lower bound in Theorem \ref{theo:tilde_mu_p} is proven by adapting the proof of the lower bound in \eqref{e:just_below_intro}.

We focus hereafter on the model with rewards.

\paragraph{Upper bound in \eqref{e:just_below_intro}.}
Recall that $u \in \R^d$ is such that $N(u)=1$. 
We also need a hyperplane $H$ such that $u+H$ is a supporting hyperplane of $B_N(1)$ at $u$.
If $B_N(1)$ is differentiable at $u$, then $u+H$ is simply the tangent hyperplane to $B_N(1)$ at $u$.
Aiming at proving \eqref{e:just_below_intro}, we investigate $T_{N,\varepsilon}(0,su)$ for large $s$.
We are thus interested in polygonal paths from $0$ to $su$ with many points (and thus many $\varepsilon^{1/d}$ rewards) and not too large length.
The following family of sets then naturally comes into play.
For $\eta>0$, we define (see Figure \ref{Fig:boule_intro}, that will be duplicated later in the article)
\[
K_\eta(u):=\{v\in H \,:\, N(u+v)\le 1+\eta\} .
\]
\begin{figure}[h]
\begin{center}
\begin{tabular}{c c}
\includegraphics[width=7cm]{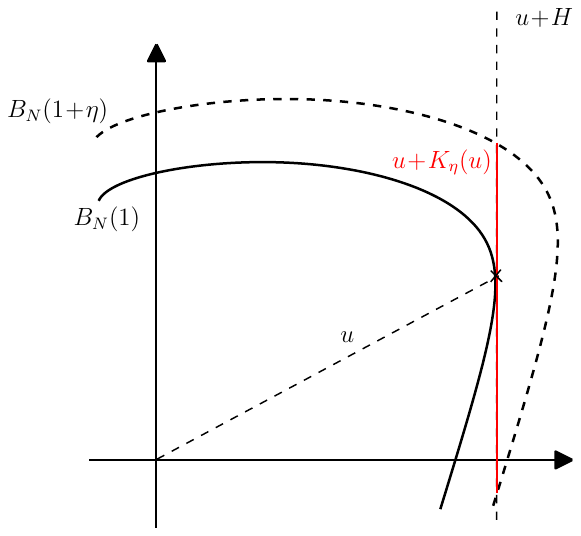} &
\includegraphics[width=7cm]{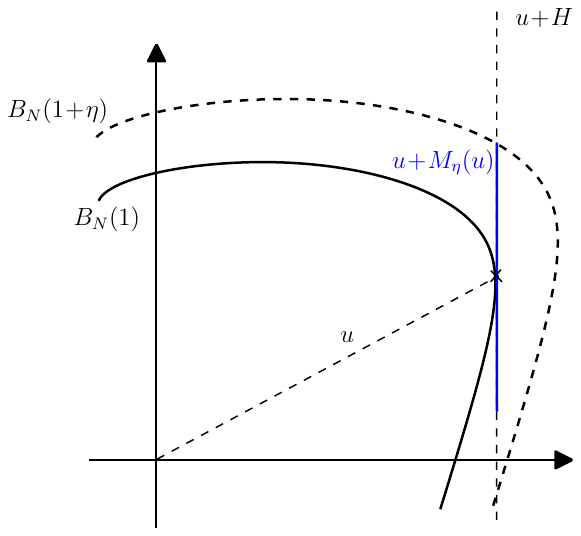}
\end{tabular}
\end{center}
\caption{Illustration of the sets $K_\eta(u)$ and $M_\eta(u)$ (note that $M_\eta(u)$ is here strictly included in $K_\eta(u)$).}
\label{Fig:boule_intro}
\end{figure}The relevance of $K_\eta(u)$ is as follows.
Consider a polygonal path $\pi=(x_0,\dots,x_n)$ from $0$ to $su$ for some $s \ge 0$ such that each of its increments $x_i-x_{i-1}$ belongs to the cone with direction $u$ and  base $K_\eta(u)$. By definition of $K_\eta(u)$, we have $N(\pi) \le s(1+\eta)$.
At a heuristic level, it may thus be natural to restrict our attention to such paths for some $\eta$ conveniently depending on $\varepsilon$.
Moreover, one can guess using a greedy algorithm (such a greedy algorithm for a different cone is used and thus defined below in the paragraph about lower bound) 
that the number of points of such a path is at most of order $|K_\eta(u)|^{1/d}s$
where $|K_\eta(u)|$ is the $(d-1)$-dimensional Lebesgue measure of $K_\eta(u)$.
Of course, the geodesics cannot be all of the previous kind. 
However, we prove the following closely related result which validates the above heuristic (see Proposition \ref{prop:Chernov}).
For some constant $c>0$,  for large $s \ge 0$ and for any path $\pi$ from $0$ to $su$, we have:
\begin{equation}\label{eq:upperbound_nombre_intro}
N(\pi) \leq (1+ \eta) s\; \Longrightarrow \; \sharp \pi \leq c|K_\eta(u)|^{1/d} s.
\end{equation}
Thus any upper bound on $N(\gamma_{N,\varepsilon}(su))$ provides in turn an upper bound on $\sharp \gamma_{N,\varepsilon}(su)$.
Note moreover that for any geodesic $\gamma_{N,\varepsilon}(su)$ from 0 to $su$, we have
\begin{equation}\label{eq:upperbound_N_intro}
N(\gamma_{N,\varepsilon}(su))  = T_{N,\varepsilon} (0,su) + \varepsilon^{1/d} \sharp \gamma_{N,\varepsilon}(su) \leq s +  \varepsilon^{1/d} \sharp \gamma_{N,\varepsilon}(su).  
\end{equation}
Thus any upper bound on  $\sharp \gamma_{N,\varepsilon}(su)$ also provides an upper bound on $N(\gamma_{N,\varepsilon}(su))$. 
Finally, we need the following crude inequality (see \eqref{e:ajoutcrude}) coming from greedy animals: for some constant $c$, almost surely, for any path $\pi$ from $0$ to a faraway point, 
\begin{equation}\label{eq:upperbound_greedy_intro}
\sharp \pi \le cN(\pi).
\end{equation} 
From this first upper bound, using repeatedly and alternately \eqref{eq:upperbound_nombre_intro} and \eqref{eq:upperbound_N_intro} we get, after a finite number of steps, 
good upper bounds on $\sharp \gamma_{N,\varepsilon}(su)$ and on $N(\gamma_{N,\varepsilon}(su))$. 
From the upper bound on $\sharp \gamma_{N,\varepsilon}(su)$  we deduce straightforwardly the desired lower bound on $T_{N,\varepsilon}(0,su)$ and thus on $\mu_{N,\varepsilon} (u)$.
The obtained bound involves the family of sets $K_\eta(u)$ through the function $g_u$, which is the reciprocal of the decreasing homeomorphism defined by $\eta \mapsto \eta^{-d}|K_\eta(u)|$.

\paragraph{Lower bound in \eqref{e:just_below_intro}.} The desired lower bound on $1-\mu_{N,\varepsilon} (u)$ comes from the construction of a good path by a greedy algorithm.
We need the following symmetric variant of $K_\eta(u)$ (see Figure \ref{Fig:boule_intro}):
\[
\qquad M_\eta(u):=K_\eta(u)\cap (-K_\eta(u)).
\]
\begin{figure}[h]
\begin{center}
\includegraphics[width=12cm]{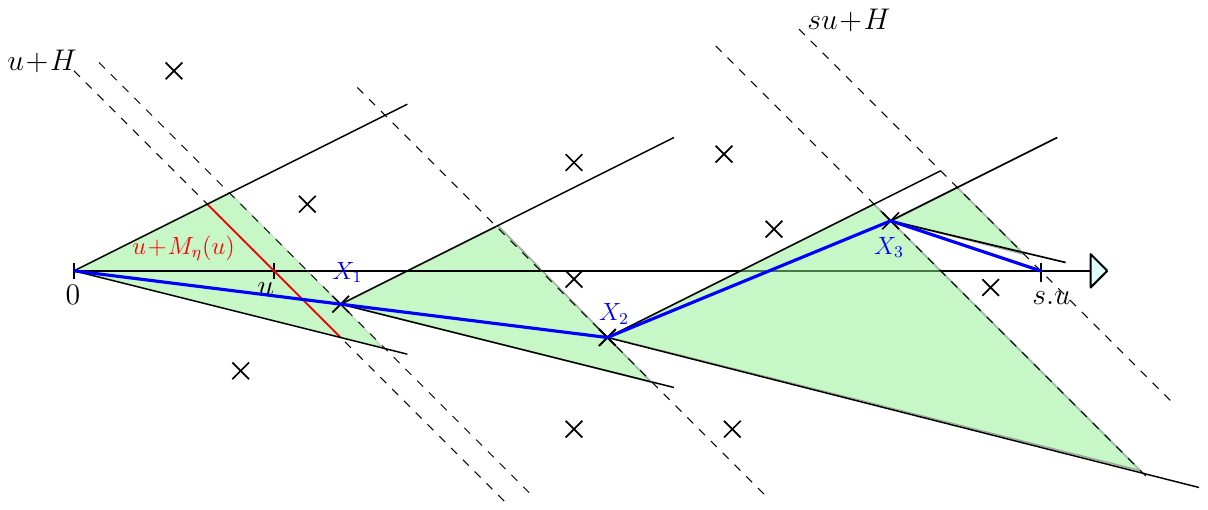}
\end{center}
\caption{Illustration of the greedy algorithm. The path constructed is drawn in blue. At each step, the path goes through the first point of $\Xi$ in the green cone in front of it.} 
\label{Fig:glouton_intro}
\end{figure}We now refer to Figure \ref{Fig:glouton_intro}.
Instead of the cone of direction $u$ and base $K_\eta(u)$ considered before, we now consider the cone of direction $u$ and base $M_\eta(u)$.
We define $X_1$ as the first (for the coordinate along $u$ in the direct sum $\R^d = \R u \oplus H$) point of $\Xi$ in the cone, $X_2$ for the first point in the translated cone with apex at $X_1$ and so on.
One key point (and this is the reason for the symmetrisation) is that the expected value of the i.i.d.\ increments of the path belongs to $\R_+ u$.
Therefore, at first order, this allows us to build a path from $0$ to a faraway point $su$ with a good control on the number of points and on the length of this path.
This is sufficient to yield the desired upper bound on $\mu_{N,\varepsilon} (u)$.
The obtained bound involves the family of sets $M_\eta(u)$ through the function $\bar g_u$, which is the reciprocal of the decreasing homeomorphism defined by 
$\eta \mapsto \eta^{-d}|M_\eta(u)|$.

\paragraph{On the proof of Theorem \ref{theo:monotonie}.} The existence of the limit stated in Theorem \ref{theo:monotonie} comes from the monotonicity of
\[
\epsilon \mapsto \frac{1 - \mu_{p, \varepsilon}(u)}{\varepsilon^{\kappa_p}(u)}\,,
\]
the bounds provided by \eqref{e:just_below_intro} in this setting (yielding that the above function is bounded) 
and to the close link between $\mu_{p, \varepsilon}(u)$ and $\tilde \mu_{p, \varepsilon}(u)$ (see Theorem \ref{theo:FPP} $(iii)$).
The monotocity stated above is derived, after a proper scaling, from monotonicity at a very basic level (before taking any limit or infimum, see Remark \ref{r:monotonie}).

\paragraph{On the link with \eqref{e:casZd}.} As explained in Section \ref{s:background_motivation}, in \cite{BGT1} the authors derived \eqref{e:casZd}, which
describes an asymptotic behavior for first passage percolation on $\Z^d$. This behavior aligns with the asymptotic behavior observed in our continuous setting when $N=\|\cdot\|_1$.
However, the overall strategy in \cite{BGT1} differs significantly.
The key result in \cite{BGT1} is to relate the discrete non oriented model to a semi-oriented and semi-continuous model on $\R^{d_1(u)} \times \Z^{d_2(u)}$ which can be understood through scaling.


\subsection{Organization of the paper}

The paper is organized as follows. 

Section \ref{s:mu} is devoted to prove some  basic results about the model with rewards: existence and classical properties of the time constant $\mu_{N,\epsilon}$ and of geodesics $\gamma_{N,\epsilon}(x,y)$ (see
Theorem \ref{t:existence}).

Section \ref{s:geo} focuses on some geometrical properties of objects associated with the norm $N$. 

Section \ref{s:mu2} is devoted to the study of the model with rewards. We follow the strategy sketched in Section \ref{s:sketch}. In Section \ref{s:glouton}, we prove a lower bound on $1-\mu_{N,\varepsilon} (u)$ using a greedy algorithm. In Section \ref{s:chernov}, we obtain an upper bound on $1-\mu_{N,\varepsilon} (u)$ using a bootstrap argument, that involves a control on $\sharp \gamma_{N,\varepsilon} (su)$ and $N (\gamma_{N,\varepsilon} (su))$.

Section \ref{s:tildemu} examines the model with balls centered at the points of $\Xi$. First we prove in Section \ref{s:geodes} that we can deal with geodesics that have nice properties. Then we compare $\tilde \mu_{N,\varepsilon} (u)$ with $\mu_{N,\varepsilon} (u)$ in Section \ref{s:comp}. Finally we adapt in Sections \ref{s:lowbis} and \ref{s:upbis} the proofs written in the study of the previous model to this setting to get the results that are not direct consequences of the comparison between $\tilde \mu_{N,\varepsilon} (u)$ and $\mu_{N,\varepsilon} (u)$.

Finally in Section \ref{s:Np} we specialize to the case where $N=\| \cdot \|_p$, the $p$-norm, for $p\in [1,\infty]$. 
In Sections \ref{s:H}, \ref{s:evalK} and \ref{s:evalM}, we compute estimates of the geometric quantities of interest. These estimates allow us to prove Theorem \ref{theo:tilde_mu_p}, and its analog for the model with rewards (Theorem \ref{theo:mu_p}) in Section \ref{s:appli}. By a monotonicity argument, we finally prove Theorem \ref{theo:monotonie} in Section \ref{s:monotonie}.

The Appendix gathers the proofs of Sections \ref{s:mu} and \ref{s:geo} that are less innovative or more technical.

\subsection{Notations}

From Section \ref{s:mu} to Section \ref{s:tildemu}, we work with a fixed general norm $N$. For that reason, we will omit the subscript $N$ in our notations, since no confusion is possible (for example we write $T_\varepsilon$, $\mu_\varepsilon$ instead of $T_{N,\varepsilon}$, $\mu_{N,\varepsilon}$). In Section \ref{s:Np}, we manipulate $p$-norms with different values of $p$, thus we reintroduce the subscript $p$ to emphasize the dependence on $p$.


\section{Existence and basic properties of $\mu_{\varepsilon}$}
\label{s:mu}

Fix a norm $N$ on $\R^d$ and consider the model where a reward of value $\varepsilon^{1/d}$ is located at each point of $\Xi$. Recall that the {\it alternative travel time} $T_\varepsilon(\pi)$ of a generalized path $\pi =(x_0,\ldots,x_n)$ is defined by
$$T_\varepsilon(\pi)=N(\pi)-\varepsilon^{1/d}\sharp \pi,$$
where $[\pi]:=\bigcup_{i=1}^n [x_{i-1},x_i]$  is the polygonal curve with vertices $(x_0,\ldots,x_n)$ and $\sharp \pi$ denotes the cardinality of the a.s. finite set $\Xi\cap 
[\pi]$.

Recall the definitions of generalized paths and of polygonal paths given in Definition \ref{def:path}. The first result states that minimizing the travel time over generalized paths does not get a better result than minimizing over polygonal paths.

\begin{lemma} For all $x,y\in \R^d$,
$$T_\varepsilon(x,y):=\inf \{ T_\varepsilon(\pi) \, | \, \pi \in\hat \Pi (x,y)\}=\inf \{ T_\varepsilon(\pi) \, | \, \pi \in \Pi (x,y)\}.$$
\end{lemma}

\begin{proof}
Let $\pi_1=(x=x_0,\ldots,x_n=y) \in \hat \Pi (x,y)$ be a generalized path. Let us denote by $a_1,\ldots,a_k$ the points of $\Xi$ which are in $[\pi_1]\setminus\{x_0,x_n\}$ ranked by their  order of apparition in the curve $[\pi_1]$. Then $\pi_2=(x_0,a_1,\ldots,a_k,x_n)$ belongs to $\Pi (x,y)$. We have by construction
 $\sharp \pi_1\le \sharp \pi_2$ and by triangle inequality, $N(\pi_2)\le N(\pi_1)$. This yields 
$T_\varepsilon(\pi_2)\le T_\varepsilon(\pi_1)$.

\end{proof}

As already noted, the {\it alternative travel time} $T_\varepsilon(\pi)$ of a polygonal path $\pi$ may be negative. Let us also note that the alternative travel time between $x$ and $y$ does not satisfy the triangle inequality $T_\varepsilon(x,y)\le T_\varepsilon(x,z)+T_\varepsilon(z,y)$.
The setting is not exactly the most classic one.
However, with some almost classical argument, we can prove the following theorem. Since  the proofs are not difficult nor very innovative, we give them for the sake of completeness in the Appendix.

\begin{theorem}
\label{t:existence}
For $\varepsilon$ small enough (depending on $N$ and $d$), 
\begin{itemize}
\item[(i)]There exists a.s. for all $x,y\in \R^d$ a finite polygonal path $\gamma_{\varepsilon}(x,y) \in \Pi (x,y)$ from $x$ to $y$ such that 
$$T_{\varepsilon}(x,y)=T_{\varepsilon}(\gamma_{\varepsilon}(x,y)) \,.$$
\item[(ii)] For every $z\in \mathbb{R}^d$, there exists a deterministic constant $\mu_{\varepsilon} (z)$ such that
\begin{equation}
\label{e:cte_temps_2}
\lim_{s \rightarrow \infty} \frac{T_{N,\varepsilon} (0,sz)}{s} =  \mu_{\varepsilon} (z) \quad \textrm{ a.s. and in }L^1.
\end{equation}
Moreover, the function $\mu_{\varepsilon} (\cdot)$ is a norm on $\R^d$.
\end{itemize}
\end{theorem}
For all $x,y\in \R^d$, any finite polygonal path $\gamma_{\varepsilon}(x,y) \in \Pi (x,y)$ from $x$ to $y$ such that $T_{\varepsilon}(x,y)=T_{\varepsilon}(\gamma_{\varepsilon}(x,y)) $ is called a {\em geodesic} from $x$ to $y$ for the time $T_{\varepsilon} (x,y)$.


\section{Some geometric results}
\label{s:geo}
As already mentioned in the introduction (Section \ref{s:sketch}),  the exponent $\kappa_p(u)$ given in Equation \eqref{e:ajout5} is expressed as a function of the rate of decay of two  geometric sets related to the unit ball of the $p$-norm, $K_\eta(u)$ and $M_\eta(u)$. In this section, we recall the definition of these geometric sets and give some  property of their volume. In particular, we show that the volume of one of them is also of the same order of magnitude as an integral, which will appear naturally later on. These results rely on proofs that are not essential for understanding the remainder of the proof and would disrupt the flow if included here. Therefore, they have been moved to the appendix.

\medskip 
Throughout the paper we will repeatedly consider an element $u\in \mathbb{R}^d$ and a hyperplane $H$ satisfying
\begin{equation}
\label{e:(u,H)}
N(u) = 1 \qquad \textrm{and} \qquad \forall v \in H, \; N(u+v)\ge N(u)=1 \,.
\end{equation}
In other words, $u$ lies on the boundary of $B_N(1)$, and $u+H$ is a supporting hyperplane of $B_N(1)$ at $u$. For any couple $(u,H)$ satisfying \eqref{e:(u,H)}, we denote by $u^\star$ the vector normal to $H$ such that $u\cdot u^\star=1$:
\begin{equation}
\label{e:u*}
\forall v \in H, \; u^\star \cdot v = 0 \qquad \textrm{and} \qquad u \cdot u^\star = 1 \,.
\end{equation}
For a given $\eta \geq 0$, let us consider the two following subsets of $H$ (see Figure \ref{Fig:boule}):
 $$K_\eta(u):=\{v\in H \,:\, N(u+v)\le 1+\eta\} \qquad \mbox{ and } \qquad M_\eta(u):=K_\eta(u)\cap (-K_\eta(u))\,.$$
We emphasize the fact that $K_\eta(u)$ and $M_\eta(u)$ depend on $u$ but also on $H$, even if this dependence on $H$ is not explicit in the notation.
The sets $K_\eta (u)$ and $M_\eta (u)$ are subsets of $H$, which is a hyperplane. We thus write $|K_\eta (u)| := |K_\eta (u)|_{d-1}$ (resp. $|M_\eta (u)| := |M_\eta (u)|_{d-1}$) to designate the $(d-1)$-dimensional Lebesgue measure of $K_\eta (u)$ (resp. $M_\eta (u)$). For $\eta >0$, we define
$$h_u(\eta):=\eta^{-d}|K_\eta(u)| \qquad \mbox{ and } \qquad \bar{h}_u(\eta)=\eta^{-d}|M_\eta(u)|.$$
\begin{prop}
\label{prop:function_h}
The functions $\eta\mapsto h_u(\eta)$ and $\eta\mapsto \bar{h}_u(\eta)$ are decreasing homeomorphism from $(0,\infty)$ to $(0,\infty)$. We denote by $g_u$ and $\bar{g}_u$ their inverse functions. Note that $g_u$ and $\bar{g}_u$ are also decreasing homeomorphism from $(0,\infty)$ to $(0,\infty)$.
\end{prop}

\setcounter{figure}{0}
\begin{figure}
\begin{center}
\begin{tabular}{c c}
\includegraphics[width=7cm]{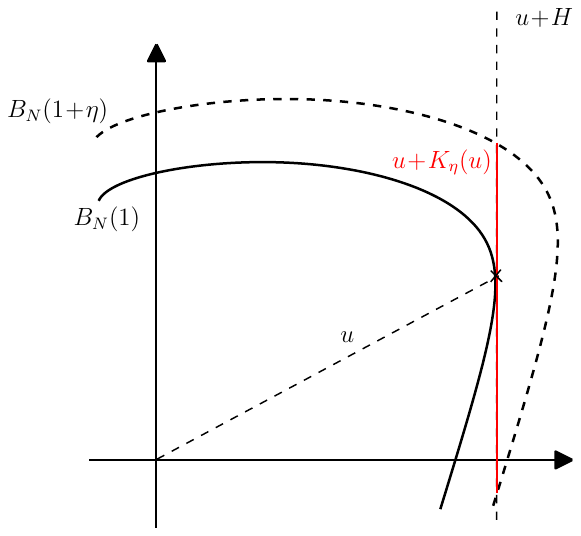} &
\includegraphics[width=7cm]{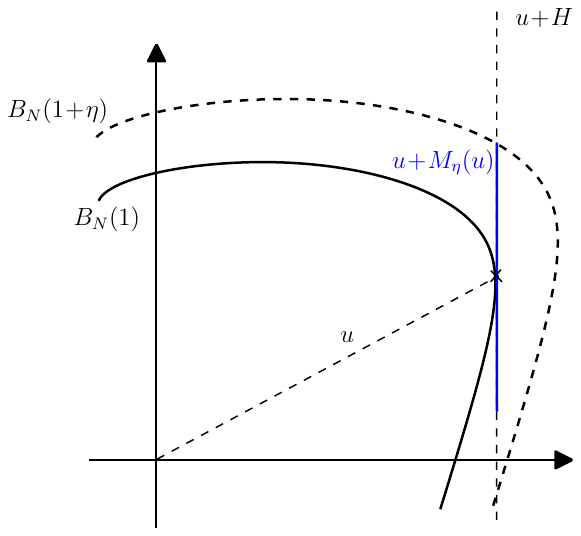}
\end{tabular}
\end{center}
\caption{Illustration of the sets $K_\eta(u)$ and $M_\eta(u)$ (note that $M_\eta(u)$ is here strictly included  in $K_\eta(u)$).}
\label{Fig:boule}
\end{figure}

We introduce now a function $I$ defined by an integral and check that this function is of the same order as the function $h_u$.
\begin{lemma}
\label{l:I+}
For any $(u,H)$ satisfying \eqref{e:(u,H)} and $u^\star$ defined by \eqref{e:u*}, for $\eta>0$, let us define,
$$ I^+ (\eta) = \int_{\mathbb{R}^d \cap \{ x: x\cdot u^\star >0 \}} \exp ( - (N(x) - (1-\eta) x \cdot u^\star)) dx. $$
Then there exists two constants $c_1,c_2>0$ depending only on $d$ and $N$ such that, for any $(u,H)$ satisfying \eqref{e:(u,H)}, for $\eta>0$, we have
\begin{equation}
\label{eq:I_1_v1}
c_1h_u(\eta) \le I^+(\eta)\le c_2h_u(\eta).
\end{equation}
Moreover, define also, for $\eta >0$,
$$ I (\eta) = \int_{\mathbb{R}^d } \exp ( - (N(x) - (1-\eta) x \cdot u^\star)) dx \,, $$
then for $\eta$ small enough (depending on $d,N,u$ and $H$), we have 
\begin{equation}\label{eq:I_1}
c_1 h_u(\eta) \le I(\eta)\le 2c_2 h_u(\eta).
\end{equation}
\end{lemma}

Finally, we  state a basic geometrical lemma that will be useful in what follows.
\begin{lemma}\label{lem:CDV} There exist two constants $c,c'>0$ depending only on $d$ and $N$ such that, for all $(u,H)$ satisfying \eqref{e:(u,H)}, for $u^\star$ defined by \eqref{e:u*}, we have 
$$c\le \|u^\star\|_2 \le c'.$$
\end{lemma}

\begin{proof} Using that $N$ and $||\cdot ||_2$ are equivalent, there exist $\alpha,\beta>0$ such that 
$$B_{\|\cdot\|_2}(\alpha)\subset B_N(1)\subset B_{\|\cdot\|_2}(\beta).$$
So, we get the lower bound on $\|u^\star\|_2$ noticing that, for $u\in B_N(1)$,
$$ 1=u\cdot u^\star\le \|u\|_2\|u^\star\|_2\le \beta \|u^\star\|_2.$$
Moreover, $u+H$ is a supporting hyperplane of $B_N(1)$ at $u$ and $u^\star$ is normal to $H$ so $B_N(1)$ is included in $\{v\in \R^d \,:\, v\cdot u^\star\le u\cdot u^\star\}$. Applying this inequality to  $\alpha u^\star/\|u^\star\|_2  \in B_{\|\cdot\|_2}(\alpha)\subset  B_N(1)$, we get  
$$\alpha\|u^\star\|_2=\alpha \frac{u^\star}{\|u^\star\|_2}\cdot u^\star\le u\cdot u^\star=1.$$ 
Finally, we get 
$$\beta^{-1}\le \|u^\star\|_2 \le \alpha^{-1}.$$
\end{proof}


\section{Study of $\mu_\varepsilon (u)$}
\label{s:mu2}
The present section is devoted to the study of the time constant $\mu_\varepsilon (u)$ associated with a general norm $N$ for the model with rewards located at the points of $\Xi$. With all $(u,H)$ satisfying \eqref{e:(u,H)} we associate the functions $g_u$ and $\bar{g}_u$ given by Proposition \ref{prop:function_h}. 
We prove the following theorem.
\begin{theorem}\label{theo:mu} There exist constants $C_1,C_2>0$ (depending only on $d$ and $N$) such that for all $(u,H)$ satisfying \eqref{e:(u,H)}, the following assertions hold.
\begin{itemize}
\item[(i)]  For $\varepsilon$ small enough (depending on $d,N,u$ and $H$), we have 
$$\bar{g}_u(C_1\varepsilon^{-1})\le 1- \mu_{\varepsilon}(u)\le g_u(C_2\varepsilon^{-1})\,.$$
\item[(ii)] For $\varepsilon$ small enough (depending on $d,N,u$ and $H$), we have, for any $\delta>0$, a.s. for large $s$, for any geodesic $\gamma_{\varepsilon}(su) \in \Pi (0,su)$ from $0$ to $su$ for the time $T_{\varepsilon} (0,su)$,
\begin{eqnarray*}
(1-\delta)\varepsilon^{-1/d} \bar{g}_u(C_1\varepsilon^{-1})s \le& \sharp \gamma_{\varepsilon}(su) &\le (1+\delta)\varepsilon^{-1/d}g_u(C_2\varepsilon^{-1})s \\
 &   N(\gamma_{\varepsilon}(su))-s  &\le  (1+\delta) g_u(C_2\varepsilon^{-1})s.
  \end{eqnarray*}
\end{itemize}

\end{theorem}
Notice that how small $\varepsilon$ has to be in Points (i) and (ii) of Theorem \ref{theo:mu} depends on $N$ and $d$ but also on $u$ and $H$. This dependence appears for instance through the use of Lemma \ref{l:I+} (see how \eqref{eq:I_1_v1} implies \eqref{eq:I_1} for $\eta$ small enough).

It is a little bit frustrating not to be able to give a lower bound on $N(\gamma_{\varepsilon}(su))$ in Theorem \ref{theo:mu}. In fact, such a lower bound exists in some specific cases, for directions $u$ in which $B_N(1)$ does not have a $(d-1)$-dimensional flat edge: see Proposition \ref{prop:minN} in Section \ref{s:chernov} for more details. However, we have no hope with our approach to obtain a lower bound in general direction, see Remark \ref{r:longueur} for a concrete example.


\subsection{Proof of the lower bound on $1-\mu_\varepsilon(u)$}
\label{s:glouton}

This section is devoted to the proof of a lower bound on $1-\mu_\varepsilon(u)$, {\it i.e.}, the proof of Proposition \ref{t:thm8} - and incidentally, as a corollary, we get a lower bound on $\sharp \gamma_\varepsilon (su)$ for large $s$ too, see Corollary \ref{c:lowersharp} below. To do so, it is enough to exhibit a path from $0$ to $su$ (for a given direction $u$ and $s$ large enough) whose $N$-length is not too high, but that gathers enough rewards. This is done in Proposition \ref{t:thm8} through a greedy algorithm, as sketched in Section \ref{s:sketch}. This algorithm adds recursively to the path $\pi = (0,x_1 \cdots , x_{n-1})$ the closest point $x_n$ of the Poisson point process $\Xi$ (in a certain sense) which is located in a cone with origin $x_{n-1}$ and oriented in the direction $u$. This cone condition gives us a good upper bound on the $N$-length of the path we construct. However, we have to be careful in our procedure to be sure that the path $\pi$ does not go too far away from the prescribed direction $u$. We thus take care to compensate any gap previously created between the direction of $x_{n-1}$ and the prescribed direction $u$ by choosing wisely the direction of the next point $x_n$ the path collects. This is the reason why it is the set $M_\eta (u)$, which is a symmetrized version of $K_\eta (u)$, that arises in the proof.

Fix $(u,H)$ satisfying \eqref{e:(u,H)}, {\em i.e.}, fix a direction $u\in \mathbb{R}^d$ with $N(u)=1$, and fix $H$ such that $u+H$ is a supporting hyperplane of $B_N(1)$ at $u$. Recall that
$$M_\eta(u):= \{ v \in H \,:\, N(u+v) \leq 1+ \eta \mbox{ and } N(u-v)  \leq 1+ \eta \} \qquad \mbox{ and } \qquad \bar{h}_u(\eta):=\eta^{-d}|M_\eta(u)| \,.$$
Proposition \ref{prop:function_h} states that   $\eta\mapsto \bar{h}_u(\eta)$ is a decreasing homeomorphism from $(0,\infty)$ to $(0,\infty)$ so we can define its inverse function $\bar{g}_u$ which is also a decreasing homeomorphism from $(0,\infty)$ to $(0,\infty)$. To prove the lower bound on $1-\mu_\varepsilon(u)$ given in (i) of Theorem \ref{theo:mu}, we prove in fact the following proposition.

\begin{prop}
\label{t:thm8}
There exists a constant $c>0$ (depending only on $d,N$) such that, for all $(u,H)$ satisfying \eqref{e:(u,H)}, for all $\varepsilon >0$,
for all $\eta\in (0,1)$, we have
$$ \mu_\varepsilon (u) \leq 1 + \eta - c\varepsilon^{1/d}  |M_\eta(u) |^{1/d}.  $$
In particular, with $C_1:=\left( \frac{2}{c} \right)^d$, we have that for $\varepsilon$ small enough (depending on $d,N,u$ and $H$),
  $$ \mu_\varepsilon (u) \leq 1 -  \bar{g}_u \left( C_1 \varepsilon^{-1} \right) \,.$$
\end{prop}
\begin{proof}
To prove this proposition, we construct, using a greedy algorithm, a path $\pi_s$ from $0$ to $su$ with length close to $s(1+\eta)$ but going through a number of points of $\Xi$ larger than $c  |M_\eta(u) |^{1/d}s$ (see Figure \ref{Fig:glouton}). 

We first introduce some notation. For $s\geq0$, let
\begin{align*}
C_\eta (s) & :=  \{ x=\lambda u +v\, : \, 0 \leq \lambda \leq s \,,\, v \in H \,,\,N(\lambda u + v) \leq \lambda (1+\eta), N(\lambda u - v) \leq \lambda (1+\eta) \}\,. 
\end{align*}
Hence,  $C_\eta (s)$ is a cone with axis $u$  and of basis  $M_\eta$ (basis not necessarily orthogonal to its axis). Thus, if $u^\star$ denotes the  vector orthogonal to $H$ such that $u\cdot u^\star=1$, we have\footnote{
This can be proven using the same change of variable $\Psi$. Indeed, 
\begin{align*}
 |C_\eta (s)|  & = \int_{\mathbb{R}^d} \1_{x \in C_\eta (s)} dx = \frac{1}{\| u^\star\|_2} \int_{\mathbb{R}^d} \1_{0\leq \lambda \leq s} \1_{v/\lambda \in M_\eta (u) } d\lambda \, dv\\
 & = \frac{1}{\| u^\star\|_2} \int_{\mathbb{R}^d} \lambda^{d-1} \1_{0\leq \lambda \leq s} \1_{w \in M_\eta (u) } d\lambda \, dw = \frac{1}{\|u^\star\|_2} | M_\eta(u) | \frac{s^d}{d} \,.
 \end{align*}
}
$$   |C_\eta (s) | = \frac{1}{d  \|u^\star\|_2} | M_\eta(u) | s^d. $$
In the algorithm, we construct a path $(0,x_1,\dots , x_n, su)$ from $0$ to $su$ such that $x_{i+1}$ is the first point of $\Xi$ in the cone $x_i + C_\eta(\infty)$. We prove that the length of such path cannot be much larger than $(1+\eta) s$ and its number of points is at least of order $|M_\eta(u)|^{1/d} s$. Optimizing on $\eta$ yields then Proposition \ref{t:thm8}.

Recall that $\Xi$ denotes the points of a Poisson point process with intensity $1$ on $\mathbb{R}^d$.
By induction we define a sequence $(X_n)_{n\ge 0}$ in $\R^d$ with $X_0=0$
and  if
$$ \lambda_{n+1} := \inf \{ s>0 \,:\, (X_n + C_\eta (s)) \cap \Xi \neq \emptyset \} \,,$$
we set $X_{n+1} = X_n +\lambda_{n+1} u + W_{n+1}$ with  $W_{n+1} \in H$ and such that $\{X_{n+1}\}= (X_n+ C_\eta (\lambda_{n+1})) \cap \Xi$. We define   $S_n = \sum_{i=1}^n \lambda_i $ and $V_n=\sum_{i=1}^n W_i$. By the proprieties of Poisson point processes, $(\lambda_i)_{i\ge 1}$ and $(W_i)_{i\ge 1}$ are i.i.d. random variables and since $M_\eta(u)$ is a symmetric set, $W_1$ has also a symmetric distribution and so $V_n$ is a symmetric random walk.

Let us now define, for $s>0$,
$$ Z(s) := \sup \{ n\geq 0 \,:\, S_n \leq s \} $$
and consider the path $\pi_s = (0,X_1,\dots, X_{Z(s)}, su)$ (see Figure \ref{Fig:glouton} for an illustration of the greedy algorithm). Let us find an upper bound of $T_\varepsilon (\pi_s)$ by finding an upper bound of its length and a lower bound of $Z(s)$, its number of points.

\begin{figure}
\begin{center}
\includegraphics[width=12cm]{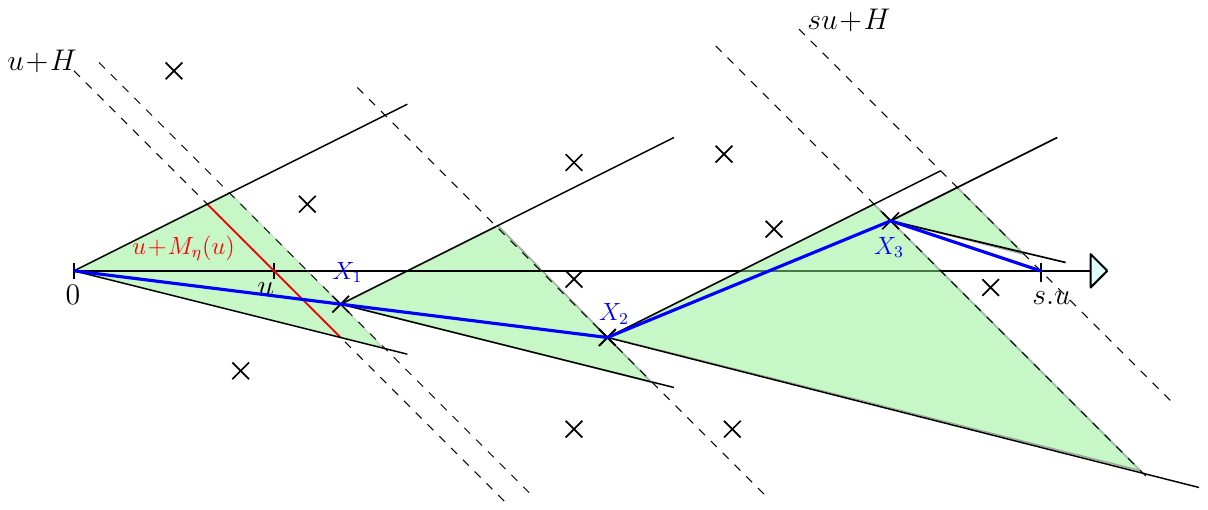}
\end{center}
\caption{Illustration of the greedy algorithm. The path constructed is drawn in blue. At each step, the path goes through the first point of $\Xi$ in the green cone in front of it.} 
\label{Fig:glouton}
\end{figure}

The sequence $(\lambda_i)_{i\geq 1}$ is  i.i.d.  and we have
\begin{align*}
\mathbb{P} (\lambda_{1} \geq t  ) & =  \exp \left( - \frac{1}{d \|u^\star\|_2}  | M_\eta(u) |t^d\right)
\end{align*}
so that
$$ \mathbb{E} (\lambda_1) =\left(\frac{ | M_\eta(u) |}{d \|u^\star\|_2} \right)^{-1/d} \int_0^\infty \exp (-v^d) dv := \frac{| M_\eta |^{-1/d} }{2c_u}  $$
where $$2c_u:=\left(\frac{1}{d \|u^\star\|_2}\right)^{1/d} \left(\int_0^\infty \exp (-v^d) dv \right)^{-1}.$$
Since $S_n = \sum_{i=1}^n \lambda_i $, a standard renewal theorem (see for instance \cite{Grimmett_Stirzaker} Section 10.2) yields that, for $s$ large enough, we have
\begin{equation}
\label{e:ajout3}
c_u |M_\eta(u) |^{1/d} s \leq Z(s) \leq 3c_u |M_\eta(u) |^{1/d} s \,.
\end{equation}
This gives a lower bound for the number of points of $\pi_s$. It remains now to find an upper bound for the length of $\pi_s$. Note first that if we consider the path $\pi'_s = (0,X_1,\dots , X_{Z(s)})$, we have by construction
$$ N (X_{i+1} - X_i) \leq (1+\eta) \lambda_{i+1} $$
so
$$ N(\pi'_s) \leq (1+\eta) S_{Z(s)} \,. $$
Thus 
$$ N(\pi_s) \leq (1+\eta) S_{Z(s)} + N(su - (S_{Z(s)} u + V_{Z(s)})) \leq (1+\eta) S_{Z(s)} + s - S_{Z(s)} + N (V_{Z(s)})\leq (1+\eta) s + N (V_{Z(s)}) \,.$$
Let us now prove that $N(V_{Z(s)})/s$ tends a.s. to $0$.
The process $(V_k)_{k\ge 0}$ is a symmetric random walk with increments $(W_k)_{k\geq 1}$ and by construction, we have
$$  N(\lambda_1 u + W_1) \leq \lambda_1 (1+\eta) $$
which yields
$$ N(W_1) \leq \lambda_1 (1+\eta) + N(\lambda_1 u) \leq 3 \lambda_1 $$
for $\eta \in (0,1)$. This implies that $\E(N(W_1))$ is finite. 
By the law of large numbers, we obtain that
$$\lim_{n\to \infty} \frac{V_{n}}{n}=0 \quad \mbox{a.s.}$$
Using that for $s$ large enough  $Z(s)\le 3c_u |M_n|^{1/d} s$ a.s., this implies that 
$$\lim_{n\to \infty} \frac{N(V_{Z(s)})}{s} = \lim_{n\to \infty} \frac{N(V_{Z(s)})}{Z(s)} \cdot   \frac{Z(s)}{s}  =0 \quad \mbox{a.s.}$$
Hence, for all $\eta \in (0,1)$, we have, for $s$ large enough,
$$T_\varepsilon (0,su) \leq  T_\varepsilon (\pi_s) \leq s \left( 1 + \eta+ \frac{N(V_{Z(s)})}{s} - \varepsilon^{1/d} c_u | M_\eta(u) |^{1/d}\right) \,.$$
We deduce
$$ \mu_\varepsilon (u) = \lim_{s\rightarrow \infty} \frac{T_\varepsilon (0,su)}{s} \leq 1 + \eta - \varepsilon^{1/d} c_u |M_\eta(u) |^{1/d} =1-\eta\left(c_u\left(\varepsilon\bar h_u(\eta)\right)^{1/d}-1\right)  \,.$$
Let us define $c_1:=\inf\{c_u, u\in B_N(1)\}$ and note, in view of Lemma \ref{lem:CDV}, that $c_1>0$. We get, for all $u\in B_N(1)$,
$$ \mu_\varepsilon (u) \le 1-\eta\left(c_1\left(\varepsilon\bar h_u(\eta)\right)^{1/d}-1\right) \,. $$
This proves the first part of Proposition \ref{t:thm8}. Taking
$$ \eta = \bar{g}_u \left( \left( \frac{2}{c_1} \right)^d \varepsilon^{-1} \right) \,, $$
(notice that $\eta < 1$ for $\varepsilon$ small enough), we obtain that $c_1\left(\varepsilon\bar h_u(\eta)\right)^{1/d}=2$, thus
$$ \mu_\varepsilon (u) \leq 1-\eta \,. $$
\end{proof}

Let us note that this implies a lower bound for the number of points taken by the geodesic, that we state in the following Corollary.
\begin{corollary}
\label{c:lowersharp}
Let $C_1 = C_1 (d,N)$ be the constant given by Proposition \ref{t:thm8}. For all $(u,H)$ satisfying \eqref{e:(u,H)}, for $\varepsilon$ small enough (depending on $d,N,u$ and $H$), we have, for any $\delta>0$, a.s. for large $s$, for any geodesic $\gamma_\varepsilon(su)$ from $0$ to $su$ for the time $T_\varepsilon(0,su)$,
\begin{equation}\label{eq:lowerboundgeodesique}
\varepsilon^{1/d}\sharp \gamma_\varepsilon(su)\ge (1-\delta) \bar{g}_u(C_1 \varepsilon^{-1})s \,.
\end{equation}
\end{corollary}

\begin{proof}
Indeed, if $\gamma_\varepsilon(su)$ denotes a geodesic from $0$ to $su$, we have
$$\varepsilon^{1/d} \sharp \gamma_\varepsilon(su)=N(\gamma_\varepsilon(su))-T_\varepsilon(0,su)\ge s \left( 1-\frac{T_\varepsilon(0,su)}{s} \right) \,. $$
Using  that $1-\frac{T_\varepsilon(0,su)}{s}$ converges to $1-\mu_\varepsilon(u)$ and Proposition \ref{t:thm8}, we get that for any $\varepsilon$ small enough, for any $\delta>0$, a.s. for $s$ large enough,
\begin{equation*}
\varepsilon^{1/d}\sharp \gamma_\varepsilon(su)\ge (1- \delta) \bar{g}_u(C_1 \varepsilon^{-1})s \,.
\end{equation*}
\end{proof}


\subsection{Proof of the upper bound on $1-\mu_\varepsilon(u)$}
\label{s:chernov}

We want now to find a lower bound for $\mu_\varepsilon(u)$, as stated in Theorem \ref{theo:mu}. The idea of the proof as been presented in Section \ref{s:sketch}. It relies on a bootstrap argument, that involves a control on the length $N(\gamma_\varepsilon(su))$ and the number of rewards $ \sharp \gamma_\varepsilon(su)$ of a geodesic  $ \gamma_\varepsilon(su)$. The key property used in the proof is Proposition \ref{prop:Chernov} below, that states that a path from $0$ to $su$ with small $N$-length cannot collect too much rewards. In the course of the proof, we will also prove the upper bounds on $N(\gamma_\varepsilon(su))$ and $ \sharp \gamma_\varepsilon(su)$ stated in Point (ii) of Theorem \ref{theo:mu}. We conclude this section by noticing that, at least under some added hypothesis, we can recover a lower bound on $N(\gamma_\varepsilon(su))$ (see Proposition \ref{prop:minN}) by a final last use of Proposition \ref{prop:Chernov} and the lower bound on $\sharp \gamma_\varepsilon(su)$ proved in Section \ref{s:glouton} (see Corollary \ref{c:lowersharp}).

We start by proving the following proposition.
\begin{prop} 
\label{prop:Chernov} Let us consider the event
$$\mathcal{E}(\eta,s,u,c):=\{\exists \pi \in \Pi (0,su),   N (\pi)\le (1+\eta )s, \sharp \pi \ge c|K_\eta(u)|^{1/d}s\}.$$
Then there exists some constants $c,c'>0$ depending only on $d$ and $N$ such that for any $(u,H)$ satisfying \eqref{e:(u,H)}, there exists $\bar{\eta}>0$ (depending on $d, N ,u$ and $H$) such that, for $\eta<\bar{\eta}$, for all $s>0$, we have
\begin{equation}\label{eq:propChernov}
 \mathbb{P} [\mathcal E (\eta,  s,u,c) ] \leq 2\exp [-s c'|K_\eta(u)|^{1/d} ].
\end{equation}
In particular, for all $\eta \in (0,\bar{\eta})$ there exists a.s. $s_\eta$ such that, for $s\geq s_\eta$, every path $\pi \in \Pi (0,su)$ such that $N(\pi) \leq (1+ \eta) s$ satisfies $\sharp \pi \leq c|K_\eta(u)|^{1/d} s$.
\end{prop}

To prove this proposition, we will need the following lemma which roughly states the same thing except that we consider now a path from the origin $0$ to the hyperplane $su+H$. We denote by $\Pi (0,su +H)$ the set of all polygonal paths $\pi \in \Pi (0,y)$ for some $y\in su+H$. We recall that the definition of $I(\eta)$ for $\eta >0$ is given in Lemma \ref{l:I+}.

\begin{lemma}\label{lemmeChernov}For $A>0$ and $(u,H)$ satisfying \eqref{e:(u,H)}, define
$$\mathcal{F}(\eta,s,u,A,H):=\{\exists \pi \in \Pi (0,su +H) , \forall x\in \pi, x\cdot u^\star \le s,  N (\pi)\le (1+\eta )s, \sharp \pi \ge As\}.$$
Then, for all $A>0$, for all $(u,H)$ satisfying \eqref{e:(u,H)}, for all $s>0$ and $\eta >0$, 
$$ \mathbb{P} [\mathcal F (\eta,s,u, A, H) ] \leq \exp [-s (A \log 2 - 2^{1 + 1/d} I (\eta)^{1/d} \eta) ] \,.$$
\end{lemma}

We first explain how Lemma \ref{lemmeChernov} implies the proposition.

\begin{proof}[Proof of Proposition \ref{prop:Chernov} using Lemma \ref{lemmeChernov}]
A path $\pi \in \Pi (0,su )$ can be decomposed into two paths $\pi_1$ and $\pi_2$ where $\pi_1$ is the restriction of the path $\pi$ until it intersects the hyperplane $su+H$ and $\pi_2$ is the intersection of the path $\pi$ afterwards. Note that $N(\pi)=N(\pi_1)+N(\pi_2)$ and moreover $N(\pi_1)\ge s$, since $\pi_1$ goes from $0$ to some $z\in su+H$ thus $N(\pi_1) \geq N (z) \geq N (su) = s$. If $\pi$ satisfies the property appearing in the definition of the event $\mathcal{E}(\eta,s,u,c)$, we thus get $N(\pi_2)\le \eta s$ and moreover we have either $\sharp \pi_1 \ge c|K_\eta(u)|^{1/d}s/2$ or  $\sharp \pi_2 \ge c|K_\eta(u)|^{1/d}s/2$. Hence, if we define  
$$\mathcal{G}(\eta,s,u,c):=\{\exists \pi \in \Pi (su,\star),   N(\pi)\le \eta s, \sharp \pi \ge \frac{c}{2}|K_\eta(u)|^{1/d}s\} \,,$$
we get 
$$\mathcal{E}(\eta,s,u,c)\subset \mathcal{F} \left(\eta,s,u,\frac{c}{2}|K_\eta(u)|^{1/d},H\right)\cup  \mathcal{G}(\eta,s,u,c) \,.$$
Using \eqref{eq:I_1} in Lemma \ref{l:I+}, we see that we can choose $c$  such that for $\eta$ small enough we have 
$$\frac{c|K_\eta(u)|^{1/d}}{2} \log 2\ge \eta 5I(\eta)^{1/d}$$
so that
$$\frac{c|K_\eta(u)|^{1/d}}{2} \log 2 - 2^{1 + 1/d} I (\eta)^{1/d} \eta\ge I (\eta)^{1/d} \eta (5-4)= I (\eta)^{1/d} \eta.$$
Thus by Lemma \ref{lemmeChernov} we get
$$ \mathbb{P} \left[ \mathcal{F} \left(\eta,s,u,\frac{c}{2}|K_\eta(u)|^{1/d},H\right) \right] \leq \exp [-s  I (\eta)^{1/d} \eta) ]\le \exp [-s  c_1|K_\eta(u)|^{1/d} ]$$
for some constant $c_1$.
Moreover, let $q=\frac{c}{2}|K_\eta(u)|^{1/d}s$, then for any $\beta >0$, using the fact that $\1_{x>0}\leq e^{\beta x}$ and by Mecke Equation (see Theorem 4.4 in \cite{last_penrose_2017})
\begin{eqnarray*}
 \mathbb{P} [\mathcal{G}(\eta,s,u,c)]&=& \mathbb{P} [\exists \pi \in \Pi ( 0 , \star),   N(\pi)\le \eta s, \sharp \pi \ge q]\\
 &\le&\mathbb{P} [\exists (x_1,\ldots,x_q)\in \Xi ,\;   \sum_{i=1}^q N(x_i-x_{i-1})\le \eta s]\\
 & \le & \mathbb{E} \left[ \sum_ {(x_1,\ldots,x_q)\in \Xi} \1_{ \eta s - \sum_{i=1}^q N(x_i-x_{i-1})\ge 0} \right]\\
  & \le &\exp(\beta \eta s) \mathbb{E} \left[ \sum_ {(x_1,\ldots,x_q)\in \Xi} \exp{  \left(  - \beta \sum_{i=1}^q N(x_i-x_{i-1})\right)} \right]\\
   &=&\exp(\beta \eta s)\left(\int_{(\mathbb{R}^d)^q} \exp(-\beta \sum_{i=1}^q N(x_i-x_{i-1}) ) dx_1 \dots dx_q \right)\\
 &=&\exp(\beta \eta s)\left(\int_{\mathbb{R}^d} \exp(-\beta N(x))dx\right)^q\\
 &=&\exp(\beta \eta s)\left(\frac{1}{\beta^d}\int_{\mathbb{R}^d} \exp(- N(x))dx\right)^q \,.
\end{eqnarray*}
Chose $\beta_0:=\beta_0(d)$ such that the bracket above is equal to $1/2$.
Then
\begin{eqnarray*}
 \mathbb{P} [\mathcal{G}(\eta,s,u,c)]&\le &\exp(\beta_0 \eta s-q\log 2 )\\
 &=&\exp \left( -\eta s \left( \frac{c}{2}h_u(\eta)^{1/d}\log 2-\beta_0 \right) \right),
\end{eqnarray*}
where $h_u(\eta)=\eta^{-d}|K_\eta(u)|$. By Proposition \ref{prop:function_h}, $h_u(\eta)$ tends to infinity as $\eta$ tends to 0, we choose $\bar{\eta}:=\bar{\eta}(d,N,u,H)$ such that $\beta_0\le\frac{c}{4}h_u(\eta)^{1/d}(2\log 2-1)$ for $\eta\le \bar{\eta}$. 
This yields
$$\mathbb{P} [\mathcal{G}(\eta,s,u,c)]\le \exp \left[-s  \frac{c}{4}|K_\eta(u)|^{1/d}) \right].$$
This implies that \eqref{eq:propChernov} holds with $c'=\min(c_1,c/4)$.
Moreover, if we set $\lambda_n = n /(c |K_\eta(u)|^{1/d})$,  we get
$$ \sum_{n=1}^\infty \mathbb{P} [ \mathcal{E} (\eta, \lambda_n , u,c) ] \leq 2 \sum_{n=1}^\infty   \exp \left[-\frac{c'}{c}n\right]  \,.$$
Thus, Borel-Cantelli Lemma yields that for $n$ large enough, every path $\pi \in \Pi( 0, \lambda_n u)$ such that $N(\pi) \leq (1+\eta) \lambda_n$ satisfies $\sharp \pi < c |K_\eta(u)|^{1/d} \lambda_n = n$. Then for $s\in [\lambda_n, \lambda_{n+1}[$, if $\pi$ is a path from $0$ to $ su$ such that $N(\pi) \leq (1+\eta) s$, adding to $\pi$ a straight line colinear to $u$ yields a path from $0$ to $\lambda_{n+1}u$  such that $N(\pi) \leq (1+\eta) \lambda_{n+1}$. Thus we deduce that $\sharp \pi < n+1$. Since $\lfloor c |K_\eta(u)|^{1/d} s\rfloor = n$, this implies in fact that $\sharp \pi \le c |K_\eta(u)|^{1/d}s$.
\end{proof}

We now prove Lemma \ref{lemmeChernov}.

\begin{proof}[Proof of Lemma \ref{lemmeChernov}] Let $q=\lceil As \rceil$. Let $\pi=(0,x_1,\ldots,x_q,x_{q+1})\in \Pi (0,su+H)$ where $x_1,\ldots,x_q$ are points of the $\Xi$ and $x_{q+1}-x_q$ is colinear to $u$. Set $x_0=0$. Thus
$$N(\pi)=\sum_{i=1}^{q+1} N(x_i-x_{i-1})=\sum_{i=1}^{q} N(x_i-x_{i-1})+|s-x_{q}\cdot u^\star|=\sum_{i=1}^{q} N(x_i-x_{i-1})+s-x_{q}\cdot u^\star \,.$$
Thus
$$N(\pi)\le (1+\eta )s \Leftrightarrow \sum_{i=1}^{q} N(x_i-x_{i-1})\le \eta s + x_{q}\cdot u^\star.$$
We get for any $\alpha, \beta >0$, using the fact that $\1_{x>0}\leq e^{ x}$ and by Mecke Equation (see Theorem 4.4 in \cite{last_penrose_2017}),
\begin{eqnarray*}
\mathbb{P} [\mathcal F (\eta, s,u,A,H) ] &=& \P(\exists x_1,\ldots,x_q\in \Xi, x_{q}\cdot u^\star\le s, \sum_{i=1}^{q} N(x_i-x_{i-1})\le \eta s + x_{q}\cdot u^\star) \\
&\le & \E\left(\sum_{x_1,\ldots,x_q\in \Xi} \1_{x_{q}\cdot u^\star\le s} \1_{ \sum_{i=1}^{q} N(x_i-x_{i-1})\le \eta s + x_{q}\cdot u^\star}\right) \\
&\le & \int_{(\R^d)^q} \exp\left(\beta ( s-x_{q}\cdot u^\star)+\alpha \left( \eta s - \sum_{i=1}^{q} N(x_i-x_{i-1}) + x_{q}\cdot u^\star \right) \right)dx_1\ldots dx_q \,.
\end{eqnarray*} 
Taking $\beta=\alpha \eta$, we get
\begin{align*}
\mathbb{P}& [\mathcal F (\eta, s,u,A,H) ]\\  &\le   \exp(2\alpha \eta s)\int_{(\R^d)^q}\exp\left(-\alpha \left( \sum_{i=1}^{q} N(x_i-x_{i-1})-(1-\eta)\sum_{i=1}^{q}(x_i-x_{i-1})\cdot u^\star \right)\right)dx_1\ldots dx_q\\
&= \exp(2\alpha \eta s)\left(\int_{\R^d}\exp\left(-\alpha(  N(x)-(1-\eta)x\cdot u^\star)\right)dx\right)^q\\
&= \exp(2\alpha \eta s)\left(\frac{I(\eta)}{\alpha^d}\right)^q.
\end{align*} 
We conclude taking $\alpha$ such that $\alpha^d=2I(\eta)$.

\end{proof}

We can now prove the upper bound on the number of points and the length of a geodesic given in (ii) of Theorem \ref{theo:mu}, \emph{i.e.}, we can prove that there exists a constant $C_2$ (depending only on $d$ and $N$) such that, for any $(u,H)$ satisfying \eqref{e:(u,H)}, for $\varepsilon$ small enough (depending on $d,N,u$ and $H$), we have for any $\delta>0$,  a.s., for $s$ large enough, for any geodesic $\gamma_\varepsilon (su)\in \Pi (0,su)$ from $0$ to $su$ for the time $T_{\varepsilon} (0,su)$,
\begin{equation}\label{eq:upperboundgeodesique}
N(\gamma_\varepsilon (su))\le s(1+(1+\delta) g_u(C_2\varepsilon^{-1})) \textrm{ and } \sharp \gamma_\varepsilon (su) \leq (1+\delta) g_u(C_2\varepsilon^{-1}) \varepsilon^{-1/d} s. 
\end{equation}

\begin{proof}[Proof of \eqref{eq:upperboundgeodesique}]
We begin by proving that there exists some $C>0$ such that, for $\varepsilon$ small enough, we have a.s. for $s$ large enough
\begin{equation}
\label{e:29}
N(\gamma_\varepsilon (su)) \leq s [1 + C \varepsilon^{1/d}]\,.
\end{equation}
We use tools that come from the study of the greedy paths and greedy lattice animals. Let 
$$ G(s) := \sup \left\{ \frac{\sharp \pi}{N(\pi)} \,:\, \pi \in \Pi( 0 , \star) \,,\, \pi  \not\subset B_N(s) \right\} \leq \sup \left\{ \frac{\sharp \pi}{N(\pi)} \,:\, \pi \in \Pi (0 , \star) \right\} := G\,.  $$
Let $G(\infty)$ be the increasing limit of $G(s)$. We have the follonwing properties : $G(\infty) \leq G$, $G(\infty)$ is constant a.s. and
$$ G(\infty) = \mathbb{E} [G(\infty)] \leq \mathbb{E} [G] \leq c \int_0^\infty \delta_1 ([r,+\infty[)^{1/d} dr := C/4 \,. $$
This is a consequence of $(11)$ and Lemma 2.1 in \cite{GM08} (which is the analog in a continuous setting of a result by Martin \cite{Martin02} in a discrete setting), and these results have already been useful to study continuous first-passage percolation (see \cite{GT17} Theorem 3.1 or \cite{GT22} Theorem 13 and Corollary 14).
Thus, a.s. for large enough $s$, for all $\pi \in \Pi ( 0, \star)$ such that $\pi \not\subset B_N(s/2) $ we have 
\begin{equation}
\label{e:ajoutcrude}
\sharp \pi \leq (C/2) N(\pi)\,.
\end{equation}
In particular, this holds for $\gamma_\varepsilon (su)$, a geodesic from $0$ to $su$. Using that $T(\gamma_\varepsilon (su)) \leq s$, we get
$$ s \geq T(\gamma_\varepsilon (su)) = N (\gamma_\varepsilon (su)) \left[1 - \varepsilon^{1/d} \frac{\sharp \gamma_\varepsilon (su)}{N(\gamma_\varepsilon (su))} \right] \geq N (\gamma_\varepsilon (su)) [1 - C \varepsilon^{1/d}/2] $$
which gives $N(\gamma_\varepsilon (su)) \leq s [1 - C \varepsilon^{1/d}/2]^{-1} \leq s [1+C \varepsilon^{1/d}]$ for $\varepsilon$ small enough.

Let's now consider $c>0$ such that the conclusion of Proposition \ref{prop:Chernov} holds, \emph{i.e.}, for any $u\in B_N(1)$, for $\eta$ small enough we have a.s. that for $s$ large enough, any path $\pi:0\mapsto su$ such that $N(\pi)\le (1+\eta)s$ go trough less than $c|K_\eta(u)|^{1/d}$ points of $\Xi$.
Now define the sequence $(\eta_n)_{n\ge 0}$ by
$$\eta_0 = C \varepsilon^{1/d} \mbox{ and } \eta_{n+1} = c|K_{\eta_n}(u)|^{1/d}\varepsilon^{1/d}.$$
The function $\eta\mapsto |K_\eta(u)|$ is continuous and increases with $\eta$ so the sequence 
 $(\eta_n)$ is monotonic. Assume that $\varepsilon$ is small enough such that $C\varepsilon^{1/d}\le 1$ and $c|K_{1}(u)|^{1/d}\varepsilon^{1/d}\le 1$. We get then by induction that the the sequence 
 $(\eta_n)$ remains in $[0,1]$ so it converges to a solution of the equation
 $$ \eta = c|K_{\eta}(u)|^{1/d}\varepsilon^{1/d}\,,$$
 \emph{i.e}, either to $0$ or to $\eta_\infty : =g_u((c^d\varepsilon)^{-1})$. Note that in both cases, if we fix some $\delta>0$, we have, for $n$ large enough, $\eta_n\le (1+\delta)\eta_\infty$.

Assume moreover that $\varepsilon$ is small enough such that Proposition \ref{prop:Chernov} holds for $\eta_0$, \emph{i.e.}, $\eta_0\le \bar{\eta}$. 
Let us prove by induction on $n$ that for all $n\geq 1$, there exists a.s. (a random) $s_n$ such that, for $s>s_n$, $\sharp \gamma_\varepsilon (su) \leq \eta_n \varepsilon^{-1/d} s$. We have seen that for $s$ large enough, $N(\gamma_\varepsilon (su)) \leq s [1 + \eta_0]$. Using Proposition \ref{prop:Chernov}, we get that for $s$ large enough $\sharp \gamma_\varepsilon (su) \leq c|K_{\eta_0}(u)|^{1/d} s = \eta_1 \varepsilon^{-1/d} s$. Note that if the sequence $(\eta_n)$ is non-decreasing, this directly implies that $\sharp \gamma_\varepsilon (su) \leq \eta_n \varepsilon^{-1/d} s$. Thus we can assume that $(\eta_n)$ is non-increasing.

Assume now that we proved that for $s$ large enough $\sharp \gamma_\varepsilon (su) \leq \eta_n \varepsilon^{-1/d} s$. Since $T(\gamma_\varepsilon (su)) \leq s$, we get that $N(\gamma_\varepsilon (su)) = T(\gamma_\varepsilon (su)) + \varepsilon^{1/d} \sharp \gamma_\varepsilon (su)  \leq s (1 + \eta_n)$. Since $\eta_n\le \eta_0\le \bar{\eta}$, we can again apply Proposition \ref{prop:Chernov} which yields that for $s$ large enough $\sharp \gamma_\varepsilon (su) \leq c|K_{\eta_n}(u)|^{1/d}  s = \eta_{n+1} \varepsilon^{-1/d} s $.
Hence, for all $n\geq0$, we have a.s. for $s$ large enough
$$ \sharp \gamma_\varepsilon (su) \leq \eta_n \varepsilon^{-1/d} s$$
and
$$ N(\gamma_\varepsilon (su)) \leq s (1 + \eta_n)\,.  $$
We conclude using that for  $n$ large enough, $\eta_n \leq (1+\delta) \eta_\infty$.
\end{proof}

\begin{proof}[End of the proof of Theorem \ref{theo:mu}]
Using \eqref{eq:upperboundgeodesique} and the inequality 
$ T_\varepsilon (0,su) \geq s - \varepsilon^{1/d} \sharp \gamma_\varepsilon (su)  $
we get that, for $\varepsilon$ small enough, we have for any $\delta>0$, a.s. for $s$ large enough,
$$ T_\varepsilon (0,su) \geq (1 - (1+\delta) g_u(c\varepsilon^{-1}) )s \,. $$
Since  $\mu_\varepsilon (u) = \lim_{s\rightarrow \infty} T_\varepsilon (0,su) / s$, we get for $\varepsilon$ small enough, for any $\delta>0$, 
$$ 1-\mu_\varepsilon (u) \leq (1+\delta)  g_u(c\varepsilon^{-1}). $$
Letting $\delta$ tend to $0$, we get the upper bound in (i) of Theorem \ref{theo:mu}. The lower bound in (i) of Theorem \ref{theo:mu} have been established in Proposition \ref{t:thm8}.
Moreover we just proved the upper bounds in (ii) of Theorem \ref{theo:mu} and the lower bound on $\sharp \gamma_\varepsilon(su)$ have been established in Corollary \ref{c:lowersharp}. 
\end{proof}

At this stage, we could hope to get a lower bound on $N(\gamma_\varepsilon (su))$, using Proposition \ref{prop:Chernov} and the lower bound on $\sharp \gamma_\varepsilon (su)$ obtained in Theorem \ref{theo:mu} $(ii)$. However, in a general setting, we cannot control the difference between $g_u$ and $\bar g_u$, thus what we get is not very satisfying. Nevertheless, we can state the following result, that will be useful for specific choices of the norm $N$. Consider the case where $| K_0 (u) | =0$. This hypothesis corresponds to the fact that $B_N (u)$ does not have a $(d-1)$-dimensional flat edge in the direction of $u$. Notice that $\hat \ell_u : \eta\mapsto |K_\eta (u)|$ is continuous and strictly increasing. Indeed $\hat \ell_u $ is obviously non-decreasing, and since $\hat \ell_u (\eta) = \eta^d h_u(\eta)$ the continuity of $\hat \ell_u$ is a consequence of the continuity of $h_u$ stated in Proposition \ref{prop:function_h}. In the course of the proof of Proposition \ref{prop:function_h}, we in fact proved that $(\hat \ell_u )^{\frac{1}{d-1}}$ is concave, see \eqref{e:ell_u}. Since $ (\hat \ell_u )^{\frac{1}{d-1}} $ is concave, non-decreasing and goes to infinity when $\eta$ goes to infinity, this function has to be increasing, thus $\hat \ell_u$ is increasing. If $| K_0 (u) | =0$, then $\hat \ell_u$ is a bijection from $[0,\infty)$ onto $[0,\infty)$. Let us denote by $\ell_u$ its inverse function. We can now state the following result.

\begin{prop}
\label{prop:minN} Let $u,H$ satisfying \eqref{e:(u,H)}.
Suppose that $| K_0 (u) | =0$, and denote by $\ell_u : [0,\infty) \rightarrow [0,\infty) $ the inverse function of $\eta\mapsto |K_\eta (u)|$. Let $c$ be the constant appearing in Proposition \ref{prop:Chernov}, and $C_1$ the constant appearing in Theorem \ref{theo:mu}. For $\varepsilon$ small enough (depending on $d,N,u$ and $H$), we have, for any $\delta \in (0,1]$, a.s. for large $s$, for any  geodesic $\gamma_\varepsilon (su)$ from $0$ to $su$ for the time $T_{\varepsilon} (0, su)$,
$$
N(\gamma_\varepsilon (su))-s\ge  \ell_u \left( (1-\delta)^d c^{-d} \varepsilon^{-1}\bar{g}_u(C_1\varepsilon^{-1})^d\right) s \,.
$$
\end{prop}

\begin{proof}[Proof of Proposition \ref{prop:minN}]
Let $c$ be the constant appearing in Proposition \ref{prop:Chernov}, and $C_1$ the constant appearing in Theorem \ref{theo:mu}.  
From the lower bound on $\sharp \gamma_\varepsilon (su)$ obtained in Theorem \ref{theo:mu} $(ii)$, we know that for $\varepsilon$ small enough, for any $\delta\in (0,1]$, we have a.s. for large $s$,
\begin{equation}
\label{e:a}
 \sharp \gamma_\varepsilon (su) \geq \left(1-\frac{\delta}{2} \right) \varepsilon ^{-1/d} \bar g _u (C_1 \varepsilon^{-1} )s \,.
 \end{equation}
Define $\eta :=  \ell_u \left( (1-\delta)^d c^{-d} \varepsilon^{-1}\bar{g}_u(C_1\varepsilon^{-1})^d\right) $.
Since we assume that $|M_0(u)|\le |K_0(u)|=0$, we get that $\bar{h}_u(\eta)=\eta^{-d}|M_\eta(u)|=o(\eta^{-d})$ as $\eta$  tends to 0. So its inverse function $\bar{g}_u$ satisfies in turn that $\bar{g}_u(x)=o(x^{-1/d})$ as $x$  tends to infinity. Using that $\ell_u(0)=0$ and $\eta\le \ell_u \left(  c^{-d} \varepsilon^{-1}\bar{g}_u(C_1\varepsilon^{-1})^d\right)$,  we get that $\eta$ tends to 0 as $\varepsilon$ tends to $0$ (uniformly with respect to $\delta \in (0,1])$.
Moreover, we have $ |K_\eta (u)| =  (1-\delta)^d c^{-d} \varepsilon^{-1}\bar{g}_u(C_1\varepsilon^{-1})^d$ by definition of $\ell_u$. This leads to 
\begin{equation}
\label{e:b}
c  |K_\eta (u)|^{1/d} = (1-\delta) \varepsilon ^{-1/d} \bar g _u (C_1 \varepsilon^{-1} )\,.
\end{equation}
By Proposition \ref{prop:Chernov}, we know that for $\varepsilon$ small enough, so that $\eta < \bar \eta$, a.s., for $s$ large enough, every path $\pi \in \Pi ( 0, su)$ such that $N(\pi)\leq (1+\eta) s$ satisfies $\sharp \pi \leq c |K_\eta (u)|^{1/d} s$. From \eqref{e:a} and \eqref{e:b}, we get that a.s., for $s$ large enough,
$$ N (\gamma_\varepsilon (su)) > (1+\eta) s $$
which ends the proof of Proposition \ref{prop:minN}.
\end{proof}

\begin{remark} 
For $(u,H)$ satisfying \eqref{e:(u,H)}, recall that $\Pi(0,su+H)$ is the set of paths from $0$ to some $y\in su+H$.  We can define the time travel from $0$ to $su+H$ by
$$T_\varepsilon(su+H):=\inf\{T_\varepsilon(\pi), \pi \in \Pi(0,su+H)\}$$ and prove, for small enough $\varepsilon$, that
$$\mu_{\varepsilon}(u,H):=\lim_{s\to \infty}\frac{T_\varepsilon(su+H)}{s}\quad \mbox{ exists a.s. and in $L^1$}.$$
Note that, in view of Lemma \ref{lemmeChernov}, one can use the same argument as above to show that $ 1- \mu_\varepsilon (u,H) \leq   g_u \left( C_2 \varepsilon^{-1} \right)$.
Moreover, to get a lower bound for this quantity, we can use a very similar greedy algorithm than the one explained in Section \ref{s:glouton} except that instead of looking at the next point of $\Xi$ in a cone of basis $M_\eta$, we look at the next point of $\Xi$ in a cone of basis $K_\eta$. The set $K_\eta$ not being symmetric, we cannot control anymore the deviation of the path with respect to the direction $u$ ($(V_n)_{n\ge 0}$ is not anymore a symmetric random walk) but the path created will still end at tome $su+y$, with $y\in H$, giving that 
  $ 1- \mu_\varepsilon (u,H) \ge   g_u \left( C_1 \varepsilon^{-1} \right).$
  Hence, we finally get 
$$g_u \left( C_1 \varepsilon^{-1} \right)\le 1- \mu_\varepsilon (u,H) \leq   g_u \left( C_2 \varepsilon^{-1} \right),$$
with the same function $g_u$ appearing  on both sides of the inequalities.
\end{remark}


\section{Study of $\tilde{\mu}_\epsilon(u)$}
\label{s:tildemu}

We now consider a general norm $N$, and study the model with balls of radii $\varepsilon^{1/d}/2$ centered at the points of $\Xi$. We can state the analog to Theorem \ref{theo:mu} in this setting, together with a comparison between the time constants in the model with rewards and the model with balls.
\begin{theorem}\label{theo:FPP} There exist some constants $C_2,C_3>0$ (depending only on $d$ and $N$, $C_2$ is the same constant as the one appearing in Theorem \ref{theo:mu}) such that for all $(u,H)$ satisfying \eqref{e:(u,H)}, the following assertions hold.
\begin{itemize}
\item[(i)]  For $\varepsilon$ small enough (depending on $d,N,u$ and $H$), we have 
$$\bar{g}_u(C_3\varepsilon^{-1})\le 1- \tilde{\mu}_{N,\varepsilon}(u)\le  1- \mu_{N,\varepsilon}(u) \le  g_u(C_2\varepsilon^{-1})\,.$$
\item[(ii)] For $\varepsilon$ small enough (depending on $d,N,u$ and $H$), for any $\delta>0$, a.s. for large $s$, there exists a geodesic $\tilde{\gamma}_{N,\varepsilon}(su) \in \hat \Pi (0,su)$ from $0$ to $su$ for the time $\tilde T_{N,\varepsilon} (0,su)$ such that
$$  N(\tilde{\gamma}_{N,\varepsilon}(su))-s\le (1+\delta) g_u(C_2\varepsilon^{-1})s.$$
\item[(iii)] Moreover, we have
$$\lim_{\varepsilon\to 0}\frac{ \mu_{N,\varepsilon}(u)-\tilde{\mu}_{N,\varepsilon}(u)}{g_u(C_2\varepsilon^{-1})}=0.$$
\end{itemize}
\end{theorem}
Let us emphasize the fact that we do not state a control on $\sharp \tilde{\gamma}_{N,\varepsilon}(su)$ in Theorem \ref{theo:FPP}: indeed $\tilde{\gamma}_{N,\varepsilon}(su) \in \hat \Pi (0,su)$ is a generalized path, that does not have to travel between points of $\Xi$, thus the quantity $\sharp \tilde{\gamma}_{N,\varepsilon}(su)$ does not have a relevant signification. However, we do obtain a control on the number of balls of the Boolean model really useful to a geodesic if we look at geodesics inside a set of paths that are easier to deal with: for more details, we refer to Proposition \ref{prop:geodes} in Section \ref{s:geodes} that states the existence of a geodesic $\check \gamma_\varepsilon (0,su)$ in a restrictive set $\check \Pi (0,su) \subset \Pi (0,su)$, and to Equation \eqref{e:ajoutfev} that gives an upper bound for $\sharp \check \gamma_\varepsilon (0,su)$ - the lower bound on $\sharp \check \gamma_\varepsilon (0,su)$ is a straigthforward consequence of the upper bound on $\tilde \mu_{N,\varepsilon}(u)$, as in the proof of Corollary \ref{c:lowersharp} in the model with rewards.

For specific choices of the norm $N$, namely when $N$ is the $p$-norm, it can be proved that the upper bound and lower bound appearing in Theorems \ref{theo:mu} $(i)$ and \ref{theo:FPP} $(i)$ are of the same order in $\varepsilon$ (see Section \ref{s:Np}). In this case, Assertion $(iii)$ in Theorem \ref{theo:FPP} implies that Theorem \ref{theo:FPP} $(i)$ is a consequence of Theorem \ref{theo:mu} $(i)$. However, it is not true in general, thus the different assertions in Theorem \ref{theo:FPP}  must be stated separately.


\subsection{Some properties of the geodesics $\tilde \gamma_\varepsilon (su)$}
\label{s:geodes}

Recall that for $x,y\in \R^d$, $\tilde T_\varepsilon(x,y)$ denotes the travel time between $x$ and $y$ in the first passage percolation model defined in Section \ref{s:def}
$$\tilde T_\varepsilon (x,y):=\inf_{\pi \in \hat\Pi(x,y)} \tilde T_\varepsilon (\pi) \,,$$
where, for a path $\pi$,
$$ \tilde T_{\varepsilon} (\pi) = N(\pi  \cap \Sigma_{ \varepsilon}^c ).$$
In this section, we prove that, for $\varepsilon$ small enough, geodesics and the time constant exist for this model and besides, we can impose some properties on the geodesics such that this model can be easily compared to the model with rewards.
We begin by proving that, as in the model with rewards, the infimum can in fact be taken only on polygonal paths.

\begin{prop}
\label{prop:geodespoly}
For every $x,y \in \mathbb{R}^d$, we have
$$ \tilde T_{\varepsilon} (x,y) \,=\, \inf_{\pi \in \hat \Pi (x,y)} \tilde T_{\varepsilon} (\pi ) \,=\,  \inf_{\pi \in  \Pi (x,y)} \tilde T_{\varepsilon} (\pi ).$$ \end{prop}

\begin{proof}
Fix $x,y\in \mathbb{R}^d$ and  consider a path $\pi \in \hat\Pi (x,y)$. We want to construct a path $\pi'\in \Pi (x,y)$ such that $\tilde T_\varepsilon(\pi')\le \tilde T_\varepsilon(\pi)$.

We associate with the path $\pi$ the sequence $(\mathcal C_1, \dots , \mathcal C_k)$ of the $k$ connected components of $\Sigma_\varepsilon$ that the curve $[\pi]$ intersects ranked by the order in which they appear in $[\pi]$ ($k\in \mathbb{N}$ may be null). If there exists $i_0 < i_1$ such $\mathcal C_{i_0} = \mathcal C _{i_1} :=\mathcal C$, we can replace $\pi$ by $\hat \pi$ defined as the concatenation of the three following  paths:
\begin{itemize}
\item the subpath $\pi_{1}$ of $\pi$ from $x$ to the first point $a$ in $[\pi] \cap \mathcal{C}$;
\item the subpath $\pi_{3}$ of $\pi$ from the last point $b$ in $[\pi] \cap \mathcal{C}$ to $y$;
\item between those subpaths, a  path $\pi_{2}$ between $a$ and $b$ satisfying $[\pi_{2}]\subset \mathcal{C}$, {\it i.e.}, that remains inside $\mathcal{C}$.
\end{itemize}
Since $\tilde T_\varepsilon (\pi_{2})=0$, we have $\tilde T_\varepsilon (\hat \pi) =\tilde T_\varepsilon (\hat \pi_{1}) +  \tilde T_\varepsilon (\hat \pi_{3}) \leq \tilde T_\varepsilon ( \pi) $, thus we can suppose that the connected components $(\mathcal C_1, \dots , \mathcal C_k)$ are distinct and  the path $\pi$ successively  enters in the distinct connected components  $(\mathcal C_1, \dots , \mathcal C_k)$ and never go back to one of them once it has left it.
For sake of clarity, we assume here that $x$ and $y$ does not belong to $\Sigma_\varepsilon$. We let the reader check that the proof below can be easily adapted if this condition does not hold.

\begin{figure}
\begin{center}
\includegraphics[width=10cm]{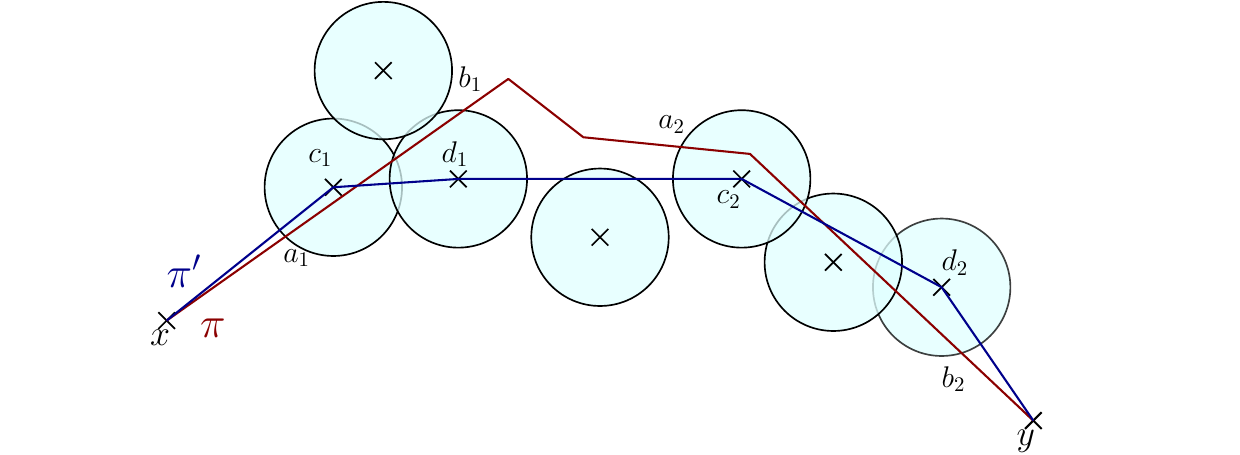}
\end{center}
\caption{The path $\pi$ (in red) is a generalized path from $x$ to $y$. In blue, a polygonal path $\pi'$ from $x$ to $y$ such that $\tilde T_\varepsilon(\pi')\le \tilde T_\varepsilon(\pi)$. } 
\label{Fig:geodesique}
\end{figure}

Let us denote by $(a_i, 1\le i \le k)$ the first point of $[\pi]$ in  $\mathcal{C}_i$ and 
$(b_i, 1\le i \le k)$ its last point (see Figure \ref{Fig:geodesique}). Then we have
$$\tilde T_\varepsilon(\pi)=\sum_{i=0}^k N(a_{i+1}-b_i)$$
with the convention $x=b_0$ and $y=a_{k+1}$. Let $c_i\in \Sigma\cap \mathcal{C}_i$ be such that $a_i\in \partial B_N(c_i,\varepsilon^{1/d}/2)$ and $d_i\in \Sigma\cap \mathcal{C}_i$ such that $b_i\in \partial B_N(c_i,\varepsilon^{1/d}/2)$. Let $\pi_i=(c_i,x_{i,1},\ldots,x_{i,n_i},d_i)$ be a polygonal path ({\it i.e.}, with vertices in $\Xi$) from $c_i$ to $d_i$ which remains in $\mathcal{C}_i$. Consider now the  path $\pi'$ from $x$ to $y$ which is the concatenation of the paths $(\pi_i,1\le i \le k)$, namely
$$\pi':=(x,c_1,x_{1,1},\ldots,x_{1,n_1},d_1,c_2,\ldots,d_k,y).$$
By construction $\pi'\in \Pi(x,y)$ and 
\begin{eqnarray}
\tilde T_\varepsilon(\pi')&=&N([x,c_{1}]\cap \Sigma_\varepsilon^c)+N([d_k,y]\cap \Sigma_\varepsilon^c)+\sum_{i=1}^{k-1} N([d_i,c_{i+1}]\cap \Sigma_\varepsilon^c) \nonumber \\
&\le & N(c_1-x)-\frac{\varepsilon^{1/d}}{2}+N(y-d_k)-\frac{\varepsilon^{1/d}}{2}+\sum_{i=1}^{k-1} \left (N(c_{i+1}-d_i)-\varepsilon^{1/d} \right).\label{e:egalite}
\end{eqnarray}
But, by triangle inequality, we have
$$N(c_{i+1}-d_i)\le N(b_i-d_i)+N(a_{i+1}-b_i)+N(c_{i+1}-a_{i+1})=\varepsilon^{1/d}+N(a_{i+1}-b_i) \,,$$
and 
$$N(c_1-x)\le N(a_1-x)+\frac{\varepsilon^{1/d}}{2} \qquad \textrm{and} \qquad N(y-d_k)\le N(y-b_k)+\frac{\varepsilon^{1/d}}{2}.$$
So we get that $\tilde T_\varepsilon(\pi')\le \tilde T_\varepsilon(\pi)$ as wanted.

\end{proof}
The next step is now to prove  the existence of geodesics and of a time constant for this model. However, the proofs are more classical in this case than for the model with rewards. Indeed, the random variables $(\tilde T_\varepsilon(x,y),x,y\in \R^d)$ are the travel times in a first passage percolation model, in particular, they are non negative, satisfy the triangle inequality and are bounded by $N(y-x)$. A classical use of Kingman's subadditive ergodic theorem enables then to prove the existence of the time constant: for every $z\in \mathbb{R}^d$, there exists a constant $ \tilde \mu_\varepsilon (z) \in [0,N(z)]$ such that
\begin{equation*}
\lim_{s \rightarrow \infty} \frac{ \tilde{T}_{\varepsilon} (0,sz)}{s} = \tilde \mu_{\varepsilon} (z) \quad \textrm{ a.s. and in }L^1.
\end{equation*}
Besides, the Euclidean case $N:=\|\cdot\|_2$,
has been studied in \cite{GT17} where a condition is given to ensure that $\tilde \mu_\varepsilon(\cdot):=\tilde\mu_{2,\varepsilon}(\cdot)$ is strictly positive in terms of a percolation event for the Boolean model $\Sigma_{\|\cdot\|_2,\varepsilon}$. This result implies in particular that $\tilde\mu_{2,\varepsilon}(\cdot)$ is a norm for small enough $\varepsilon$. Using the equivalence of the norm on $\R^d$, we have that $\Sigma_{N,\varepsilon}\subset \Sigma_{\|\cdot\|_2,c\varepsilon}$ for some $c>0$ so it is  easy to deduce that, in fact, for any norm $N$, the function $\tilde\mu_{N,\varepsilon}(\cdot)$ is strictly positive for small enough $\varepsilon$.\footnote{ It should be possible to characterize exactly the set of $\varepsilon$ such that $\tilde\mu_{N,\varepsilon}(\cdot)$ is a norm in terms of an event of percolation for the Boolean model $\Sigma_{N,\varepsilon}$ similar to the one given in \cite{GT17}.}
When $\tilde \mu_{\varepsilon}:=\tilde\mu_{N,\varepsilon}(\cdot)$ is a norm, a shape theorem also holds:
$$\mbox{for any $\delta>0$, for $t$ large enough } \quad \{z\in\R^d, \tilde T_\varepsilon(0,z)\le t\} \subset (1+\delta)B_{ \tilde \mu_{\varepsilon}} (1) \mbox{ a.s.}$$
where $B_{ \tilde \mu_{\varepsilon}} (1)$ denotes the unit ball for the norm $\tilde \mu_{\varepsilon}$.
This implies in particular that for any $n\in \N$ and any $x,y\in B_N(n)$, there exists a.s. some $K>0$ such that
$$\tilde T_\varepsilon (x,y):=\inf \{\tilde T_\varepsilon (\pi) \,:\, \pi \in \Pi(x,y)\}=\min \{\tilde T_\varepsilon (\pi) \,:\, \pi \in \Pi(x,y), \pi \subset B_N(K)\} \,,$$
since the last set of paths is a.s. finite. This gives the existence, for any $x,y\in \R^d$, of a geodesic $\gamma_\varepsilon(x,y)\in \Pi(x,y)$, \emph{i.e.}, a polygonal path such that
$$\tilde T_\varepsilon (x,y)=\tilde T_\varepsilon (\gamma_\varepsilon(x,y)).$$
Of course, such geodesic is not unique since we can modify it inside any connected component of $\Sigma_\varepsilon$.
To compare the model with rewards and the first passage percolation model, we will need in fact more properties on the geodesics we consider. So we introduce here a third set of paths, namely $\check{\Pi}_\varepsilon (x,y)$, and prove that the geodesics can be chosen inside this set. For a path $\pi = (x=x_0, x_1, \dots , x_n=y)$ from $x$ to $y$, let us denote by $(\mathcal{C}_1,\ldots,\mathcal{C}_k)$ the connected component of $\Sigma_\varepsilon$ that the curve $[\pi]$ intersects. 
Define $\check{\Pi}_\varepsilon (x,y)$ as the set of paths  from $x$ to $y$ such that
 \begin{itemize}
\item[(i)] $\pi\in \Pi(x,y)$.
\item[(ii)] When $[\pi]$ exits a connected component of $\Sigma_\varepsilon$, it never re-enters again.
\item[(iii)] for every connected component $\mathcal C_i$, the subpath $[\pi] \cap \mathcal C_i$ can be written as $\pi_i = (a_i,x_{i,0}, \dots , x_{i,m_i}, b_i)$ where $m_i\ge 0$, $a_i$ (resp. $b_i$) is in the boundary of $\mathcal C_i$ and $B_N (x_{i,0}, \varepsilon^{1/d}/2)$ (resp. $B_N (x_{i,m_i}, \varepsilon^{1/d}/2)$) and for every $j\in \{ 0 ,\dots , m_i-1\}$, $N ( x_{i, j + 1} - x_{i,  j}) \leq \varepsilon^{1/d}$.
\end{itemize}
We emphasize that we impose that $m_i\ge 0$ in the third point. This means that each time the path enters a connected component  $\mathcal{C}$ of $\Sigma_\varepsilon$, it gets at least  trough one point of $\Xi\cap \mathcal{C}$.

\begin{prop}
\label{prop:geodes} For $\varepsilon$ small enough there exists a.s.,
for every $x,y \in \mathbb{R}^d$, a path $\check\gamma_\varepsilon(x,y)\in \check \Pi_{\varepsilon} (x,y)$ such that 
$$ \tilde T_{\varepsilon} (x,y)  = \tilde T_{\varepsilon}(\check\gamma_\varepsilon(x,y))\,.$$
\end{prop}

\begin{proof}
Consider $\pi\in\hat \Pi(x,y)$ a geodesic between $x$ and $y$. We have already noticed that such path necessarily satisfies Point (ii) of the definition of $\check \Pi_{\varepsilon} (x,y)$. We construct now of new path $\pi'$ between $x$ and $y$ in the exact same way than in proof of Proposition \ref{prop:geodespoly} except that we impose moreover that the polygonal subpath $\pi_i=(c_i,x_{i,1},\ldots,x_{i,n_i},d_i)$ created inside $\mathcal{C}_i$ between $c_i$ and $d_i$ satisfies $N ( x_{i, j + 1} - x_{i,  j}) \leq \varepsilon^{1/d}$ (such sequence of points necessarily exists since $\mathcal{C}_i$ is a connected component of balls of diameter $\varepsilon^{1/d}$). 
Besides, since $\pi$ is a geodesic, we necessarily have $\tilde T_\varepsilon(\pi)=\tilde T_\varepsilon(\pi')$ so in particular, Inequality \eqref{e:egalite} must be an equality. This implies that  $[\pi']$ cannot intersect other connected components of $\Sigma_\varepsilon$ than the one already intersected by $\pi$. So we deduce that each time $\pi'$ enters a connected component $\mathcal{C}$ of $\Sigma_\varepsilon$, it gets at least  trough one point of $\Xi\cap \mathcal{C}$.
Hence, $\pi'$ is indeed a path of $\check \Pi_{\varepsilon} (x,y)$.
\end{proof}


\subsection{Comparison between $\tilde{\mu}_\epsilon(u)$ and $\mu_\epsilon(u)$}
\label{s:comp}

The comparison between $\tilde{\mu}_\epsilon(u)$ and $\mu_\epsilon(u)$ is made in two steps. The first and easy step is to notice that, by looking at nice geodesics for the model with balls, we can prove that $\tilde T_\varepsilon (0,x) \geq T_\varepsilon (0,x)$ for every $x$, thus $\tilde{\mu}_\epsilon(u) \leq \mu_\epsilon(u)$. In the second step, we have to prove that $\tilde T_\varepsilon (0,x) -  T_\varepsilon (0,x)$ is not too big. When a path $\pi = (0,x_1,\dots , x_n)$ goes through points $(x_i)$ of $\Xi$ that are at $N$-distance bigger than $\varepsilon^{1/d}$, the corresponding balls of the Boolean model do not overlap, thus $\tilde T_\varepsilon (\pi) \leq T_\varepsilon (\pi)$. The travel time $\tilde T_\varepsilon (\pi)$ may be larger than $T_\varepsilon (\pi)$ only if $\pi$ goes through balls of the Boolean model that overlap. The difference $\tilde T_\varepsilon (\pi) -  T_\varepsilon (\pi)$ can thus be controlled for a path $\pi= (0,x_1,\dots , x_n)$ by obtaining an upper bound on the numbers of couples of points among the $(x_i)$ that are at $N$-distance less than $\varepsilon^{1/d}$ (see Inequality \eqref{e:ajout}). This upper bound is proved through Lemmas \ref{lem:BversA} and \ref{lem:majA}, using ideas that are quite similar to the ones used in Section \ref{s:chernov}.

By Proposition \ref{prop:geodes}, for all $x\in \R^d$,we know that there exists a geodesic $\tilde \gamma_\varepsilon (x)=(0=x_0,\ldots,x_{n}=x) \in \check{\Pi}_\varepsilon (0,x)$ for $\tilde T_\varepsilon (0,x)$ such that, for all $0< i< n$, $x_i\in \Xi$ and if $x_i$ and $x_{i+1}$ are in the same connected component of $\Sigma_\varepsilon$, $N(x_{i+1}-x_i)\le \varepsilon^{1/d}$. 
Let $(\mathcal C_1,\dots , \mathcal C_k)$ be the distinct connected components of $\Sigma_\varepsilon$ that the curve $[\tilde \gamma_\varepsilon (x)]$ intersects, and write $[\tilde \gamma_\varepsilon (x)] \cap \mathcal C_i =  (a_i,x_{i,0}, \dots , x_{i,m_i}, b_i)$ as in property $(iii)$ of the definition of $\check{\Pi}_{\varepsilon} (0,x)$. Then
$$ N(\tilde \gamma_\varepsilon (x)\cap \Sigma_\varepsilon)  = \sum_{i=1}^k N(\tilde \gamma_\varepsilon (x)\cap \mathcal C_i)$$
and
$$ N(\tilde \gamma_\varepsilon (x)\cap \mathcal C_i) = \frac{\varepsilon^{1/d}}{2} + \sum_{j=0}^{m_i - 1} N (x_{i,j+1} - x_{i,j}) +   \frac{\varepsilon^{1/d}}{2} \leq (m_i+1) \varepsilon^{1/d} \,.$$
We get that 
\begin{equation}\label{e:pointFPP}
N(\tilde \gamma_\varepsilon (x)\cap \Sigma_\varepsilon)\le \varepsilon^{1/d} \sharp \tilde\gamma_\varepsilon (x) 
\end{equation}
thus
$$\tilde{T}_\varepsilon(0,x)\ge T_\varepsilon(0,x).$$
This already yields that $\tilde{\mu}_\epsilon(u)\ge \mu_\epsilon(u)$. 

Let's now find an upper bound for $\tilde{\mu}_\epsilon(u)-\mu_\epsilon(u)$.
For any polygonal path $\pi=(x_0,\ldots,x_{q+1})$ with  $x_i\in \Xi$ for $1\le i \le q$, we define 
$$Y(\pi):=\{j\in \llbracket 1 \,:\, q+1 \rrbracket, \exists i<j, N(x_j-x_i)\le \varepsilon^{1/d}\} \,.$$
Note that the balls $B_N(x_j,\varepsilon^{1/d}/2)$, $j \notin Y(\pi)$ are pairwise disjoint. By discarding, if necessary, the first and last of these balls, we obtain that
$$\tilde{T}(\pi)\le N(\pi)-\varepsilon^{1/d}(q-1-\card (Y(\pi)))\leq T(\pi)+\varepsilon^{1/d}(2+\card (Y(\pi))).$$
Applying this inequality to a geodesic  $\gamma_\varepsilon(su)$ from $0$ to $su$  in the model with rewards studied  previously, we get
\begin{equation}
\label{e:ajout}
\tilde{T}(\gamma_\varepsilon(su))\le T(\gamma_\varepsilon(su))+\varepsilon^{1/d}(2+ \card ( Y(\gamma_\varepsilon(su))) ).
\end{equation}
We obtain an upper bound for $\card (Y(\gamma_\varepsilon(su)))$ by proving the following two lemmas.

\begin{lemma}\label{lem:BversA} For a path $\pi=(x_0,\ldots,x_{q+1})$ let us define $$Z(\pi):=\{j\in \llbracket 1 \,:\, q+1 \rrbracket,  N(x_j-x_{j-1})\le 5\varepsilon^{1/d}\}.$$
Then we have a.s., for any geodesic $\gamma$ for the model with rewards,  $\card (Y(\gamma)) \le 2\card (Z(\gamma))$.
\end{lemma}

\begin{lemma}\label{lem:majA} Let $C_2$ be such that Theorem \ref{theo:mu} holds. Fix $(u,H$) satisfying \eqref{e:(u,H)}. Then, for any $\lambda>0$, there exists $\varepsilon_0>0$ such that for $\varepsilon<\varepsilon_0$, we have, a.s. for $s$ large enough, for any geodesic $\gamma_\varepsilon (su)$ from $0$ to $su$ for the model with rewards,
$$\card ( Z(\gamma_\varepsilon(su)) ) \le \lambda g_u(C_2\varepsilon^{-1})\varepsilon^{-1/d}s.$$
\end{lemma}

Assuming  these lemmas hold, we get that for any $\lambda>0$, for $\varepsilon$ small enough we have a.s. for $s$ large enough,
$$\tilde{T}(\gamma_\varepsilon(su))\le T(\gamma_\varepsilon(su))+(2 \varepsilon^{1/d}+2\lambda g_u(C_2\varepsilon^{-1})s)$$
which implies that 
$$\limsup_{\varepsilon\to 0} \frac{\tilde{\mu}_\epsilon(u)-\mu_\epsilon(u)}{g_u(C_2\varepsilon^{-1})}\le 2\lambda$$
and so Theorem \ref{theo:FPP} $(iii)$ holds.

\begin{proof}[Proof of Lemma \ref{lem:BversA}]
To prove this inequality between the cardinal of $Y(\pi)$ and $Z(\pi)$, we in fact prove that 
 $$\card  (Y(\pi)\setminus Z(\pi) )\le \card ( Z(\pi) ).$$
 Let us first note that for a path $\pi=(0,x_1,\ldots, x_q)$ with $x_i\in \Xi$ for $1 \leq i \leq q$, we have a.s. (since $0 \notin \Xi$ a.s.) $$T_\varepsilon(\pi)=\sum_{i=1}^q \left( N(x_i-x_{i-1})-\varepsilon^{1/d}\right) \,.$$
Thus if $T_\varepsilon(\pi)\le -K\varepsilon^{1/d}$ it implies in particular that $\card (Z(\pi))\ge K$.
Let now $\gamma=(0,x_1,\ldots,x_q,x)$ be a geodesic from $0$ to $x$ for $T_\varepsilon$ and let $j(1)<\ldots<j(l)$ be the indices such that 
$$Y(\gamma)\setminus Z(\gamma)=\{x_{j(1)},\ldots,x_{j(l)}\}.$$
Let us define, for $i\le q$, the truncated path
$\gamma_i=(0,x_1,\ldots,x_i)$.
Let us prove by induction on $i$ that 
\begin{equation}\label{eq:recZ}
\card ( Z(\gamma_{j(i)}))\ge i
\end{equation}
which would imply that
$$\card ( Z(\gamma))\ge \card (Z(\gamma_{j(l)}))\ge l=\card (Y(\gamma)\setminus Z(\gamma) )\,.$$
We have by definition $x_{j(1)}\in Y(\gamma)\setminus Z(\gamma)$, \emph{i.e.}, there exists some $k<j(1)$ such that $N(x_{j(1)}-x_k)\le \varepsilon^{1/d}$ but $N(x_{j(1)}-x_{j(1)-1})\ge 5\varepsilon^{1/d}$.
Using that 
$$\varepsilon^{1/d}\ge N(x_k-x_{j(1)})\ge T_\varepsilon(x_k,x_{j(1)})=T_\varepsilon(x_k,x_{j(1)-1})+T_\varepsilon(x_{j(1)-1},x_{j(1)})\ge T_\varepsilon(x_k,x_{j(1)-1})+3\varepsilon^{1/d} $$
we get that 
$$T_\varepsilon(x_k,x_{j(1)-1})\le -2\varepsilon^{1/d}$$
which implies in particular that 
$$\card (Z(\gamma_{j(1)}))\ge \card ( Z((x_k,\ldots,x_{j(1)-1})))\ge 1.$$
Fix $i$, and assume now that for every $m<i$, we have $\card ( Z(\gamma_{j(m)}))\ge m$. Let $k<j(i)$ be such that $N(x_{j(i)}-x_k)\le \varepsilon^{1/d}$. Let $m<i$ be such that $j(m)<k\le j(m+1)$ with the convention $j(0)=0$. We have
$$\varepsilon^{1/d}\ge T_\varepsilon(x_k,x_{j(i)})=T_\varepsilon(x_k,x_{j(m+1)-1})
+\sum_{p=m+1}^{i}T_\varepsilon(x_{j(p)-1},x_{j(p)})+ \sum_{p=m+1}^{i-1}T_\varepsilon(x_{j(p)},x_{j(p+1)-1}) \,. $$
By definition of the indexes $j(p)$, we have $T_\varepsilon(x_{j(p)-1},x_{j(p)})\ge 3\varepsilon^{1/d}$ so we get
$$ T_\varepsilon(x_k,x_{j(m+1)-1})+
 \sum_{p=m+1}^{i-1}T_\varepsilon(x_{j(p)},x_{j(p+1)-1}) \le \varepsilon^{1/d}(1-3(i-m)) \le -2(i-m)\varepsilon^{1/d} \,.$$
 This implies in particular that 
 $$ \card ( Z((x_k,\ldots,x_{j(i)})))\ge i-m$$
 for some $k\ge j(m)$.
But, by hypothesis, we have
 $$\card ( Z((0,\ldots,x_{k-1})))\ge \card ( Z((0,\ldots,x_{j(m)})))\ge m$$ 
 so we get 
 $$\card ( Z(\gamma_{j(i)}))\ge i \,.$$

\end{proof}

\begin{proof}[Proof of Lemma \ref{lem:majA}] Fix some $C_2$ such that Theorem \ref{theo:mu} holds. Fix some $\lambda>0$.
 Setting $\eta_\infty:=g_u(C_2\varepsilon^{-1})$, we want to prove that  for $\varepsilon$ small enough, we have, a.s. for $s$ large enough, for any geodesic $\gamma_\varepsilon (su)$ from $0$ to $su$ for the model with rewards,
 $$\card (Z(\gamma_\varepsilon(su)) ) \le \lambda \eta_\infty\varepsilon^{-1/d}s.$$
Recall that we have a.s. for $s$ large enough
$N(\gamma_\varepsilon(su))\le s(1+2\eta_\infty)$ and  $\sharp \gamma_\varepsilon(su)\le 2\eta_\infty\varepsilon^{-1/d}s$. Thus, it is sufficient to show that for $s$ large enough, for any path $\pi$ from $0$ to $su+H$ such that $N(\pi)\le s(1+2\eta_\infty)$ and $\sharp \pi\le 2\eta_\infty\varepsilon^{-1/d}s$ , we have $$\sharp Z(\pi)\le \lambda \eta_\infty\varepsilon^{-1/d}s.$$
Let $q_0:=\lambda \eta_\infty\varepsilon^{-1/d}s$, $q_1:=  2\eta_\infty\varepsilon^{-1/d}s$,  $\eta:=2\eta_\infty$ and set
$$\mathcal{H}(s):=\{\exists \pi \in \Pi (0, su+H),    N (\pi)\le (1+\eta )s, \sharp \pi \le q_1, \; \card ( Z(\pi) )\ge q_0 \}.$$
Proving that $\mathcal{H}(s)$ does not occur a.s. for $s$ large enough will yield Lemma \ref{lem:majA}. To bound $\P(\mathcal{H}(s))$, we roughly use the same idea as in Section \ref{s:chernov}. For a sequence of points $(x_i)_{1\le i \le q}$ in $ \Xi$, let us denote by $(t_i)_{1\le i \le q}$ their increments, \emph{i.e.}, $t_i=x_i-x_{i-1}$ (with the convention $t_1 = x_1$).
 Then, we have
\begin{equation*}
\mathbb{P} [\mathcal H (s) ] =\sum_{q\le q_1}\sum_{Z\subset \{1,\ldots,q\}, |Z|\ge q_0} \P(\mathcal{A}_Z)
\end{equation*}  
 where 
 $$\mathcal{A}_Z:=\{\exists (x_i)_{1\le i \le q}\in \Xi^q, x_{q}\cdot u^\star\le s,  \sum_{i=1}^q N(t_i)\le \eta s + x_{q}\cdot u^\star, \forall j \in Z, t_j\cdot u^\star\le 5\varepsilon^{1/d}\}.$$
 Using Chernov inequality, we have for all $\alpha, \beta >0$,
\begin{align*}
\mathbb{P} &  [\mathcal A_Z ] \\ &\le  \int_{(\R^d)^q} \exp\left(\beta(s-x_{q}\cdot u^\star)+\alpha(\eta s +x_{q}\cdot u^\star - \sum_{i=1}^{q} N(t_i))\right)\prod_{j\in Z} \exp\left(\frac{\alpha}{2}(4\varepsilon^{1/d}-(t_j\cdot u^\star))\right)
dt_1\ldots dt_q \,.
\end{align*} 
Taking $\beta=\alpha \eta$, we get
\begin{eqnarray*}
\mathbb{P} [\mathcal A_Z  ] &\le &   \exp(\alpha (2\eta s+2\varepsilon^{1/d} \card (Z))\int_{(\R^d)^q}\exp\left(-\alpha \sum_{i=1}^{q} \left[ N(t_i)- \left(1-\eta-\frac{1_{j\in Z}}{2} \right) (t_i\cdot u^\star) \right]\right)dt_1\ldots dt_q\\
&=&   \exp(2\alpha (\eta s+\varepsilon^{1/d}\card (Z))\left(\frac{I(\eta)}{\alpha^d}\right)^{q-\card (Z)}\left(\frac{I(\frac{1}{2}+\eta)}{\alpha^d}\right)^{\card (Z)}\\
&=&   \exp(2\alpha (\eta s+\varepsilon^{1/d}\card (Z))\left(\frac{I(\eta)^{1-\card (Z)/q}I(\frac{1}{2}+\eta)^{\card (Z)/q}}{\alpha^d}\right)^q.
\end{eqnarray*}

We know that $\lim_{\eta \to 0} I(\eta)=\lim_{\eta \to 0} h_u(\eta)=+\infty$ whereas $I(\eta)$ is  bounded on $[1/2,1]$. Since $\eta_\infty$ tends to 0 as $\varepsilon$ tends to 0, we deduce that for $\varepsilon$ small enough, $I(\eta)\ge I(\frac{1}{2}+\eta)$ and so we get that for $\card (Z)\ge q_0$ and $q\le q_1$,
\begin{align*}
I(\eta)^{1-\card (Z)/q}I\left(\frac{1}{2}+\eta \right)^{\card (Z)/q} & \le I(\eta)^{1-q_0/q_1}I\left(\frac{1}{2}+\eta \right)^{q_0/q_1}\\ &  =I(\eta)^{1-\lambda/2}I\left(\frac{1}{2}+\eta\right)^{\lambda/2}\\ & \le c_1 I(2\eta_\infty)^{1-\lambda/2}
\end{align*}
for $\lambda\in (0,1)$ and $c_1:=\sup\{I(x),x\in[1/2,1]\}\vee 1$.
Using \eqref{eq:I_1} in Lemma \ref{l:I+}, which states that $I$ is of the same order than $h_u$ near 0, we also get, for some constant $c_2$, for $\varepsilon$ small enough,
$$I(2\eta_\infty)\le c_2 h_u(2\eta_\infty) \le  c_2h_u(\eta_\infty)=c_2 C_2\varepsilon^{-1}.$$
So we get, for some constant $c_3>0$,
\begin{eqnarray*}
\mathbb{P} [\mathcal H (s) ] 
&\le& \sum_{q_0\le q\le q_1}2^q  \exp(2\alpha (2\eta_\infty s+\varepsilon^{1/d}q_1)\left(\frac{c_3 \varepsilon^{-1+\lambda/2}}{\alpha^d}\right)^{q}.\\
&\le& \sum_{q\ge q_0}  \exp(8\alpha \eta_\infty s)\left(\frac{2c_3 \varepsilon^{-1+\lambda/2}}{\alpha^d}\right)^{q}
\end{eqnarray*}
Taking $\alpha^d=4c_3\varepsilon^{-1+\lambda/2}$, \emph{i.e.}, $\alpha=c_4\varepsilon^{-(1/d)+\lambda/(2d)}$, we obtain
\begin{eqnarray*}
\mathbb{P} [\mathcal H (s) ] 
&\le& 2  \exp( 8c_4\varepsilon^{-(1/d)+\lambda/(2d)}\eta_\infty s- q_0\log 2)\\
&=& 2  \exp( -s\varepsilon^{-1/d}\eta_\infty(\lambda\log 2-8c_4\varepsilon^{\lambda/(2d)})).
\end{eqnarray*}
Hence, for any $\lambda \in (0,1)$, if $\varepsilon$ is small enough, this quantity decreases to 0 exponentially fast in $s$. Taking 
$$s_n:=n\varepsilon^{1/d}/(\lambda \eta_\infty)$$
Borel-Cantelli Lemma yields that there exists a.s. $n_0$ such that, for $n\ge n_0$,  $\mathcal{H}(s_n)$ does not occur.  Since for $s\in [s_n,s_{n+1}[$, $\lfloor \lambda \eta_\infty \varepsilon^{-1/d}s\rfloor = \lfloor \lambda \eta_\infty \varepsilon^{-1/d}s_n\rfloor$, we deduce using the same argument as in proof of Proposition \ref{prop:Chernov}, that in fact $\mathcal{H}(s)$ does not occur for $s$ large enough.

\end{proof}

\subsection{Proof of the lower bound on $1- \tilde{\mu}_\epsilon(u)$}
\label{s:lowbis}

When $1- \mu_\epsilon(u)$ is of order $g_u(C_2\varepsilon^{-1})$ when $\varepsilon$ goes to $0$, it is a straightforward consequence of Theorem \ref{theo:FPP} $(iii)$ that the same holds for $1- \tilde{\mu}_\epsilon(u)$. It will be the case for specific choices of the norm $N$, namely the $p$-norms, see Section \ref{s:Np}. However, in general setting, the lower bound on $1- \mu_\epsilon(u)$ given by Theorem \ref{theo:mu} is $\bar{g}_u(C_1\varepsilon^{-1})$, and there is no way to deduce the same lower bound for $1- \tilde \mu_\epsilon(u)$ through Theorem \ref{theo:FPP} $(iii)$. It is however easy to adapt the greedy algorithm used to prove the lower bound on $1- \mu_\epsilon(u)$ in Section \ref{s:glouton} to get the same type of lower bound on $1- \tilde \mu_\epsilon(u)$. The idea is to add an extra space between two consecutive points of the process $\Xi$ that are collected by the greedy algorithm, to make sure that the corresponding balls of diameter $\varepsilon^{1/d}$ in the Boolean model do not intersect. Let us get into details.

Fix $(u,H)$ satisfying \eqref{e:(u,H)} and recall that
$$M_\eta:= \{ v \in H \,:\, N(u+v)  \leq 1+ \eta \mbox{ and } N(u-v)  \leq 1+ \eta \}  \qquad \textrm{and} \qquad \bar h_u (\eta) := \eta^{-d} | M_\eta (u)| \,.$$
We state the following analog of Proposition \ref{t:thm8}.

\begin{prop}
\label{prop:gloutonespace}
There exists a constant $c'>0$ (depending only on $d,N$) such that, for all $(u,H)$ satisfying \eqref{e:(u,H)}, for all $\varepsilon >0$, for all $\eta\in (0,1)$, we have
$$ \tilde \mu_\varepsilon (u) \leq 1 + \eta - c' \varepsilon^{1/d}  |M_\eta |^{1/d} \,. $$
In particular, with $C_1' := \left(\dfrac{2}{c'}\right)^d$, we have that for $\varepsilon$ small enough (depending on $d,N,u$ and $H$),
$$\tilde \mu_\varepsilon (u) \leq 1 - \bar g_u (C'_1 \varepsilon^{-1}) \,.$$
\end{prop}

\begin{proof}
We use almost the same greedy algorithm as in the proof of Proposition \ref{t:thm8} except that we impose also that two consecutive points taken by the greedy path must be at distance at least $1$. 
Recall that for $s\geq0$ 
\begin{align*}
C_\eta (s) & :=  \{ x=\lambda u +v\, : \, 0 \leq \lambda \leq s \,,\, v \in H \,,\,N(\lambda u + v) \leq \lambda (1+\eta), N(\lambda u - v) \leq \lambda (1+\eta) \}\,,
\end{align*}
and we have
$$   |C_\eta (s) | = \frac{1}{d \| u^\star\|_2}  | M_\eta | s^d \,.$$
Recall that $\Xi$ denotes the points of a Poisson point process with unit intensity  on $\mathbb{R}^d$. 
By induction we define a sequence $(X_n)_{n\ge 0}$ in $\R^d$ with $X_0=0$
and  if
$$ \lambda_{n+1} := \inf \{ s>1 \,:\, (X_n + C_\eta (s)) \cap \Xi \neq \emptyset \} $$
(we emphasize the fact that we require that $s>1$), we set $X_{n+1} = X_n +\lambda_{n+1} u + W_{n+1}$ with  $W_{n+1} \in H$ and such that $X_{n+1}= (X_n+ C_\eta (\lambda_{n+1})) \cap \Xi$. We define $S_n = \sum_{i=1}^n \lambda_i $ and $V_n=\sum_{i=1}^n W_i$. By the properties of Poisson point processes, $(\lambda_i)_{i\ge 1}$ and $(W_i)_{i\ge 1}$ are i.i.d. random variables and since $M_\eta (u)$ is a symmetric set, $W_1$ has also a symmetric distribution and so $V_n$ is a symmetric random walk.

Let us now define, for $s>0$,
$$ Z(s) := \sup \{ n\geq 0 \,:\, S_n \leq s \} $$
and consider the path $\pi_s = (0,X_1,\dots, X_{Z(s)}, su)$. Let us find an upper bound of $\tilde T_\varepsilon (\pi_s)$ by finding an upper bound of its length and a lower bound of $Z(s)$, its number of points.

The sequence $(\lambda_i)_{i\geq 1}$ is  i.i.d.  and we have for $t \ge 1$,
\begin{align*}
\mathbb{P} (\lambda_{1} \geq t  ) & =  \exp \left(- \frac{1}{d \| u^\star\|_2}| M_\eta  (u)|  (t^d-1) \right)
\end{align*}
so that
\begin{align*}
 \mathbb{E} (\lambda_1) & =1+\int_1^\infty \exp  \left(- \frac{1}{d \| u^\star\|_2}| M_\eta  (u)|  (t^d-1) \right) du \\ 
 & =
  1+ \left( \frac{ | M_\eta (u)|}{d \| u^\star\|_2}  \right)^{-1/d} \int_{( | M_\eta (u) | /\| u^\star\|_2 )^{1/d}}^\infty \exp (-v^d) dv := 1 +  \frac{| M_\eta (u)|^{-1/d} }{2 c'_u}
  \end{align*}
where 
$$ 2 c'_u :=   \left( \frac{1}{d \| u^\star\|_2}  \right)^{1/d}  \left( \int_{( | M_\eta (u) | /\| u^\star\|_2 )^{1/d}}^\infty \exp (-v^d) dv \right)^{-1} \,.$$
Using that for any $\eta\in (0,1)$, $M_\eta (u) \subset B_N(0,3)$ we see that the  $(d-1)$-dimensional volume $|M_\eta (u)|$ of $M_\eta (u)$ is bounded by some constant $c_1'$ (depending only on $d$ and $N$) for $\eta\in (0,1)$. By Lemma \ref{lem:CDV}, we also know that $c\leq \| u^\star\|_2 \leq c'$ for constants $c,c' \in (0,\infty)$ (depending only on $d$ and $N$). We deduce that there exist some constants $c_2,c_3>0$ (depending only on $d$ and $N$ also) such that for $\eta\in (0,1)$,
$$ c_2| M_\eta | ^{-1/d}\le \mathbb{E} (\lambda_1) \le c_3| M_\eta | ^{-1/d} \,.$$
Since $S_n = \sum_{i=1}^n \lambda_i $,  a standard renewal theorem yields that, setting $c_4=1/(2c_3)$ and $c_5=2/c_2$,  for $s$ large enough, we have a.s.
\begin{equation}
\label{e:ajout2}
 c_4 |M_\eta |^{1/d} s \leq Z(s) \leq c_5 |M_\eta |^{1/d} s \,.
 \end{equation}
This gives a lower bound for the number of points in $\pi_s$. The rest of the proof of Proposition \ref{prop:gloutonespace} is a copy of the proof of Proposition \ref{t:thm8} with Equation \eqref{e:ajout2} instead of \eqref{e:ajout3}.
\end{proof}

\subsection{Upper bound on $N (\tilde \gamma_{\varepsilon} (su)) $}
\label{s:upbis}

To complete the proof of Theorem \ref{theo:FPP}, it remains to prove $(ii)$,{\it i.e.}, the existence of a geodesic $\tilde \gamma_{\varepsilon} (su) \in \hat \Pi (0,su)$ from $0$ to $su$ for the time $\tilde T_\varepsilon (0,su)$ such that $N(\tilde \gamma_{\varepsilon} (su) )$ is upper bounded. In fact, we will choose our geodesic in $ \check \Pi_\varepsilon (0,su)$. By Proposition \ref{prop:geodes}, there exists $\check \gamma_{\varepsilon} (su) \in \check \Pi_\varepsilon (0,su)$ that is a geodesic from $0$ to $su$ for $\tilde T_\varepsilon (0,su)$.
Since $\check \gamma_\varepsilon (su) \in \check \Pi_\varepsilon (0,su)$,  Equation \eqref{e:pointFPP}  states that
$$
N(\check \gamma_\varepsilon (su)\cap \Sigma_\varepsilon)\le \varepsilon^{1/d} \sharp \check\gamma_\varepsilon (su) $$
which gives that 
$$
N(\check \gamma_\varepsilon (su))\le \tilde T(\check \gamma_\varepsilon (su))+  \varepsilon^{1/d} \sharp \check\gamma_\varepsilon (su)\le s+\varepsilon^{1/d}  \sharp \check\gamma_\varepsilon (su).$$
One can check that the proof of \eqref{eq:upperboundgeodesique} (which gives an upper bound for the length and the number of points of the geodesic in the model with rewards) only uses a similar inequality between $N(\gamma_\varepsilon (su))$ and  $\sharp \gamma_\varepsilon (su)$. Mimicking this proof, we can so also obtain  that, for any $\delta>1$,
\begin{equation}
\label{e:ajoutfev}
N(\check\gamma_\varepsilon (su))\le s(1+\delta g_u(C_2\varepsilon^{-1})) \textrm{ and } \sharp\check\gamma_\varepsilon (su) \leq \delta g_u(C_2\varepsilon^{-1}) \varepsilon^{-1/d} s. 
\end{equation}


\section{$p$-norm}
\label{s:Np}

In this section, we suppose that $N$ is the $p$-norm for some $p\in [1,\infty]$. We put a subscript $p$ in our notations to emphasize the dependence on $p$. For a given $u\in \mathbb{R}^d$, we recall that the definition of $d_i(u)$, $i\in \{1,\dots, 4\}$ is given in \eqref{e:ajout4}. For a given $p\in [1,\infty]$ the definition of $\kappa_p(u)$ is given in \eqref{e:ajout5}.

This section is organized as follows. First in Section \ref{s:H} we specify for each $p\in [1,\infty]$ and each $u\in \mathbb{R}^d$ such that $\| u \|_p =1$ what hyperplane $u+H$ we consider as a supporting hyperplane of $B_p (1)$ at $u$, and fix some notations. In this setting, we evaluate $|K_\eta (u)|$ in Section \ref{s:evalK}. This must be done separately for $p \in (1,\infty)$, $p=1$ and $p=+\infty$, since the estimations are based on a proper rescaling of $K_\eta (u)$ which is not exactly of the same nature in those different situations. Then in Section \ref{s:evalM} we compare $|M_\eta (u)|$ to $|K_\eta (u)|$. This allows us to apply the results proved for a general norm $N$ to the $p$-norm in Section \ref{s:appli} to prove Theorem \ref{theo:tilde_mu_p} and the analog for the model with rewards, namely Theorem  \ref{theo:mu_p} . Finally in Section \ref{s:monotonie} we prove Theorem \ref{theo:monotonie}, {\it i.e.}, the existence of $\lim_{\varepsilon \rightarrow 0} (1- \tilde \mu_\varepsilon (u))/ \varepsilon^{\kappa_p(u)}$, by an argument of monotonicity, at least for some $p$ and $u$.

\subsection{Choice and description of the supporting hyperplane $u+H$}
\label{s:H}

We have to deal separately with the cases $p\in (1,\infty)$, $p=1$ and $p=\infty$. From now on, until the end of Section \ref{s:Np}, we always consider that the hyperplane $H$ we work with is the one described here.

\paragraph{Case $p\in (1,\infty)$.} Consider $N=\| \cdot \|_p$ for some $p\in  (1,\infty)$. Let $u\in \mathbb{R}^d$ be such that $\| u \|_p=1$. Let $d_1 := d_1 (u) $ be the number of coordinates of $u$ that are not null. By invariance of the model under symmetries along the hyperplanes of coordinates and by permutations of coordinates, we can suppose that $u = (u_1, \dots, u_{d_1}, 0 , \dots , 0)$ with $u_i > 0$ for $i\in \{1,\dots , d_1\}$. The ball $B_p(1)$ is strictly convex, and the (unique) supporting hyperplane $u+H$ of $B_p (1)$ at $u$ is given by
$$H:=\{v\in \R^d \,:\, v\cdot u^\star =0\}$$
where $u^\star=(u_1^{p-1},\ldots,u_{d_1}^{p-1},0,\ldots,0)$ satisfies $u\cdot u^\star = \sum_{i=1}^{d_1} u_i ^p = 1$ ($u^\star$ is the dual vector of $u$). Note that we have $H=H_1\oplus H_2$ where $H_1\subset \mbox{Vect}(e_1,\ldots,e_{d_1})$ and 
$H_2=\mbox{Vect}(e_{d_1+1},\ldots,e_{d})$, for $(e_1,\dots , e_d)$ the canonical orthonormal basis of $\mathbb{R}^d$. For $v\in H$, we write $v=v_1+v_2$ with $v_1 = (v_{1,1}, \dots , v_{1,d_1} , 0 , \dots , 0)\in H_1$ and $v_2 = (0 \dots , 0 , v_{2,d_1 + 1}, \dots , v_{2,d})\in H_2$.

\paragraph{Case $p=1$.} Consider $N= \| \cdot \|_1$ and let $u\in \mathbb{R}^d$ be such that $\| u \|_1 = 1$. We can suppose that $u=(u_1,\ldots,u_{d_1},0,\ldots,0)$ with 
$u_i>0$ for all $i\in \{1,\ldots , d_1\}$ with $d_1 = d_1 (u)$. In this case, a supporting hyperplane $u+H$ of $B_1(1)$ at $u$ is given by
$$H:=\{v\in \R^d \,:\, v\cdot u^\star =0\}$$
where $u^\star=(1,\ldots,1,0,\ldots,0)$ satisfies $u \cdot u^\star = \sum_{i=1}^{d_1} u_i = 1$. We emphasize the fact that this hyperplane $u+H$ is not the unique supporting hyperplane of $B_1 (1)$ at $u$ if $d_1 < d$, but this is the most natural choice of supporting hyperplane since it is the most symmetric. As for $p\in (1,\infty)$, note that we have $H=H_1\oplus H_2$ where $H_1\subset \mbox{Vect}(e_1,\ldots,e_{d_1})$ and 
$H_2=\mbox{Vect}(e_{d_1+1},\ldots,e_{d})$, for $(e_1,\dots , e_d)$ the canonical orthonormal basis of $\mathbb{R}^d$. For $v\in H$, we write $v_1 = (v_{1,1}, \dots , v_{1,d_1} , 0 , \dots , 0)\in H_1$ and $v_2 = (0 \dots , 0 , v_{2,d_1 + 1}, \dots , v_{2,d})\in H_2$.

\paragraph{Case $p=\infty$.} Consider $N= \| \cdot \|_\infty$ and let $u\in \mathbb{R}^d$ be such that $\| u \|_\infty = 1$. We can suppose that $u=(1,\ldots, 1, u_{d_3+1},\ldots,u_d)$ with 
$u_i\in [0,1)$ for all $i\in \{d_3 +1,\ldots , d\}$ with $d_3 : = d_3 (u)$. In this case, a supporting hyperplane $u+H$ of $B_\infty (1)$ at $u$ is given by
$$H:=\{v\in \R^d \,:\, v\cdot u^\star =0\}$$
where $u^\star=(1/d_3,\ldots,1/d_3,0,\ldots,0)$ (with $d_3$ non null coordinates) satisfies $u \cdot u^\star = \sum_{i=1}^{d_3} 1 / d_3 = 1$. We emphasize the fact that this hyperplane $u+H$ is not the unique supporting hyperplane of $B_\infty (1)$ at $u$ if $d_3>1$, but this is the most natural choice of supporting hyperplane since it is the most symmetric. Note that we have $H=H_3\oplus H_4$ where $H_3\subset \mbox{Vect}(e_1,\ldots,e_{d_3})$ and 
$H_4=\mbox{Vect}(e_{d_3+1},\ldots,e_{d})$, for $(e_1,\dots , e_d)$ the canonical orthonormal basis of $\mathbb{R}^d$. For $v\in H$, we write $v_3 = (v_{3,1}, \dots , v_{3,d_3} , 0 , \dots , 0)\in H_3$ and $v_4 = (0 \dots , 0 , v_{4,d_3 + 1}, \dots , v_{4,d})\in H_4$.

\subsection{Evaluation of $|K_\eta (u)|$}
\label{s:evalK}

The aim of this section is to prove the following estimate of $|K_\eta (u)|$.
\begin{prop}
\label{prop:K} 
Consider $N=\| \cdot \|_p$ for some $p\in  [1,\infty]$. Let $u\in \mathbb{R}^d$ be such that $\| u \|_p=1$ and let $u+H$ be the supporting hyperplane of $B_p(1)$ at $u$ as defined in Section \ref{s:H}. There exist two constants $A_1:=A_1(p,d,u)$, $A_2:=A_2(p,d,u)$, such that for any $\eta \in (0,1]$, we have
$$A_1 \eta^{  \gamma_p (u)}\le |K_\eta (u)|  \le A_2 \eta^{ \gamma_p (u)} $$
with
\begin{equation}
\label{e:defgamma}
 \gamma_p (u) := \left\{ \begin{array}{ll} \frac{d_1(u) -1}{2} + \frac{d_2(u)}{p} & \textrm{ if } p\in (1,\infty) \,, \\  d_2 (u) = d-d_1 (u) & \textrm{ if } p=1\,, \\  d_3 (u) - 1 = d - (d_4 (u) + 1) & \textrm{ if } p=\infty \,. \end{array} \right.
 \end{equation}
\end{prop}
The proof of Proposition \ref{prop:K} is based on the simple idea that $K_\eta (u)$, when properly rescaled with $\eta$, looks roughly like a Euclidean ball in $H$. However, the good rescaling depends on $p$, and has to be done in a specific way for $p=1$ and $p=\infty$.

\subsubsection{$p$-norm with $p\in(1,\infty)$}
\label{s:evalKp}

For a given vector $u\in \mathbb{R}^d$ such that $\|u\|_p=1$, we recall that the hyperplane $H$ we consider, and its description as $H=H_1\oplus H_2$, are given in Section \ref{s:H}. The good rescaling in $\eta$ in this case is of order $\eta^{1/2}$ in $H_1$ whereas it is of order $\eta^{1/p}$ in $H_2$. We make it clear in the following lemma.
\begin{lemma}\label{Lemme:normp}Let  $p\in (1,\infty)$, let $u\in \mathbb{R}^d$ such that $\|u\|_p=1$ and the corresponding hyperplane $H$ as defined in Section \ref{s:H}, and let
$$\hat{K}_\eta(u):=\{v\in H \,:\, \|u+\eta^{1/2}v_1+\eta^{1/p} v_2 \|_p^p \le 1+\eta\} \,.$$
Then, there exist $\eta_0>0$, $R_1>R_0>0$ (depending on $p$, $d$ and $u$) such that, for all $\eta\in (0,\eta_0)$,
$$B_H(R_0)\subset \hat{K}_\eta(u) \subset B_H(R_1)$$
where $B_H(r) := \{ v\in H \,:\, \| v\|_2 \leq r \}$.
\end{lemma}

\begin{proof}[Proof of Lemma \ref{Lemme:normp}]Let us find $R_0$ such that $B_H(R_0)\subset \hat{K}_\eta(u)$ for $\eta \in (0,1]$.
We have 
$$|1+t|^p=1+pt+ O(t^2) \,.$$
Thus, there exists $c>0$, such that, for all $t\in [-1,1]$, 
$$|1+t|^p\le 1+pt+ ct^2.$$
Let $R$ be small enough such that $R\le \min\{u_i, i\le d_1\}$ and $d R^p+ c R^2\sum_{i\le d_1} u_i^{p-2} \le 1$.
Let $v\in H$ be such that $\|v\|_{\infty}\le R$. Then
\begin{equation*}
\|u+\eta^{1/2}v_1+\eta^{1/p}v_2\|_p^p =\eta \|v_2\|_p^p+ \sum_{i\le d_1}|u_i+\eta^{1/2}v_{1,i}|^p \,.
\end{equation*}
If $\eta\le 1$, we have for all $1 \leq i \leq d_1$,
$$\left\| \frac{\eta^{1/2}v_{1,i}}{u_i}\right\| \le \frac{R}{u_i}\le 1\,,$$
hence we have
$$|u_i+\eta^{1/2}v_{1,i}|^p\le 
u_i^p+\eta^{1/2}p\, v_{1,i}\,u_i^{p-1}+\eta\, c\,v_{1,i}^2\,u_i^{p-2} \,.$$
Using that $v_1\cdot u^\star=0$ and $||u||_p=1$, we get
\begin{eqnarray*}
\|u+\eta^{1/2}v_1+\eta^{1/p}v_2\|_p^p  &\le& \eta d \|v_2\|_\infty^p+ 1+ \eta c\sum_{i\le d_1} v_{1,i}^2u_i^{p-2}\\
&\le &1+\eta \left(d R^p+ c R^2\sum_{i\le d_1} u_i^{p-2}\right)\\
&\le & 1+\eta
\end{eqnarray*}
which proves that $v\in \hat{K}_\eta(u)$. 
Using that $\|v\|_{2} \geq \|v\|_{\infty}$,
we get that $B_H(R)\subset \hat{K}_\eta(u)$.

Let us now find $\eta_0>0$ and $R_1>0$ such that $ \hat{K}_\eta(u)\subset B_H(R_1)$ for $\eta \in (0,\eta_0]$.
We have 
$$|1+t|^p=1+pt+\frac{p(p-1)}{2}t^2+ O(t^3) \,.$$
Thus, there exists $c>0$ such that, for all $t\in \R$, 
$$|1+t|^p\ge 1+pt+\frac{p(p-1)}{2}t^2- c|t|^3.$$
Let $R>0$ be such that 
$$\min \left( R^p, \frac{p(p-1)}{2}R^2\min_{i\le d_1} u_i^{p-2} \right) > 2$$
and let $\eta_0>0$ be such that $R^3\,c\,\eta_0^{1/2}\sum_{i\le d_1}u_i^{p-3}=1$.
Let $v\in H$ be such that $\|v\|_{\infty}= R$. 
We have, for all $1\leq i \leq d_1$,
$$|u_i+\eta^{1/2}v_{1,i}|^p\ge 
u_i^p+\eta^{1/2}\,p\,v_{1,i}\,u_i^{p-1}+\eta \,\frac{p(p-1)}{2}\,v_{1,i}^2 \,u_i^{p-2}-c\eta^{3/2}\,|v_{1,i}|^3\,u_i^{p-3} \,.$$
Using that $v_1\cdot u^\star=0$ and $\|u\|_p=1$, we get, for $\eta\in (0,\eta_0)$
\begin{eqnarray*}
\|u+\eta^{1/2}v_1+\eta^{1/p}v_2\|_p^p& \ge& \eta  \|v_2\|_p^p+ 1+ \eta\, \frac{p(p-1)}{2}\sum_{i\le d_1} v_{1,i}^2\, u_i^{p-2}-c\,\eta^{3/2}\sum_{i\le d_1}|v_{1,i}|^3\,u_i^{p-3}\\
&\ge &1+\eta \left( \|v_2\|_\infty^p+ \frac{p(p-1)}{2}\|v_1\|_\infty^2\min_{i\le d_1} u_i^{p-2}-\|v_1\|_\infty^3\,c\,\eta_0^{1/2}\sum_{i\le d_1}u_i^{p-3}\right)\\
&\ge &1+\eta \left( \|v_2\|_\infty^p+ \frac{p(p-1)}{2}\|v_1\|_\infty^2\min_{i\le d_1} u_i^{p-2}-R^3\,c\,\eta_0^{1/2}\sum_{i\le d_1}u_i^{p-3}\right)\\
&> & 1+\eta
\end{eqnarray*}
using that $R= \|v\|_\infty=\max( \|v_1\|_\infty, \|v_2\|_\infty)$.
Hence, for $\eta\in (0,\eta_0)$, we get that  $\hat{K}_\eta(u) \cap \{v\in H\,:\, \|v\|_\infty=R\}=\emptyset$. Since $\hat{K}_\eta$ is connected and contains $0$, we deduce that $\hat{K}_\eta(u) \subset  \{v\in H \,:\, \|v\|_\infty<R\}$ for  $\eta\in (0,\eta_0)$. Using that $\|v\|_{2}\le \sqrt{d}\, \|v\|_{\infty}$, it implies that  $ \hat{K}_\eta(u)\subset B_H(\sqrt{d} R)$.
\end{proof}

\begin{proof}[Proof of Proposition \ref{prop:K}, case $p\in (1,\infty)$] 
Recall that 
\begin{eqnarray*}
{K}_{\hat\eta}(u)&=&\{v\in H \,:\, \| u+v \|_p \le 1+\hat\eta\}\\
\hat{K}_\eta(u)&=&\{v\in H \,:\, \| u+\eta^{1/2}v_1+\eta^{1/p}v_2 \|_p^p \le 1+\eta\}\, ,
\end{eqnarray*}
thus
$$v=v_1+v_2\in \hat{K}_\eta(u) \Longleftrightarrow \eta^{1/2}v_1+\eta^{1/p}v_2\in K_{\hat{\eta}}(u) \mbox{ with } \hat{\eta}:=(1+\eta)^{1/p}-1 \,.$$
Hence 
$$|\hat{K}_{\eta}(u)|=|K_{\hat{\eta}}(u)|\, \eta^{-\gamma_p (u)}$$
with 
$$\gamma_p (u):=\frac{d_1(u)-1}{2}+\frac{d_2}{p}\,.$$
Lemma \ref{Lemme:normp} implies that there exists constants $A'_1 (p,d,u)$, $A'_2 (p,d,u)$ such that, for $\eta$ small enough, 
$$ A'_1 \leq | \hat K_\eta (u)| \leq A'_2 $$
thus 
$$ A'_1\eta^{\gamma_p (u)}  \leq |  K_{\hat \eta} (u)| \leq A'_2 \eta^{\gamma_p (u)} \,.$$
Since $\hat{\eta}\sim \eta/p$, we also get the existence of constants $A''_1 (p,d,u)$, $A''_2 (p,d,u)$ such that, for $\hat \eta \leq \hat \eta_0$ small enough,
$$ A''_1 \hat \eta^{\gamma_p (u)}  \leq |  K_{\hat \eta} (u)| \leq A''_2 \hat \eta^{\gamma_p (u)} \,.$$
Since $\eta \mapsto \eta^{\gamma_p (u)}$ is bounded away from $0$ and $\infty$ on $[\hat \eta_0 , 1]$, we obtain the existence of constants $A_1 (p,d,u)$, $A_2 (p,d,u)$ such that, for every $ \eta\in (0,1]$,
$$ A_1 \eta^{\gamma_p (u)}  \leq |  K_{\eta} (u)| \leq A''_2 \eta^{\gamma_p (u)} \,.$$
\end{proof}

\subsubsection{$1$-norm}

For a given vector $u\in \mathbb{R}^d$ such that $\|u\|_1=1$, we recall that the hyperplane $H$, and its description as $H=H_1\oplus H_2$, are given in Section \ref{s:H}. The good rescaling in $\eta$ is now of order $1$ in $H_1$ whereas it is of order $\eta$ in $H_2$. We make it clear in the following lemma.
\begin{lemma}\label{Lemme:norm1}Let $u\in \mathbb{R}^d$ such that $\|u\|_1=1$ and the corresponding hyperplane $H$ as defined in Section \ref{s:H}, and let
$$\hat{K}_\eta(u):=\{v\in H \,:\, \|u+v_1+\eta\, v_2 \|_1 \le 1+\eta\} \,.$$
Then, there exist $R_1>R_0>0$ (depending on $d$ and $u$) such that, for all $\eta\in (0,1]$,
$$B_H(R_0)\subset \hat{K}_\eta(u) \subset B_H(R_1)$$
where $B_H(r) := \{ v\in H \,:\, \| v\|_2 \leq r \}$.
\end{lemma}

\begin{proof}[Proof of Lemma \ref{Lemme:norm1}] As for the $p$-norm, it is sufficient to prove that for $\|v\|_{\infty}$ small enough, we have $v\in \hat{K}_\eta(u)$, and for $\|v\|_{\infty}=3$ for instance, we have $v\notin \hat{K}_\eta(u)$. We have
$$\| u+v_1+\eta \,v_2\|_1= \|u+v_1 \|_1+\|\eta v_2\|_1=\sum_{i\le d_1} |u_i+v_{1,i}|+ \eta \|v_2\|_1\,.$$
If $\|v\|_{\infty}\le \min(u_i,i\le d_1)\wedge (1/d)$, we have $|u_i+v_{1,i}|=u_i+v_{1,i}$ for all $1 \leq i \leq d_1$, thus using that $\|u \|_1=1$ and $v_1\in H$ thus $\sum_{i=1}^{d_1} v_{1,i} =0$, we get 
$$\|u+v_1+\eta v_2 \|_1= 1 + \eta \|v_2\|_1\le 1 + \eta \,,$$
thus $v\in \hat K_\eta (u)$.

Assume now that $\|v\|_{\infty} = 3$. Either $\|v_1\|_{\infty}= 3$: then,
since $\|u+v_1\|_1\ge \|v_1\|_{\infty} -\|u\|_{\infty}\ge 2$, we get 
$$\|u+v_1+\eta\, v_2\| _1\ge 2. $$
Or $\| v_2 \|_{\infty}= 3$: then using that $v_1\in H$ and so $\|u+v_1\|_1\ge 1$, we get
 $$\|u+v_1+\eta v_2\|_1\ge 1+3\eta \,.$$
 In both cases, $v\notin \hat K_\eta (u)$.
\end{proof}

\begin{proof}[Proof of Proposition \ref{prop:K}, case $p=1$] It is similar to the proof in the case $p\in (1,\infty)$. We now have 
$$|\hat{K}_{\eta}(u)|=|K_{\eta}(u)|\eta^{-d_2(u)} := |K_{\eta}(u)|\eta^{-\gamma_1 (u)}$$
with $\gamma_1(u) := d_2 (u)$, which yields similarly to the existence of constants $A_1 (1,d,u)$, $A_2 (1,d,u)$ such that, for every $ \eta\in (0,1]$,
$$ A_1 \eta^{\gamma_1 (u)}  \leq |  K_{\eta} (u)| \leq A_2 \eta^{\gamma_1 (u)} \,.$$
\end{proof}

\subsubsection{$\infty$-norm}

For a given vector $u\in \mathbb{R}^d$ such that $\|u\|_\infty=1$, we recall that the hyperplane $H$, and its description as $H=H_3\oplus H_4$, are given in Section \ref{s:H}. The good rescaling in $\eta$ is now of order $\eta$ in $H_3$ whereas it is of order $1$ in $H_4$. We make it clear in the following lemma.
\begin{lemma}\label{Lemme:norminfty}Let $u\in \mathbb{R}^d$ such that $\|u\|_\infty =1$ and the corresponding hyperplane $H$ as defined in Section \ref{s:H}, and let
$$\hat{K}_\eta(u):=\{v\in H \,:\, \|u+ \eta \,v_3+  v_4 \|_\infty \le 1+\eta\} \,.$$
Then, there exist $R_1>R_0>0$ (depending on $d$ and $u$) such that, for all $\eta\in (0,1]$,
$$B_H(R_0)\subset \hat{K}_\eta(u) \subset B_H(R_1)$$
where $B_H(r) := \{ v\in H \,:\, \| v\|_2 \leq r \}$.
\end{lemma}

\begin{proof}[Proof of Lemma \ref{Lemme:norminfty}]If $\|v\|_{\infty}\le \min(1-u_i, i>d_3) $ and $v\in H$, then for all $i> d_3$ we have 
$$ | u_i + v_{4,i} | \leq u_i + |v_{4,i}| \leq u_i + \|v\|_{\infty} \leq 1 \,,$$
and for all $i\leq d_3$ we have
$$ | u_i + \eta v_{3,i} | \leq u_i + \eta |v_{3,i}| \leq  1+  \eta \|v\|_{\infty} \leq 1+\eta \,,$$
thus
$$\|u+\eta \,v_3+v_4 \|_{\infty} \leq  1+\eta\, ,$$
and so $v\in \hat{K}_\eta(u)$.

If $\|v\|_{\infty}= 2d$, either $\|v_4\|_{\infty}= 2d$ and then 
$$\|u+\eta v_3+v_4\|_{\infty} \ge \|v_4\|_{\infty} -1>1+\eta\,,$$
or $\|v_3\|_{\infty}= 2d$ and then since $\sum_{i=1}^{d_3} v_{3,i}=0$, this implies that there exists some index $i_0\le d_3$ such that $v_{3,i_0}\ge 2$, and then  
$$\| u+\eta v_3+v_4 \|_{\infty}\ge u_{i_0}+\eta \, v_{3,i_0}> 1+\eta \,.$$
In both cases, we conclude that $v\notin \hat K_\eta (u)$.
\end{proof}

\begin{proof}[Proof of Proposition \ref{prop:K}, case $p=\infty$]
In this case, we have 
$$|\hat{K}_{\eta}(u)|=|K_{\eta}(u)|\eta^{-(d_3-1)} := |K_{\eta}(u)|\eta^{-\gamma_\infty (u)}$$
with $\gamma_\infty(u) := d_3 (u)-1$, which yields similarly to the existence of constants $A_1 (\infty,d,u)$, $A_2 (\infty,d,u)$ such that, for every $ \eta\in (0,1]$,
$$ A_1 \eta^{\gamma_\infty (u)}  \leq |  K_{\eta} (u)| \leq A_2 \eta^{\gamma_\infty (u)} \,.$$
\end{proof}


\subsection{Evaluation of $|M_\eta (u)|$}
\label{s:evalM}

We now compare $|M_\eta (u)|$ with $|K_\eta(u)|$.

\begin{prop} 
\label{prop:M}
Consider $N=\| \cdot \|_p$ for some $p\in  [1,\infty]$. Let $u\in \mathbb{R}^d$ be such that $\| u \|_p=1$ and let $u+H$ be the supporting hyperplane of $B_p(1)$ at $u$ as defined in Section \ref{s:H}. There exists $\Delta = \Delta (p,d,u) >0$ such that, for all $\eta \in (0,1]$,
$$ \Delta |K_\eta(u)| \leq  |M_\eta (u)| \leq  |K_\eta(u)| \,.$$
\end{prop}

\begin{proof} We write the proof for $p\in (1,\infty)$ but it can be easily adapted to the cases $p=1$ and $p=\infty$.
Let $u=(u_1,\ldots,u_{d_1},0,\ldots,0)$ with $u_i>0$ and $\|u \|_p=1$,
$$H:=\{v\in \R^d \,:\, v\cdot u^\star =0\}$$
where $u^\star=(u_1^{p-1},\ldots,u_{d_1}^{p-1},0,\ldots,0)$ as given in Section \ref{s:H}.
Recall that Lemma \ref{Lemme:normp} states that if
$$\hat{K}_\eta (u) :=\{v\in H, ||u+\eta^{1/2}v_1+\eta^{1/p}v_2||_p^p \le 1+\eta\} \,,$$
then there exist $\eta_0>0$, $R_1>R_0>0$ such that, for all $\eta\in (0,\eta_0)$,
$$B_H(R_0)\subset \hat{K}_\eta (u) \subset B_H(R_1)$$
with $B_H (r) :=\{ v\in H \,:\, \| v \|_2 \leq r \}$. Define $\hat{M}_\eta (u):=\hat{K}_\eta (u) \cap (-\hat{K}_\eta (u))$, we also have 
$$B_H(R_0)\subset \hat{M}_\eta (u) \subset B_H(R_1)$$
and so 
$$\Delta \leq \frac{|\hat{M}_{\eta} (u)| }{|\hat{K}_{\eta} (u)|} \leq 1$$
for some $\Delta (p,d,u) >0$.
Using that a change of variable yields that $|K_{\hat \eta} (u)|=\eta^{\gamma_p(u)} |\hat{K}_{\eta} (u)|$ and $|M_{\hat \eta} (u)|=\eta^{\gamma_p(u)} |\tilde{M}_{\eta} (u)|$ 
with $\hat{\eta}:=(1+\eta)^{1/p}-1$, we also get 
$$\Delta \leq \frac{|{M}_{\eta} (u)| }{|{K}_{\eta} (u)|} \leq 1 \,.$$
\end{proof}


\subsection{Proof of Theorem \ref{theo:tilde_mu_p}}
\label{s:appli}

This section is devoted to the proof of Theorem \ref{theo:tilde_mu_p}, and the control \eqref{e:longueur} on the length of geodesics. We also state the analog of Theorem \ref{theo:tilde_mu_p} for $\mu_{p,\varepsilon}$, together with estimates on the associated geodesics $\gamma_{p,\varepsilon}(su)$, and we compare the behavior of $ \mu_{p,\varepsilon}(u)$ and $\tilde{\mu}_{p,\varepsilon}(u)$.

\begin{theorem}\label{theo:mu_p} For all $p\in[1,\infty]$, for all $u\in \mathbb{R}^d$ such that $\|u\|_p=1$, there exist constants $C_i = C_i (p,d,u) >0$, $i=1,\dots ,6$ such that the following assertions hold.
\begin{itemize}
\item[(i)]  For $\varepsilon$ small enough (depending on $p,d,u$), we have 
$$C_1\varepsilon^{\kappa_p(u)} \le 1- \mu_{p,\varepsilon}(u) \le C_2\varepsilon^{\kappa_p(u)}.$$
\item[(ii)] For $\varepsilon$ small enough (depending on $p,d,u$), we have a.s. for large $s$, for any geodesic $\gamma_{p,\varepsilon}(su) \in \Pi (0,su)$ from $0$ to $su$ for the time $T_{p,\varepsilon} (0,su)$,
\begin{eqnarray*}
 C_3\varepsilon^{\kappa_p(u) }  s\le & \varepsilon^{ \frac{1}{d}}\sharp \gamma_{p,\varepsilon}(su) & \le C_4\varepsilon^{\kappa_p(u) }  s\\
&  \| \gamma_{p,\varepsilon}(su) \|_p -s&  \le C_5\varepsilon^{\kappa_p(u)} s.
  \end{eqnarray*}
\item[(iii)] Moreover, except if $p=1$ and $d_1 (u) = d$, or if $p=\infty$ and $d_3 (u) =1 $, for $\varepsilon$ small enough (depending on $p,d,u$), we have a.s. for large $s$, for any geodesic $\gamma_{p,\varepsilon}(su) \in \Pi (0,su)$ from $0$ to $su$ for the time $T_{p,\varepsilon} (0,su)$,
$$C_6\varepsilon^{\kappa_p(u)} s  \le  \| \gamma_{p,\varepsilon}(su) \|_p -s  \,.$$
\item[(iv)] Finally we have
$$\lim_{\varepsilon\to 0}\frac{ \mu_{p,\varepsilon}(u)-\tilde{\mu}_{p,\varepsilon}(u)}{\varepsilon^{\kappa_p (u)}}=0.$$
\end{itemize}
\end{theorem}


\begin{proof} This is an application of Theorems \ref{theo:mu}, \ref{theo:FPP} and Proposition \ref{prop:minN}. Fix $p\in [1,\infty]$, $u\in \mathbb{R}^d$ such that $\| u\|_p =1$, $H$ as given in Section \ref{s:H} (thus $(u,H)$ satisfies \eqref{e:(u,H)}) and recall that the definition of $\gamma_p (u)$ is given in \eqref{e:defgamma}. By Proposition \ref{prop:K}, we know that there exist $A_1(p,d,u),A_2 (p,d,u)$ such that for any $\eta \in (0,1]$,
\begin{equation}
\label{e:fin0}
A_1 \eta^{  \gamma_p (u)}\le |K_\eta (u)|  \le A_2 \eta^{ \gamma_p (u)} \,.
\end{equation}
Recall that $h_u (\eta) = \eta^{-d} |K_\eta (u)|$, thus for any $\eta \in (0,1]$,
$$ A_1 \eta^{  \gamma_p (u) - d}\le h_u (\eta) \le A_2 \eta^{ \gamma_p (u) - d} \,.$$
By definition, $g_u = h_u ^{-1}$, thus the previous inequalities imply that for any $x$ large enough (such that $g_u(x) \leq 1$), we have
$$ B_1 x^{\frac{1}{\gamma_p(u) - d}} \leq g_u(x) \leq B_2  x^{\frac{1}{\gamma_p(u) - d}} $$
with $B_1 (p,d,u) = A_2^{\frac{-1}{\gamma_p(u) - d}}$ and $B_2 (p,d,u) = A_1^{\frac{-1}{\gamma_p(u) - d}}$. Applying these inequalities to $x = C \varepsilon^{-1}$ for some constant $C = C(d,p)$ we obtain that for $\varepsilon$ small enough,
$$ B_1' \varepsilon^{\frac{1}{d-\gamma_p(u) }} \leq g_u( C \varepsilon^{-1} ) \leq B'_2  \varepsilon^{\frac{1}{d-\gamma_p(u)}} $$
for some constants $B'_i = B'_i (p,d,u)$. Notice that by definition on $\gamma_p (u)$ (see \eqref{e:defgamma}) and $\kappa_p(u)$ (see \eqref{e:ajout5}), we have
$$ \kappa_p (u) = \frac{1}{d- \gamma_p (u)} \,, $$
thus we have just proved that for $\varepsilon $ small enough,
\begin{equation}
\label{e:fin1}
B_1' \varepsilon^{\kappa_p (u)} \leq g_u( C \varepsilon^{-1} ) \leq B'_2  \varepsilon^{\kappa_p(u)}\,.
\end{equation}
Similarly, combining Proposition \ref{prop:K} with Proposition \ref{prop:M}, and recalling that $\bar h_u (\eta) = \eta^{-d} |M_\eta (u)|$ and $\bar g_u = \bar h_u^{-1}$, we obtain the existence of $B''_i = B''_i (p,d,u)$, $i=1,2$, such that, for every $\varepsilon$ small enough,
\begin{equation}
\label{e:fin2}
B_1'' \varepsilon^{\kappa_p (u)} \leq \bar g_u( C \varepsilon^{-1} ) \leq B''_2  \varepsilon^{\kappa_p(u)}\,.
\end{equation}
Finally, recall that in Proposition \ref{prop:minN} we defined  $\ell_u : [0,\infty) \rightarrow [0,\infty)$ as the inverse function of $\hat \ell_u : \eta \mapsto |K_\eta (u)|$ when $|K_0(u)| =0$. First, since $K_0 (u) = \{ v\in H \,:\, \|u+v\|_p = \|u\|_p )\}$, the case $| K_0 (u)| \neq 0$ corresponds to the case where $u$ belongs to a $(d-1)$-dimensional flat edge of $B_p(1)$. For $p\in (1,\infty)$, $B_p(1)$ is strictly convex thus $|K_0(u)| =0$ for every direction $u$. For $p=1$, $|K_0(u)| = 0$ if and only if $d_1(u)<d$. For $p=\infty$, $|K_0(u)| = 0$ if and only if $d_3(u)\geq 2$. We suppose from now on that we are in one of those cases, {\it i.e.}, that $|K_0(u)| = 0$. In other words, we restrict ourselves to the cases where $\gamma_p(u) > 0$. Equation \eqref{e:fin0} tells us that for any $\eta \in (0,1]$,
$$ A_1 \eta^{  \gamma_p (u)}\le  \hat \ell_u (\eta)  \le A_2 \eta^{ \gamma_p (u)} \,.$$
This implies the existence of constants $\mathcal{A}_i (p,d,u)$, $i=1,2$, such that for every $x$ small enough (so that $\ell_u (x) \leq 1$), we have
$$ \mathcal A_1 x^{ \frac{1}{ \gamma_p (u)}}\le   \ell_u (x)  \le \mathcal A_2 x^{ \frac{1}{\gamma_p (u)}} \,.$$
In particular, for constants $C,C'$ depending on $d$ and $u$, if $x= C' \varepsilon^{-1} \bar g_u (C \varepsilon^{-1})^d$, we have for $\varepsilon$ small enough,
$$ x \geq C' \varepsilon^{-1}  ( B_1'' \varepsilon^{\kappa_p (u)})^d = \mathcal A_3 \varepsilon^{-1 + d \kappa_p(u)}  = \mathcal A_3 \varepsilon^{\kappa_p(u) \gamma_p(u)} $$
for a constant $\mathcal A_3 = \mathcal A _3 (p,d,u)$. Since $\ell_u$ is increasing, we get for $\varepsilon $ small enough that
\begin{equation}
\label{e:fin3}
\ell_u (x) \geq \ell_u (\mathcal A_3 \varepsilon^{-1 + d \kappa_p(u)}) \geq  \mathcal A_1 (\mathcal A_3 \varepsilon^{\kappa_p(u) \gamma_p(u)}) ^{ \frac{1}{ \gamma_p (u)}} = \mathcal A_4 \varepsilon^{\kappa_p(u)}
\end{equation}
for a constant $\mathcal A_4 = \mathcal A_4 (p,d,u)$. Combining Theorems \ref{theo:mu}, \ref{theo:FPP} and Proposition \ref{prop:minN} with Equations \eqref{e:fin1}, \eqref{e:fin2} and \eqref{e:fin3} ends the proof of Theorems \ref{theo:mu_p} and \ref{theo:tilde_mu_p}, and Equation \eqref{e:longueur}.
\end{proof}

\begin{remark}
\label{r:longueur}
It is not a surprise that our proof does not provide a lower bound for the $N$-length of the geodesic in any cases. Consider for instance the case $d=2$ and $N=\| \cdot \|_1$: choose $u\in \mathbb{R}^d$ satisfying $\|u\|_1=1$ and $d_1 (u) =2$, for instance $u=(1/2,1/2)$. If we consider only oriented paths from $0$ to $s u$, going successively vertically to the north or horizontally to the east, those paths have a $\| \cdot \|_1$-length equal to $s$. The maximal number $M_s$ of rewards, {\it i.e.}, of points of the Poisson point process $\Xi$, that such an oriented path from $0$ to $su$ can collect has been introduced and studied by Hammersley \cite{Hammersley}, and is known to be of order $1$ - in fact, this problem is even solvable, it was proved by Logan and Shepp and by Vershik and Kerov in 1977, and in a more probabilistic way by Aldous and Diaconis \cite{AldousDiaconis} in 1995, that $\lim_{s\rightarrow \infty} M_s / s =1$. Thus, by restricting ourselves to oriented paths from $0$ to $su$ (whose $\| \cdot \|_1$-length is $s$), it is already possible to obtain a quantity of rewards of order $s$, and thus a gain in time of order $\varepsilon^{1/d} s = \varepsilon^{1/d_1(u)}s$. Since our approach does not allow us to get something better than the order of $1-\mu_{1,\varepsilon} (u)$ and $\sharp \gamma_{1,\varepsilon} (su) $, we have no hope it allows us to exclude that  $\| \gamma_{1,\varepsilon} (su) \|_1$ could be $s$.
\end{remark}


\subsection{Monotonicity}
\label{s:monotonie}

Theorem \ref{theo:monotonie} is now a direct consequence of the following proposition.
\begin{prop}
\label{prop:monotonie}
Suppose we are in one of the following cases:
\begin{itemize}
\item[(i)] $p\in [1,\infty]$ and $u = (1,0,\dots ,0)$ ;
\item[(ii)] $p\in \{1,2\}$, whatever $u\in \mathbb{R}^d$ such that $\|u\|_p =1$ ;
\item[(iii)] $p=\infty$ and $u\in \mathbb{R}^d$ such that $\|u\|_\infty=1$ and $d_3 (u) \in \{1,2\}$.
\end{itemize}
Then the function $\frac{1-\mu_{p,\epsilon}(u)}{\varepsilon^{\kappa_p (u)}}$ increases with respect to $\varepsilon$.
\end{prop}
Indeed, Proposition \ref{prop:monotonie} implies that
$$  \lim_{\varepsilon \rightarrow 0} \frac{1- \mu_{p,\epsilon}(u)}{\varepsilon^{\kappa_p (u)}} $$
exists by monotonicity, and Theorem \ref{theo:mu_p} $(iv)$ implies that
$$ \lim_{\varepsilon \rightarrow 0} \frac{1- \tilde \mu_{p,\epsilon}(u)}{\varepsilon^{\kappa_p (u)}} =  \lim_{\varepsilon \rightarrow 0} \frac{1- \mu_{p,\epsilon}(u)}{\varepsilon^{\kappa_p (u)}} \,.$$

We prove first (i) for $p\in (1,\infty)$ which also implies, by isotropy, that (ii) holds for the Euclidean norm. We then prove (ii) for the $1$-norm and in the last section, we establish (iii).

\begin{remark}
\label{r:monotonie}
In the realm of application of Proposition \ref{prop:monotonie}, we obtain the monotonicity of $\varepsilon \mapsto \frac{1-\mu_{p,\epsilon}(u)}{\varepsilon^{\kappa_p(u)}}$ as a consequence of a much stronger property. Indeed, by a rescaling argument, we express $ \frac{1-\mu_{p,\epsilon}(u)}{\varepsilon^{\kappa_p(u)}}$ as 
$$ \frac{1-\mu_{p,\epsilon}(u)}{\varepsilon^{\kappa_p(u)}} = \min_{s\rightarrow \infty} \,\inf_{\pi=(x_0\dots , x_{q+1}) \in\Pi (0, su)}\, \sum_{i=1}^q f_{p,d,u,s,\pi,i} (\varepsilon) $$
for given functions $ f_{p,d,u, s,\pi,i}$. What we actually prove is the monotonicity of each one of those functions $ f_{p,d,u,s,\pi,i}$. The monotonicity of all the functions $ f_{p,d,u, s,\pi,i}$ is not true for every $p\in [1,\infty]$ and every $u\in \mathbb{R}^d$ such that $\|u\|_p =1$. However, it doesn't imply that the monotonicity of $\varepsilon \mapsto \frac{1-\mu_{p,\epsilon}(u)}{\varepsilon^{\kappa_p(u)}}$ is not true, only that our approach cannot work.
\end{remark}

\subsubsection{ $p$-norm with $p\in (1,\infty)$}
Consider the $p$-norm with $p\in (1,\infty)$ and the direction $e_1=(1,0,\ldots,0)$. We want to prove that the fraction $(1-\mu_{p,\epsilon}(e_1))/\varepsilon^{\kappa_p(e_1)}$ increases with respect to $\varepsilon$. To simplify  notations, we write 
$$\kappa := \kappa_p(e_1) = \frac{1}{d - \frac{d-1}{p}} = \frac{p}{pd - d + 1}  \,.$$ 
Let $H= \mbox{Vect}(e_2,\ldots,e_{d})$ and for $x\in \R^d\setminus H$, we write $x=\lambda_x(e_1+t_{x})$ with $\lambda_x\in \R$, $t_{x}\in H$.
For $\pi=(0,x_1,\ldots,x_q,se_1)$ a path from $0$ to $se_1$, let us define $y_i:=x_{i+1}-x_i$ with the convention $x_0=0$, $x_{q+1}=se_1$. 
Then, a.s., $$T_{p,\varepsilon}(\pi)=\sum_{i=0}^q \left( ||y_i||_p-\varepsilon^{1/d}\right)+\varepsilon^{1/d}$$
and so 
$$\mu_\varepsilon(e_1)=\lim_{s\to \infty} \inf_{\pi\in \Pi(0, su)} \frac{\sum_{i=0}^q \left( ||y_i||_p-\varepsilon^{1/d}\right)}{s}=\lim_{s\to \infty} \inf_{\pi\in \Pi(0, su)} \frac{\sum_{i=0}^q \left( |\lambda_{y_i}|(1+||t_{y_i}||^p_p)^{1/p}-\varepsilon^{1/d}\right)}{s}.$$
Using that $s=\sum \lambda_{y_i}$, we get 
\begin{align*}
\frac{1-\mu_\epsilon(e_1)}{\varepsilon^\kappa} 
& = \lim_{s\to \infty} \sup_{\pi\in \Pi(0, su)} \sum_{i=0}^q \frac{\lambda_{y_i} -|\lambda_{y_i}|(1+||t_{y_i}||^p_p)^{1/p} + \epsilon^{1/d} }{s\varepsilon^\kappa}. \\
\end{align*}
Let $\psi:\R^d\mapsto \R^d$ be the linear map defined  for $x=\lambda(e_1+t)$ with $t\in H$  by
$$\psi(\lambda(e_1+t)):=\varepsilon^{-\kappa+\frac{1}{d}}\lambda(e_1+\varepsilon^{\frac \kappa p }t).$$
Recall that $\kappa=p/(pd-d+1)$ (we are here in the case $d_1=1$ and $d_2=d-1$). Using that $H$ is of dimension $d-1$, we get 
$\det(\psi)=\varepsilon^\alpha$ with 
$$\alpha=d \left(-\kappa+\frac 1 d \right)+(d-1)\frac{\kappa}{p}=\kappa\left(-d+\frac{d-1}{p}\right)+1 =0.$$
Hence, $\psi$ preserves the Lebesgue measure on $\mathbb{R}^d$ and so $\psi^{-1}(\Xi)$ has the same law as $\Xi$. So we also have
\begin{align*}
\frac{1-\mu_\epsilon(e_1)}{\varepsilon^\kappa} 
& = \lim_{s\to \infty} \sup_{\pi\in \Pi(0, su)} \sum_{i=0}^q \frac{\varepsilon^{-\kappa+\frac{1}{d}}\lambda_{y_i} -\varepsilon^{-\kappa+\frac{1}{d}}|\lambda_{y_i}|(1+||\varepsilon^{\frac \kappa p }t_{y_i}||^p_p)^{1/p} + \epsilon^{1/d} }{\varepsilon^{-\kappa+\frac{1}{d}}s\varepsilon^\kappa} \\
& = \lim_{s\to \infty} \sup_{\pi\in \Pi(0, su)} \sum_{i=0}^q \frac{|\lambda_{y_i}|(\mbox{sgn}(\lambda_{y_i}) -(1+\varepsilon^{ \kappa  }||t_{y_i}||^p_p)^{1/p}) + \epsilon^{\kappa} }{s\varepsilon^\kappa}. \\
\end{align*}
We then conclude the proof using the following elementary lemma.
\begin{lemma}\label{lemme:monotonie} For any $c\ge 0$, the functions $f:x\mapsto \frac{1-(1+cx)^{1/p}}{x}$ and  $g:=x\mapsto \frac{-1-(1+cx)^{1/p}}{x}$ are non-decreasing on $\R_+$
\end{lemma}

Indeed, assuming this lemma,  the function $(1-\mu_\epsilon(e_1))/\varepsilon^\kappa$ is increasing with respect to  $\varepsilon$ since each term of the previous sum increases with it.

\begin{proof}[Proof of Lemma \ref{lemme:monotonie}] By a change a variable, it is sufficient to prove the case $c=1$. Moreover since $g(x)=f(x)-\frac{2}{x}$, it is sufficient to prove that $f$ is non-decreasing.
We have
$$f'(x)=\frac{1}{x^2}\left((1+x)^{\frac{1}{p}-1}\left(-\frac{x}{p}+1+x \right)-1\right).$$
So 
$$f'(x)\ge 0 \Longleftrightarrow 1+\left(1-\frac{1}{p}\right)x\ge (1+x)^{1-\frac{1}{p}}.$$
The last inequality holds for all $x\ge 0$ by concavity of the function $(1+x)^{1-1/p}$.
\end{proof}
\subsubsection{ $1$-norm}
Consider the $1$-norm.
Let $u=(u_1,\ldots,u_{d_1},0,\ldots,0)$ with 
$u_i>0$ and $||u||_1=1$. Recall the definition of $H$, $H_1$ and $H_2$ given in Section \ref{s:H}. For $x\in \R^d\setminus H$, we write $x=\lambda_x(u+t_{1,x}+t_{2,x})$ with $\lambda_x\in \R$, $t_{1,x}\in H_1$ and $t_{2,x}\in H_2$.
With the same convention as in the previous paragraph, we have
\begin{align*}
\frac{1-\mu_\epsilon(u)}{\varepsilon^\kappa} 
& = \lim_{s\to \infty} \sup_{\pi \in \Pi(0,su)} \sum_{i=0}^q \frac{\lambda_{y_i} -|\lambda_{y_i}|\; ||u+t_{1,y_i}+t_{2,y_i}||_1 + \epsilon^{1/d} }{s\varepsilon^\kappa}. \\
\end{align*}
Recall that now we have  $\kappa:=\kappa_d(u)=1/d_1$.
Let $\psi:\R^d\mapsto \R^d$ be the linear map defined  for $x=\lambda(u+t_1+t_2)$ with $t_1\in H_1$ and $t_2\in H_2$ by
$$\psi(\lambda(u+t_1+t_2)):=\varepsilon^{-\kappa+\frac{1}{d}}\lambda(u+t_1+\varepsilon^{ \kappa  }t_2).$$
 Using that  $H_2$ is of dimension $d_2$, we get 
$\det(\psi)=\varepsilon^\alpha$ with 
$$\alpha=d \left(-\kappa+\frac 1 d \right)+d_2\kappa=1+ (d_2-d)\kappa=0.$$
As before, $\psi$ preserves the Lebesgue measure and so we get, using that $||u+t_1+t_2||_1=||u+t_1||_1+||t_2||_1$,
\begin{align*}
\frac{1-\mu_\epsilon(u)}{\varepsilon^\kappa} 
& = \lim_{s\to \infty} \sup_{\pi \in \Pi(0,su)} \sum_{i=0}^q \frac{|\lambda_{y_i}|(\mbox{sgn}(\lambda_{y_i}) -(||u+t_{1,y_i}||_1+\varepsilon^{ \kappa  }||t_{2,y_i}||_1)) + \epsilon^{\kappa} }{s\varepsilon^\kappa} \\
& = \lim_{s\to \infty} \sup_{\pi \in \Pi(0,su)} \sum_{i=0}^q \left(\frac{|\lambda_{y_i}|(\mbox{sgn}(\lambda_{y_i}) -||u+t_{1,y_i}||_1)}{s\varepsilon^\kappa}+\frac{|\lambda_{y_i}| \, ||t_{2,y_i}||_1 + 1 }{s}\right). \\
\end{align*}
Since, by definition of $H$, $||u+t_{1,y_i}||_1\ge 1$, the numerator $|\lambda_{y_i}|(\mbox{sgn}(\lambda_{y_i}) -||u+t_{1,y_i}||_1)$ is non positive and so each term of the previous sum  is again increasing with respect to $\varepsilon$.

\subsubsection{ $\infty$-norm}

Consider the $\infty$-norm. By symmetry, it is sufficient to prove the result for
$u=(1,u_{2},\ldots,u_{d})$ or $u=(1,1,u_{3},\ldots,u_{d})$ with
$u_i\in [0,1)$. Recall the definition of $H,H_3$ and $H_4$ given in Section \ref{s:H}. For $x\in \R^d\setminus H$, we write $x=\lambda_x(u+t_{3,x}+t_{4,x})$ with $\lambda_x\in \R$, $t_{3,x}\in H_3$ and $t_{4,x}\in H_4$.
With the same convention as in the previous paragraph, we have
\begin{align*}
\frac{1-\mu_\epsilon(u)}{\varepsilon^\kappa} 
& = \lim_{s\to \infty} \sup_{\pi \in \Pi(0,su)} \sum_{i=0}^q \frac{\lambda_{y_i} -|\lambda_{y_i}|\; ||u+t_{3,y_i}+t_{4,y_i}||_\infty + \epsilon^{1/d} }{s\varepsilon^\kappa}. \\
\end{align*}
Recall that now  $\kappa:=\kappa_d(u)=1/(d_4+1)$.
Let $\psi:\R^d\mapsto \R^d$ be the linear map defined  for $x=\lambda(u+t_3+t_4)$ with $t_3\in H_3$ and $t_4\in H_4$ by
$$\psi(\lambda(u+t_3+t_4)):=\varepsilon^{-\kappa+\frac{1}{d}}\lambda(u+\varepsilon^{ \kappa  }t_3+t_4).$$
As before, $\psi$ preserves the Lebesgue measure  since $d(-\kappa+\frac 1 d)+(d_3-1)\kappa=0$
and so we get
\begin{align*}
\frac{1-\mu_\epsilon(u)}{\varepsilon^\kappa} 
& = \lim_{s\to \infty} \sup_{\pi : 0 \to su} \sum_{i=0}^q \frac{|\lambda_{y_i}|(\mbox{sgn}(\lambda_{y_i}) -||u+\varepsilon^{ \kappa  }t_{3,y_i}+t_{4,y_i}||_\infty) + \epsilon^{\kappa} }{s\varepsilon^\kappa}. \\
\end{align*}

Let us  check that if $d_3\in \{1,2\}$, the function 
$f(x)=\frac{1-||u+xt_{3}+t_{4}||_\infty}{x}$
is non-decreasing on $\R^*_+$ for any $(t_3,t_4)\in H_3\times H_4$.

If $d_3=1$, then $H_3=\{0\}$. So using that  $1-||u+t_{4}||_\infty\le 0$, the function $f$ is indeed non-decreasing.

If $d_3=2$, then $H_3=\mbox{Vect}(e_2-e_1)$, so there exists some $c\in \R$ such that $t_3=c(e_2-e_1)$ and by symmetry, we can assume $c\ge 0$. Writing $u=e_1+e_2+u_4$ with $u_4\in H_4$,  we have, for $x\ge 0$,
  $$||u+xt_{3}+t_{4}||_\infty=\max(1+cx,||u_4+t_4||_\infty) \quad \mbox{i.e.} \quad f(x)=-\max\left(c,\frac{||u_4+t_{4}||_\infty-1}{x}\right)$$
and one can check that the function $x\mapsto\max(c,c'/x)$ is non-increasing on $\R^*_+$ for any $(c,c')\in \R_+\times \R$.


\section*{Appendix}

\subsection*{A : Proof of Theorem \ref{t:existence}}

\begin{proof}[Proof of Theorem \ref{t:existence} $(i)$: Existence of a geodesic.] To prove that a finite geodesic exists between two points, we need to show that paths with a large $N$-length cannot have a small travel time. Let us denote by $\Pi(0, \star)$ the set of polygonal paths starting from $0$. For any $s>0$ and $k>1$,  using a union bound in the first inequality, we have
\begin{eqnarray*}
\P(\exists\pi \in \Pi(0, \star), N(\pi)\ge ks, T_\varepsilon(\pi)\le s)
&\le & \sum_{i\ge 0} \P(\exists \pi\in \Pi(0, \star), \# \pi=i,  N(\pi)\ge ks, T_\varepsilon(\pi)\le s)  \\
&= & \sum_{i\ge 0} \P(\exists \pi\in \Pi(0, \star), \# \pi=i,  ks \le N(\pi)\le s+i\varepsilon^{1/d}) \\
&= & \sum_{i\ge (k-1)s\varepsilon^{-1/d}} \P(\exists \pi\in \Pi(0, \star), \# \pi=i,  N(\pi)\le s+i\varepsilon^{1/d}).
\end{eqnarray*}
Setting $\xi_0=0$, we have, for $\beta>0$, using that $1_{x\ge 0}\le e^{\beta x}$ and by Mecke Equation (see Theorem 4.4 in \cite{last_penrose_2017}),
\begin{eqnarray*}
 \P(\exists \pi\in \Pi(0, \star), \# \pi=i,  N(\pi)\le s+i\varepsilon^{1/d})
 &=&  \P(\exists (\xi_j)_{j\le i}\in \Xi,\sum_{j=1}^i N(\xi_{j}-\xi_{j-1})\le s+i\varepsilon^{1/d})\\
 &\le &  \int_{(\R^d)^i} \exp(\beta(s+i\varepsilon^{1/d} -  \sum_{j=1}^i N(\xi_{j}-\xi_{j-1}))d\xi_1 \ldots d\xi_i\\
  &\le & \exp(\beta s) \left(\int_{\R^d} \exp(\beta(\varepsilon^{1/d} -  N(z))dz\right)^i\\
 &\le & \exp(\beta s) \left(\frac{\exp(\beta \varepsilon^{1/d})}{\beta^d}\int_{\R^d} \exp( -  N(z))dz\right)^i.\\
\end{eqnarray*}
Let's take  $\beta$ such that $\frac{1}{\beta^d}\int_{\R^d} \exp( -  N(z))dz<1/4
$ and 
$\varepsilon_0$ small enough such that $\exp(\beta \varepsilon_0^{1/d})<2$.
We get then, for $\varepsilon<\varepsilon_0$,
\begin{equation}\label{e:majorationlongueur}
 \P(\exists \pi\in \Pi(0, \star), \# \pi=i,  N(\pi)\le s+i\varepsilon^{1/d})\le \frac{\exp(\beta s)}{2^i}.
 \end{equation}
In particular, for all $s> 0$, for $\varepsilon<\varepsilon_0$,
\begin{equation}\label{e:geodesiquecourte}
\lim_{k\to \infty} \P(\exists \pi\in \Pi(0, \star), N(\pi)\ge ks, T_\varepsilon(\pi)\le s)=0.
\end{equation}
Fix $n\in \N$ and let us now prove that there exists a.s. a geodesic from $x$ to $y$ for any $x,y\in B_N(n)$. Using \eqref{e:geodesiquecourte}, we get that there exists a.s. a (random) $K$ such that for any polygonal path $\pi$ starting from $0$ with $N(\pi)\ge Kn$, we have $T_\varepsilon(\pi)\ge 4n$.
Let now $\bar \pi=(x,x_1,\ldots,x_k,y)$ be a polygonal path from $x$ to $y$ such that $N(\bar \pi)\ge (K+1)n$. Then $ \pi=(0,x_1,\ldots,x_k,y)$ is a polygonal path starting from $0$ and by triangle inequality $N(\pi)\ge N(\bar \pi)-N(x)\ge Kn$ and so  that $T_\varepsilon(\pi)\ge 4n$. Using that
 $T_\varepsilon(\pi)\le N(x)+T_\varepsilon(\bar \pi)$, we deduce that $T_\varepsilon(\bar \pi)\ge 3n$. Since $T_\varepsilon(x,y)\le N(y-x)\le 2n$, we get in particular that, for all $x,y\in B_N(n)$, 
\begin{eqnarray*}
\inf \{T_\varepsilon (\pi) \,:\, \pi \in \Pi(x,y)\} & = & \inf \{T_\varepsilon (\pi) \,:\, \pi \in \Pi(x,y), \pi \subset B_N(K')\}\\ & = & \min \{T_\varepsilon (\pi) \,:\, \pi \in \Pi(x,y), \pi \subset B_N(K')\}
\end{eqnarray*}
with $K':=(K+2)n$.
The last infimum is taken on a finite set of paths since there is a finite number of points of $\Xi$ in $B_N(K')$. This implies the almost sure existence of a finite geodesic between $x$ and $y$.
This holds for any $n\in \N$, so we deduce the a.s. existence of a geodesic for any $x,y\in \R^d$.

\end{proof}

\begin{proof}[Proof of Theorem \ref{t:existence} $(ii)$: Existence of the time constant.]
Let us define, for $x\in \R^d$
$$J_\varepsilon(x) = \inf \{ T_\varepsilon(\pi) \, : \, \pi \in\Pi (x,\star)\} \quad \mbox{ and } \quad X_\varepsilon(x,y)=T_\varepsilon(x,y)-J_\varepsilon(x)-J_\varepsilon(y) \,,$$
where $\Pi (x,\star)$ is the set of polygonal paths starting from $x$. Taking $\pi=(x)$ (the path reduced to the single point $x$), we note that $J_\varepsilon(x)\le 0$. Moreover, note that Equation \eqref{e:geodesiquecourte} implies that $J_\varepsilon(x)$ and $T_\varepsilon(x,y)$ are a.s. finite for $\varepsilon$ small enough so $X_\varepsilon(x,y)$ is well defined.  Moreover, since  
$J_\varepsilon(x)\le T_\varepsilon(x,y)$, we get that  $X_\varepsilon(x,y)$ is non-negative.
Let us prove that $X_\varepsilon(x,y)$ satisfies the triangle inequality, \emph{i.e.}, for $x,y,z\in \R^d$, $X_\varepsilon(x,z)\le X_\varepsilon(x,y)+X_\varepsilon(y,z)$.

Let $\gamma_\varepsilon(y,x)=(y,a_1,\ldots,a_{n-1},x)$ be a geodesic from $y$ to $x$ and $\gamma_\varepsilon(y,z)=(y,b_1,\ldots,b_{m-1},z)$ be a geodesic from $y$ to $z$. Set  $a_0=b_0=y$ and $a_n=x$, $b_m=z$. Let $i_0,j_0$ be two indices such that $a_{i_0}=b_{j_0}$ and such that $\{a_{i_0+1},\ldots,a_{n}\}$ and  $\{b_{j_0+1},\ldots,b_{m}\}$ are disjoint. Such indices exist since $a_0=b_0=y$. Let
  $\gamma_1:= (a_{i_0},\ldots,a_{n-1},x)$ and  $\gamma_2:= (b_{j_0},\ldots,b_{m-1},z)$  ($\gamma_1$ and $\gamma_2$ can be reduced to a point if $i_0=n$ or $j_0=m$). Let $\gamma'_1:= (y,a_1\ldots, a_{i_0})$ and  $\gamma'_2:= (y,b_1\ldots, b_{j_0})$. We have
$$T_\varepsilon(x,y)= T_\varepsilon(\gamma_1)+T_\varepsilon(\gamma'_1)+\varepsilon^{1/d}1_{\{a_{i_0}\in \Xi\}}$$
$$T_\varepsilon(y,z)=   T_\varepsilon(\gamma_2)+T_\varepsilon(\gamma'_2)+\varepsilon^{1/d}1_{\{b_{j_0}\in \Xi\}}$$
(the potential reward located at $a_{i_0}=b_{j_0}$ is taking into account both in $T_\varepsilon(\gamma_1)$ and in  $T_\varepsilon(\gamma'_1)$). Moreover, noticing that  $(x,a_{n-1},\ldots,a_{i_0}=b_{j_0},\ldots,b_{m-1},z)$ is a polygonal path from $x$ to $z$ with distinct vertices, we get
$$T_\varepsilon(x,z)\le T_\varepsilon(\gamma_1)+T_\varepsilon(\gamma_2)+\varepsilon^{1/d}1_{\{a_{i_0}\in \Xi\}}.$$
Besides, $T_\varepsilon(\gamma'_1)\ge J_\varepsilon(y)$ and $T_\varepsilon(\gamma'_2)\ge J_\varepsilon(y)$.
So 
$$T_\varepsilon(x,y)+T_\varepsilon(y,z)\ge 2J_\varepsilon(y)+T_\varepsilon(x,z).$$
This yields
$$X_\varepsilon(x,y)+X_\varepsilon(y,z)\ge  2J_\varepsilon(y)+T_\varepsilon(x,z)-J_\varepsilon(x)-2J_\varepsilon(y)-J_\varepsilon(z)=X_\varepsilon(x,z)$$
and so the random variables $(X_\varepsilon(x,y),x,y\in \R^d)$ satisfy the triangle 
inequality. In particular, for any $u\in \R^d$, if we set for $m\ge n\ge 0$, $X_{n,m}:=X_\varepsilon(n u,m u)$, the process $(X_{n,m}, 0\le n \le m)$ is subadditive:
$$X_{l,m}\le X_{l,n}+X_{n,m} \mbox{ for any } 0\le l\le n\le m.$$
To apply Kingman's subadditive ergodic theorem, we must also check that $\E(X_{0,n})$ is finite.
We have 
$$0\le X_{0,n}= T_\varepsilon(0,nu)-J_\varepsilon(nu)-J_\varepsilon(0).$$
Using that $T_\varepsilon(0,nu)\le nN(u)$, we get
$$\E(X_{0,n})\le nN(u) +2\E(|J_\varepsilon(0)|).$$
Recall that $J_\varepsilon(0)\le 0$ and we have for $s\ge 0$,
\begin{equation*}
P(J_\varepsilon(0)\le -s)\le  \sum_{i\ge 0} \P(\exists \pi\in \Pi( 0, \star), \# \pi=i, N(\pi)\le -s+\varepsilon^{1/d} i).
\end{equation*}
Note that the bound obtain in \eqref{e:majorationlongueur} holds in fact also if $s<0$, so we get, for $\varepsilon<\varepsilon_0$,
\begin{equation*}
P(|J_\varepsilon(0)|\ge s)
\le  \sum_{i\ge 0}\frac{\exp(-\beta s)}{2^i}=2\exp(-\beta s)
\end{equation*}
which proves the integrability of
$J_\varepsilon(0)$ for $\varepsilon$ small enough. Using the stationarity and the ergodicity of the process,
Kingman's subadditive ergodic theorem \cite{Kingman} implies the existence of a limit
$$\mu_\varepsilon(u): = \lim_{n\to \infty} \frac{X_{0,n}} {n} = \inf_{n\ge 0} \frac{\E (X_{0,n}) }{n}\mbox{ a.s. and in $L^1$.}$$
Note that we have $\mu_\varepsilon(u)\ge 0$ since $X_{0,n}\ge 0$. Moreover, since, for any $n\ge 0$, the random variables $J_\varepsilon(nu)$ have the same law and have finite expectation, we get that
$$\lim_{n\to \infty}\frac{J_\varepsilon(nu)+J_\varepsilon(0)}{n}=0 \mbox{ a.s. and in $L^1$.}$$
So we also get 
$$\mu_\varepsilon(u)= \lim_{n\to \infty} \frac{T_\varepsilon(0,nu)} {n} \mbox{ a.s. and in $L^1$.}$$
It just remains to prove that this limit holds in fact for $s$ going to infinity, $s\in \R$. We write, for $s>0$,
\begin{eqnarray*}
 \frac{T(0,\lfloor s \rfloor u)-N(su-\lfloor s \rfloor u)}{s} &\le &\frac{T(0,su)}{s} \le \frac{T(0,\lfloor s \rfloor u)+N(su-\lfloor s \rfloor u)}{s}\\
\frac{T(0,\lfloor s \rfloor u)}{\lfloor s \rfloor} \frac{\lfloor s \rfloor}{s}-\frac{(s-\lfloor s \rfloor )N(u)}{s} &\le &\frac{T(0,su)}{s} \le \frac{T(0,\lfloor s \rfloor u)}{\lfloor s \rfloor} \frac{\lfloor s \rfloor}{s}+\frac{(s-\lfloor s \rfloor )N(u)}{s}\\
\end{eqnarray*}
which yields
$$\mu_\varepsilon(u)= \lim_{s\to \infty} \frac{T_\varepsilon(0,su)} {s} \mbox{ a.s. and in $L^1$.}$$
It remains to prove that for $\varepsilon$ small enough $\mu_\varepsilon(\cdot)$ is a norm. Triangle inequality and homogeneity are  straightforward. For the separation, using \eqref{e:majorationlongueur}, we have, for $\varepsilon$ small enough and $u\in \R^d$ such that $N(u)=1$,
\begin{equation*}
 \P \left( T_\varepsilon(0,su)\le \frac{s}{2} \right) \le \sum_{i\ge s\varepsilon^{-1/d}/2}\P \left( \exists \pi\in \Pi(0, \star), \# \pi=i,  N(\pi)\le \frac{s}{2}+i\varepsilon^{1/d}  \right) \le \frac{2\exp(\beta s/2)}{2^{s\varepsilon^{-1/d}/2}}
 \end{equation*}
which tends to 0 as $s$ tends to infinity if $\varepsilon$ is small enough (uniformly in $u$). In particular, this implies that, for small enough $\varepsilon$, for all $u\in B_N(1)$, $\mu_\varepsilon(u)\ge 1/2$. Hence, we get that, for small enough $\varepsilon$, $\mu_\varepsilon(\cdot)$ is a norm and,  in fact, for all $u\in\R^d$, 
$$\frac{N(u)}{2}\le \mu_\varepsilon(u)\le N(u).$$

\end{proof}

\subsection*{B : Proofs of Section \ref{s:geo}}

\begin{proof}[Proof of Proposition \ref{prop:function_h}]

Let $(u,H)$ satisfying \eqref{e:(u,H)}, {\em i.e.}, let $u \in \mathbb{R}^d$ be such that $N(u)=1$ and let $u+H$ be a supporting hyperplane of $B_N (1)$ at $u$. Let $\eta \geq 0$. We recall the following definitions (see Figure \ref{Fig:boule} for an illustration):
$$K_\eta(u):=\{v\in H \,:\, N(u+v)\le 1+\eta\} \qquad \mbox{ and } \qquad M_\eta(u):=K_\eta(u)\cap (-K_\eta(u))$$
and
$$h_u(\eta):=\eta^{-d}|K_\eta(u)| \qquad \mbox{ and } \qquad \bar{h}_u(\eta):=\eta^{-d}|M_\eta(u)|.$$

Let us prove that $h_u$ is a decreasing homeomorphism from $(0,\infty)$ to $(0,\infty)$. The reader can check that the proof can be easily adapted to show that $\bar h_u$ is also a decreasing homeomorphism from $(0,\infty)$ to $(0,\infty)$.

First notice that $v\in H \mapsto N (u+v)$ is convex. Indeed, for all $v,v' \in H$, for all $\lambda \in [0,1]$, we have
\begin{align*}
N(u + \lambda v + (1-\lambda) v') & = N (\lambda (u+v) + (1-\lambda) (u+v'))\\ &  \leq N (\lambda (u+v)) + N ((1-\lambda) (u+v')) =  \lambda N (u+v) + (1-\lambda) N ( u+v') \,.
\end{align*}
By convexity, for all  $\lambda \in [0,1]$ and $\eta,\eta'>0$ we have
\[
\lambda K_\eta(u) + (1-\lambda) K_{\eta'}(u) \subset K_{\lambda\eta+(1-\lambda)\eta'}(u)
\]
so
\[
|K_{\lambda\eta+(1-\lambda)\eta'}(u)|^{\frac 1 {d-1}} \ge |\lambda K_\eta(u) + (1-\lambda) K_{\eta'}(u)|^{\frac 1 {d-1}}.
\]
Using Brunn-Minkowski's inequality, we have
\[
|\lambda K_\eta(u) + (1-\lambda) K_{\eta'}(u)|^{\frac 1 {d-1}}\ge \lambda |K_\eta(u)|^{\frac 1 {d-1}} + (1-\lambda) |K_{\eta'}(u)|^{\frac 1 {d-1}}\,,
\]
so
\begin{equation}
\label{e:ell_u}
\eta \mapsto |K_\eta(u)|^{\frac 1 {d-1}} 
\end{equation}
is concave and so, in particular, is continuous. This already proves that $h_u$ is continuous.
Let $v \in H$ and $\eta>0$. 
For all $\lambda \in [0, 1]$, by convexity, we have
$$ N (u+\lambda v) \leq \lambda N(u+v) + (1-\lambda) N(u) = \lambda N(u+v) + (1-\lambda)$$
thus
\[
N(u+v)-1 \leq \eta \implies N(u+\lambda v) -1 \leq \lambda \eta \,,
\]
and so
\[
\lambda K_\eta(u) \subset K_{\lambda \eta}(u)\,.
\]
This gives that for all $0<\eta_1 \le \eta_2$,
\[
 \left(\frac{\eta_1}{\eta_2}\right)^{d-1}|K_{\eta_2}(u)| \le |K_{\eta_1}(u)|
\]
so the function $r_u(\eta):=  \frac{|K_\eta(u)|}{\eta^{d-1}}$
is non-increasing. This implies that $h_u(\eta)=r_u(\eta)/\eta$ is decreasing. 
Moreover, by triangle inequality, we have, for any $u\in \R^d$ such that $N(u)=1$, 
$$\{v\in H \,:\, N(v)\le \eta\}\subset K_\eta(u)\subset \{v\in H \,:\, N(v)\le 2+\eta\}. $$ 
Using that $N$ is equivalent to the euclidean norm, we get, for some constants $c,c'$ depending only on $d$ and $N$,
$$\{v\in H,\|v\|_2\le c\eta\}\subset K_\eta(u)\subset \{v\in H, \|v\|_2\le c'(2+\eta)\}. $$ 
Since $H$ is $(d-1)$-dimensional, this gives that, for some $A:=A(d,N)>0$ and $B:=B(d,N)>0$,
$$A\eta^{d-1} \le |K_\eta(u)|\le B(2+\eta)^{d-1}\,, $$ 
and so 
$$A\eta^{-1} \le h_u(\eta)\le B\eta^{-d}(2+\eta)^{d-1}. $$
This implies in particular that  $h_u$ tends to $+\infty$ at 0 and to $0$ at infinity.

\end{proof}

\begin{proof}[Proof of Lemma \ref{l:I+}]
Let $\| \cdot \|_2$ be the Euclidean norm.  We will show the existence of a function $\overline{\psi} := \overline{\psi}_\eta : H \rightarrow H$ that is an {\em affine} transformation with positive determinant such that
\begin{equation}
\label{e:27}
N (u+ \overline{\psi} (t))  \leq 1+ \eta \textrm{ for all }t\in H\textrm{ such that } \|t\|_2 \leq 1
\end{equation}
and 
\begin{equation}
\label{e:28}
N (u+ \overline{\psi} (t)) \geq1+ \eta \textrm{ for all }t\in H\textrm{ such that } \|t\|_2 \geq d-1.
\end{equation}
Let us prove that such a function $\overline{\psi}$ exists. The set
$$ K_\eta(u) = \{ v \in H \,:\, N(u+v) \leq 1+ \eta \}  $$
is compact, convex, with non-empty interior. Thus, by John-Loewner Theorem (see for instance Theorem III in \cite{Fritz}), there exists a centered ellipsoïde $J$ of $H$ and a $c\in H$ such that
$$  c+J \subset K_\eta(u) \subset c + (d-1) J. $$

Let $B_{H} (r):=B_{\|\cdot\|_2,H}(r) $ be the ball of $H$ of radius $r$ for the Euclidean norm $\| \cdot \|_2$. Let $\psi : H\mapsto H$ be the linear function with positive determinant such that $\psi (B_{H}(1)) = J$ and $\overline{\psi} := c + \psi$. Then, we have
$$ \overline{\psi} (B_{H}(1)) = c+J \subset K_\eta \subset c + (d-1)J = \overline{\psi} (B_{H} (d-1)) $$
which shows that \eqref{e:27} and \eqref{e:28} hold. Moreover, we can find bounds on the determinant of $\psi$. Indeed, we have 
$$ \det (\psi) = \frac{| \psi (B_{H}(1))|}{|B_{H} (1)|}=\frac{| \psi (B_{H}(1))|}{\cV_{d-1}}  $$
where $\cV_{d-1}$ is the volume of the unit euclidean ball of $\R^{d-1}$. Using that
$$ |\psi (B_{H} (1))| \leq |K_\eta(u) | \leq | \psi (B_{H} (d-1)) | $$
we get
\begin{equation}
\label{e:ajout1}
 \cV_{d-1}\det (\psi) \leq |K_\eta(u) | \leq  (d-1)^{d-1}\cV_{d-1}\det (\psi)\,.
 \end{equation}
Recall now that
$$ I^+ (\eta) = \int_{\mathbb{R}^d \cap \{ x:x\cdot u^\star >0 \}} \exp ( - (N(x) - (1-\eta) x \cdot u^\star)) dx \,.$$
By the change of variable\footnote{
This change of variable can be written as follows. Let $f_1 = u^\star / \|u^\star\|_2$ and let $(f_2,\dots, f_d)$ be an orthonormal basis of $H$, thus $(f_1,\dots , f_d)$ is an orthonormal basis of $\mathbb{R}^d$. Let $x = \sum_{i=1}^d x_i f_i$ and $u=\sum_{i=1}^d u_i f_i$ be the decomposition of $x$ and $u$ in this basis. Consider now the basis $(u, f_2, \dots , f_d)$ and $x= \lambda u + \sum_{i=1}^d v_i f_i$ the decomposition of $x$ in this basis. Then $x_1 = \lambda u_1$ and for every $i\in \{2,\dots , d\}$ we have $x_i = x \cdot f_i = \lambda u_i + v_i$. Thus the change of basis is given by $\Psi : (\lambda, v_2 , \dots , v_d ) \mapsto (x_1, x_2, \dots , x_d) = (\lambda u_1, v_2 + \lambda u_2, \dots, v_d + \lambda u_d) $. The Jacobian of $\Psi$ is $J_\Psi (\lambda, v_2, \cdots, v_d) = \textrm{det} \Psi' (\lambda, v_2, \dots, v_d) = u_1 = u \cdot u^\star / \| u^\star\|_2  = 1/  \| u^\star\|_2$.
}
$x=\lambda u +v$ with $\lambda=x\cdot u^\star$ and $v\in H$, we have 
\begin{align*}
I^+ (\eta) & = \frac{1}{\|u^\star\|_2} \int_0^\infty \int_{H} \exp(-(N(\lambda u + v)-(1-\eta) \lambda ))dv \,d\lambda\\
&=\frac{1}{\|u^\star\|_2} \int_0^\infty \int_{H}  \lambda^{d-1} \exp \left(-\eta \lambda - \eta \lambda \frac{N(u+w)-1}{\eta} \right)dw \,d\lambda \,.
\end{align*}
Let $x=\lambda \eta$ and $w=\overline{\psi}(t)$. Then
$$ I^+ (\eta)  = \frac{1}{\|u^\star\|_2} \eta^{-d} \det (\psi) G (\eta) $$
with
$$ G(\eta) = \int_0^\infty \int_{H} x^{d-1} \exp \left(-x -x \frac{N(u+\overline{\psi}(t))-1}{\eta} \right) dt \, dx \,. $$
Let us prove that there exists two constants $c,c'>0$ (depending only on $d$ and $N$) such that $c \leq G(\eta) \leq c'$ for $\eta >0$. Let us define
$$ g_\eta (t) := \frac{N (u + \overline{\psi}(t)) - 1}{\eta} \,.$$
First recall that if $t\in B_{H} (1)$, then $g_\eta (t) \leq 1$. Thus 
$$ G(\eta) \geq \int_0^\infty \int_{B_{H} (1)} x^{d-1} \exp (-2x) dt \,dx =\cV_{d-1}\int_0^\infty x^{d-1} \exp (-2x)  dx=\frac{\cV_{d-1}d!}{2^d} := c\, . $$
On the other side, if $\|t\|_2\geq d$ then $g_\eta (t) \geq 1$. Let $t_0=\overline{\psi}^{-1} (0)$. Note that $g_\eta(t_0) = 0$ and so $\|t_0\|_2 \leq d$. Using the convexity of $g_\eta$, we get that for $\|t\|_2\geq d$,
$$ g_\eta (t) \geq \frac{\|t-t_0\|_2}{d+\|t_0\|_2} \geq \frac{\|t\|_2}{2d} - \frac{1}{2} \,.$$
Since $g_\eta$ is non-negative on $H$, this lower bound holds in fact on $H$ (notice that $(\| t\|_2 / (2d)) - (1/2) <0$ for $\|t\|_2 < d$). This yields
\begin{align*}
G(\eta) & \leq  \int_0^\infty \int_{H} x^{d-1} \exp\left(-x -x \left(\frac{\|t\|_2}{2d} - \frac{1}{2} \right) \right)dt \,dx  \\
& \leq  \int_0^\infty \int_{\mathbb{R}^{d-1}} x^{d-1} \exp \left(-\frac{x}{2} - \frac{\|x t\|_2}{2d}  \right)dt \,dx  \\
& \leq  \int_0^\infty     \exp(-\frac{x}{2} ) dx  \int_{\mathbb{R}^{d-1}}\exp\left(- \frac{\|v\|_2}{2d}  \right)dv :=c' \,<\, \infty\,.
\end{align*}
Hence we get, using \eqref{e:ajout1}, that for $\eta >0$
\begin{equation*}
 \frac{c}{\|u^\star\|_2} ((d-1)^{d-1}\cV_{d-1})^{-1} |K_\eta(u)| \le \eta^{d}I^+(\eta)\le \frac{c'}{\|u^\star\|_2} (\cV_{d-1})^{-1}|K_\eta(u)| \,.
\end{equation*}
We conclude the study of $I^+ (\eta)$ using Lemma \ref{lem:CDV} which bounds $\|u^\star\|_2$ uniformly in $u\in B_N(1)$.
We now study $I (\eta)$ defined by
$$ I (\eta) = \int_{\mathbb{R}^d } \exp ( - (N(x) - (1-\eta) x \cdot u^\star)) dx. $$
We have $I(\eta)= I^+ (\eta)+I^- (\eta)$ where
$$ I^- (\eta) = \int_{\mathbb{R}^d \cap \{ x:x\cdot u^\star <0 \}} \exp ( - (N(x) - (1-\eta) x \cdot u^\star)) dx. $$
Note that for $\eta \in [0,1]$, $0\le I^- (\eta) \le \int_{\mathbb{R}^d } \exp ( - N(x) ) dx<\infty$ so $I^-$ is bounded around $0$. On the contrary, since $I^+(\eta)$ is of the same order as $h_u(\eta):=K_\eta(u)\eta^{-d}$, Proposition \ref{prop:function_h} implies that $I^+$ tends to infinity at $0$. So we get that $I(\eta)\sim I^+(\eta)$ for $\eta$ going to $0$ and so for $\eta$ small enough, we have 
$$
c_1 h_u(\eta) \le I(\eta)\le 2c_2 h_u(\eta)
$$
as desired.
\end{proof}

\bibliographystyle{plain}

\end{document}